\documentclass[12pt]{amsart}
\usepackage{graphicx}
\usepackage{tikz}        
\usepackage{amssymb}
\usepackage{epstopdf}
\usepackage[mathscr]{eucal}
\usepackage{amsmath,amssymb,amscd}
\usepackage{color}
\usepackage{epsfig}
\usepackage{bbold}
\newtheorem{theorem}{Theorem}
\newtheorem{lemma}{Lemma}

\newtheorem{definition}{Definition}
\newtheorem{corollary}{Corollary}
\newtheorem{criterion}{Criterion}
\newtheorem{proposition}{Proposition}
\theoremstyle{definition}
\newtheorem{remark}{Remark}

\newtheorem{example}{Example}

\def \mb{\mathbb}

\def \mk{\mathfrak}

\def \R{\mb R}                 

\def \a{\alpha}         
\def \b{\beta}           
\def \D{\Delta}         
\def \th{\theta}       
\def \t{\widetilde}     

\newcommand {\csn} {\text{csn}}
\newcommand {\arccsn} {\text{arccsn}}
\newcommand {\sn} {\text{sn}}
\newcommand {\ctn} {\text{ctn}}

\def \S{\mb S}        
\def \H{\mb H}        
\def \M{\mb M}        
\def\I{\mathbb{I}}
\def\v{{\bf v}}
\def\u{{\bf u}}

\newcommand {\q} {\mathbf{q}}
\newcommand {\p} {\mathbf{p}}
\newcommand {\F} {\mathbf{F}}

\def \and{\mbox{and}}
\usepackage[top=4cm, bottom=4cm, left=3.5cm, right=3.5cm]{geometry}
\title{Central configurations of the curved $N$-body problem}
\begin{document}
\maketitle
\markboth{Florin Diacu, Cristina Stoica, and Shuqiang Zhu}{Central configurations of the curved $N$-body problem}
\author{\begin{center}
{ \bf Florin Diacu}$^{1,2}$, {\bf Cristina Stoica}$^3$, and {\bf Shuqiang Zhu$^2$}\\
\bigskip
{\footnotesize $^1$Pacific Institute for the Mathematical Sciences\\
and\\
$^2$Department of Mathematics and Statistics\\
University of Victoria\\
P.O.~Box 1700 STN CSC\\
Victoria, BC, Canada, V8W 2Y2\\
\bigskip
$^3$Department of Mathematics\\
Wilfrid Laurier University\\
75 University Avenue West\\
Waterloo, ON, Canada, N2L 3C5\\
\bigskip
diacu@uvic.ca, cstoica@wlu.ca, zhus@uvic.ca\\
}
\end{center}


\begin{abstract}
We consider the $N$-body problem of celestial mechanics in spaces of nonzero constant curvature. Using the concept of locked inertia tensor, we compute the moment of inertia for systems moving on spheres and hyperbolic spheres and show that we can recover the classical definition in the Euclidean case. After proving some criteria for the existence of relative equilibria, we find a natural way to define the concept of central configuration in curved spaces using the moment of inertia, and show that our definition is formally similar to the one that governs the classical problem. The existence criteria we develop for central configurations help us provide several examples and prove that, for any given point masses on spheres and hyperbolic spheres, central configurations always exist. We end our paper with results concerning the number of central configurations that lie on the same geodesic, thus extending the celebrated theorem of Moulton to hyperbolic spheres and pointing out that it has no straightforward generalization to spheres, where the count gets complicated even in the case $N=2$.
\end{abstract}



{
\tableofcontents

}


\section{Introduction}

The notion of central configuration for the $N$-body problem of celestial mechanics was introduced by Pierre-Simon Laplace in 1789 in connection with the discovery of Eulerian and Lagrangian orbits, \cite{Laplace}, \cite{Euler}, \cite{Lagrange}. But a first systematic study of this concept appeared only in 1900, when Otto Dziobek published a fundamental paper on central configurations, \cite{Dziobek}. Research in this direction has continued ever since, showing over the past decades that central configurations are essential for understanding the equations of motion that describe the $N$-body problem. Although breakthroughs are rare in this difficult area of mathematics, some recent progress has been made on the Wintner-Smale conjecture, which we will discuss later in detail.

\subsection{Motivation}

In 1772 Joseph Louis Lagrange found the equilateral solutions of the 3-body problem and rediscovered the collinear orbits, whose existence Leonhard Euler had proved a decade earlier.
These particular motions, called homographic because their configurations stay similar to themselves for all time, can be decomposed into homothetic solutions and relative equilibria. The former are dilations and/or contractions of the particle system without rotation, whereas the latter are rotations without dilations or contractions, such that the mutual distances remain constant during the motion. Starting from the homothetic Lagrangian orbits, Laplace noticed that it may be simpler to seek the geometrical configurations that remain similar to themselves, which we now call central configurations, instead of looking for the homographic solutions to the differential equations, \cite{Wintner}. From the mathematical point of view, central configurations do not involve the time variable and are described by the system 
$$
\nabla U(\q)=\lambda\nabla I(\q),
$$
where $\q$ gives the positions of the bodies, $U$ is the force function (the negative of the potential), $I$ is the moment of inertia as defined in \eqref{long} below, $\lambda$ is a constant, and $\nabla$ denotes the gradient. Every central configuration automatically provides classes of relative equilibrium, homothetic, and homographic orbits. Therefore the dynamical question of finding certain solutions of an ordinary differential equation is reduced to an algebraic system, a methodology often used in this field.
   
\subsection{Importance}

Research done since 1900 has shown that the concept of central configuration opens a path towards understanding the $N$-body problem. Not only that it provides a method for finding periodic solutions, which Henri Poincar\'e deemed as a key towards untangling systems of differential equations, but it appears in various other circumstances. For instance, it was shown that when three or more bodies tend to a simultaneous collision, or when they scatter to infinity, they do so tending asymptotically to a central configuration, \cite{Saari}, \cite{Saari2}.

However, finding central configurations is far from easy. Basic questions related to them are often difficult to answer. One such question is known as the Wintner-Smale conjecture, which became notorious after Stephen Smale placed it sixth on his list of open problems for the 21st century, \cite{Smale}. The problem asks whether, for given $N$ positive masses, the number of planar central configurations is finite or not. So far, the conjecture is solved only for $N=3,4,$ and $5$, see \cite{Moeckel} and \cite{Albouy}. In all these cases the answer is that the set of central configurations is finite. But it is possible that for more than five bodies this set is infinite. If so, it may be countable or contain a continuum, as it actually happens when some masses are negative or charges are embedded in the system, \cite{Alfaro}, \cite{Roberts}.

\subsection{Brief history}

We consider here the motion of $N$ point masses in spaces of constant Gaussian curvature $\kappa\ne 0$, namely spheres for $\kappa>0$ and hyperbolic spheres for $\kappa<0$. This problem stems from the work of J\'anos Bolyai and Nikolai Lobachevsky, done in the 1830s, who independently had the idea of generalizing celestial mechanics to hyperbolic space, being among the first to understand that the laws of physics are related to the geometry of the universe, \cite{Bolyai}, \cite{Lobachevsky}. The analytic form of the potential, given by the cotangent of the distance, was introduced in 1870 by Ernest Schering in hyperbolic space,  \cite{Schering1}, \cite{Schering2}, and in 1873 by Wilhelm Killing for spheres, \cite{Killing}. Heinrich Liebmann proved two properties that established this potential as the natural extension of the Newtonian model to spaces of constant curvature. At the turn of the 20th century, he showed that the cotangent potential in the case of the Kepler problem (the motion of one body about a fixed attractive centre) is a harmonic function in the 3-dimensional (but not in the 2-dimensional) case and that every bounded orbit is closed, \cite{Liebmann1}, \cite{Liebmann2}. The same properties are true in the classical problem, \cite{Bertrand}. Although attempts at other extensions of the Newtonian potential to spaces of constant curvature existed, they were short-lived. Robert Lipschitz, for instance, proposed such a model, but his solution to the Kepler problem involved elliptic integrals, so it could not be explicitly expressed, \cite{Lipschitz}.

More recently, work in this direction was pursued by the Russian school of celestial mechanics, especially for the equations describing the motion of two bodies, which unlike in the Euclidean case are not integrable, \cite{Kozlov}, \cite{Shchepetilov}. In the past few years the problem was intensely researched in the general case of $N$ bodies, using various forms of the equations of motion, both in extrinsic and intrinsic coordinates. The chosen topics orbited around finding new relative equilibria, as well as rotopulsators (the solutions that generalize the concept of homographic orbits) and establishing their properties, including various types of stability, \cite{Diacu01}, \cite{Diacu02}, \cite{Diacu03}, \cite{Diacu05}, \cite{Diacu07},  \cite{Diacu77}, \cite{Diacu78},  \cite{Diacu06}, \cite{Diacu08},  \cite{Diacu09}, \cite{Diacu10}, \cite{Diacu11},  \cite{Diacu12}, \cite{Diacu13}, \cite{Diacu-Popa}, \cite{Diacu14}, \cite{Garcia}, \cite{Martinez1}, \cite{Martinez2}, \cite{Mont},  \cite{Perez}, \cite{Tibboel1}, \cite{Tibboel2}, \cite{Tibboel3}, \cite{Zhu}.

One other reason for pursuing these topics is related to possible applications towards deciding whether space is elliptic, flat, or hyperbolic. This question was already asked by Lobachevsky and Gauss. The former used observations on the Earth's parallax, while the latter measured the angles of a triangle formed by three mountain peaks, apparently hoping to see whether their sum was below or above $\pi$ radians. But none of them succeeded to provide an answer since the observation and measurement errors were larger than the potential deviation of the physical space from zero curvature, \cite{Kragh}. Bernhard Riemann's advance in differential geometry was also motivated by this question in connection with the general relationship between physics and the geometry of the universe, \cite{Riemann}.
Many other attempts to solve this problem were made in the mean time, including the so-called boomerang experiment, which analyzed the cosmological background radiation, \cite{Ber}. All of them, however, failed to provide a definite answer on whether the physical space is flat or not.

A potential way to offer a solution to this problem would be to mathematically find stable orbits that exist in, say, flat space, but not in hyperbolic and elliptic space, and then seek them in the universe through astronomical observations. A successful attempt of this kind could determine the geometric nature of the physical space. In fact, a small step in this direction was already made by showing that the Lagrangian relative equilibria of the 3-body problem appear only in the Euclidean space for nonequal masses, while it is well known that such orbits exist in the solar system, such as the equilateral triangles formed by the Sun, Jupiter, and any of the Trojan asteroids, \cite{Diacu11}. But we don't know yet whether some quasiperiodic orbits of nonequal masses, for instance, come close to Lagrangian solutions in curved space, such that it could be hard to distinguish between the two. Proving that such quasiperiodic orbits don't exist would offer a strong argument that our universe is flat. At this point, however, we don't seem to have the analytical tools to address this problem.

\subsection{Our goal}

In this paper we extend the concept of central configuration to the $N$-body problem in spaces of constant Gaussian curvature. Our idea was to find a formal definition that resembles the classical one. To achieve this goal we had to formulate first the correct definition of the moment of inertia for 3-spheres and hyperbolic 3-spheres, such that it agrees with the standard definition known in the Euclidean space. This step proved more difficult than we expected, also because of a terminological mixup that had occurred in the past few decades in the literature pertaining to the Newtonian $N$-body problem. A main obstacle was that, in the 3-dimensional case, the definition of the moment of inertia we considered suitable for our purposes did not match the one in the Euclidean space when the curvature takes the value zero. But in the end we found a way out with the help of the concept of locked inertia tensor used in geometric mechanics and thus clarified the semantic confusion that had occurred in recent years. 

We also wanted to develop some criteria for the existence of central configurations and apply them towards finding new classes of such mathematical objects. The reward was higher than expected when we understood that any central configuration on a 3-sphere delivers two classes of relative equilibria, whereas any central configuration on hyperbolic 3-spheres provides three such classes. Unlike in the Euclidean case, however, central configurations do not lead to homothetic orbits, in general. The loss of this property is not only because spheres and hyperbolic spheres are not vector spaces, so the concept of similarity doesn't make much sense, but also for dynamical reasons. In Euclidean space, bodies released from a central configuration with zero initial velocities collide simultaneously. While this happens in some highly symmetric cases in curved space as well, it doesn't happen in general. For example, for fixed points on spheres, which are central configurations, the bodies don't move at all.

We also included in this first paper on central configurations of the curved $N$-body problem a complete proof that for any masses on spheres and hyperbolic spheres central configurations exist. Finally, we added some results on the number of geodesic central configurations, in the spirit of the classical theorem proved by Forest Ray Moulton in the classical case, \cite{Moulton}.

\subsection{Summary and organization}

We will further summarize our results in the context of how the rest of this paper is organized. In Section 2 we discuss the concept of moment of inertia in Euclidean space and point out that some confusion occurred during the past few decades in celestial mechanics on what really this means. Our discussion is necessary for two reasons; first, we need to clarify the concept in Euclidean space such that we can find a way to define it in spaces of constant curvature; second, once we understand how to define it for nonzero curvature,  we need to recover the definition given in Euclidean space when the curvature tends to zero. In Section 3 we perform this task by starting from the notion of locked inertia tensor, which is a generalization of the moment of inertia for any manifold. After that we find the correct definitions for the moment of inertia for 3-dimensional spheres and hyperbolic spheres and see that they agree with the definition known in Euclidean space.

In Section 4 we introduce the equations of motion of the curved $N$-body problem and their integrals of motion. We point out that the value of the curvature is irrelevant when dealing with qualitative results and that only its sign matters. Therefore we can consider the motion of the particle system only on the unit sphere and the unit hyperbolic sphere, an approach we use for the rest of the paper. Section 5 is devoted to relative equilibria, solutions of the equations of motion for which the particle system behaves like a rigid body. We show that there are five classes of relative equilibria on 3-dimensional spheres and hyperbolic spheres, as they naturally follow from the isometry groups of these spaces. We end this section with some examples of relative equilibria, which suggest that we can recover these solutions from configurations we take at some given time instant.  In Section 6 we then develop two criteria for the existence of relative equilibria, one for the sphere and the other for the hyperbolic sphere, and establish the relationship between relative equilibria and the locked inertia tensor through the concept of mechanics systems with symmetry, first introduced by Smale, \cite{Smale70-1}. This relationship
confirms our definition for the moment of inertia.

These results prepare us for what follows in Section 7, where we can finally define the concept of central configuration using the two previous criteria proved for the existence of relative equilibria. Unlike in Euclidean space, we can introduce the new class of special central configurations on the sphere (as opposed to what we
call ordinary central configurations), which stems from the fact that fixed-point solutions occur in this case. No such central configuration exist on the hyperbolic sphere. We also define the new notions of geodesic, $\S^2, \S^3, \H^2$, and $\H^3$ central configurations and find some of their properties. An important tool for classifying central configurations is that of equivalence classes, which we introduce in Section 8, where we also 
prove several results about them, including ways to
reduce their study to convenient settings.

In Section 9 we prove two criteria for the existence of central
configurations, one in $\S^3$ and the other in $\H^3$, and
compute the value of the constant $\lambda$ involved in 
the definition of central configurations. Section 10 is devoted to proving the existence of central configurations in $\S^3$ and $\H^3$ for any given point masses. For this purpose we look at central configurations seen as critical points of the potential in spaces of constant curvature. We also extend here the Wintner-Smale conjecture from the Euclidean space to $\S^3$ and $\H^3$. In Section 11 we prove a result about central configurations that is in a way an analogue of the centre of mass property known in the classical case. Section 12 presents many examples of central configurations and discusses the relative equilibria that correspond to them. The investigation of whether Moulton's theorem about the number of collinear central configurations in Euclidean space can be extended to spaces of nonzero constant curvature is the subject of Section 13. We show that the theorem is true in $\H^3$, but fails to generalize to $\S^3$, where even the case of two bodies leads to a complicated count. Finally, Section 14 draws some conclusions and maps some further directions of research.

\section{The moment of inertia in Euclidean space}

In this section we will discuss the notion of moment of inertia in Euclidean space, aiming to find a proper definition of this concept for an $N$-body system in spaces of constant curvature, a goal we will achieve in the next section. At this stage we do not need any equations of motion, since the moment of inertia does not depend on them. The reason for dealing with this issue here is related to the fact that we will use this concept later in the definition of  central configurations.

\subsection{The physical concept}

The moment of inertia first appeared under this name in one of Euler's works of 1765, \cite{Eu}. On page 166, he wrote in Latin: ``Momentum inertiae corporis respectu eujuspiam axis est summa omnium productorum, quae oriuntur, si singula corporis elementa per quadrata distantiarum suarum ab axe multiplicentur.'' The term apparently made it into dictionaries sometime between 1820 and 1830, \cite{dictionary}. In the spirit of Euler, we can define this concept as follows.

\begin{definition}\label{moment-of-inertia}
The moment of inertia is the sum of the products of the mass and the square of the perpendicular distance to the axis of rotation of each particle in a body rotating about an axis.
\end{definition} 

According to the above definition, given above for a rigid body, the moment of inertia $I$ for a system of $N$ point masses, $m_1,\dots, m_N$, relative to the $z$-axis in some $xyz$-coordinate system of the Euclidean space $\mathbb R^3$, must be of the form
\begin{equation}\label{short}
I=\sum_{i=1}^N m_i(x_i^2+y_i^2),  
\end{equation}
where the position of the body $m_i$ is given by the vector
${\bf q}_i=(x_i,y_i,z_i)$. The moment of inertia has the same expression \eqref{short} if we restrict the motion to the plane $\mathbb R^2$ and assume that the rotation takes place about the origin of some $xy$-coordinate system, where the position vector for the body $m_i$ is now ${\bf q}_i=(x_i,y_i)$. 

In celestial mechanics, as long as the motion is restricted to $\mathbb R^2$, the moment of inertia is taken as in \eqref{short} or, sometimes, as half this quantity.
We will soon clarify the reason for which some authors introduce the factor $\frac{1}{2}$, but it is more important for now to note that in celestial mechanics  the moment of inertia is taken in $\mathbb R^3$ as
\begin{equation}\label{long}
I=\sum_{i=1}^n m_i(x_i^2+y_i^2+z_i^2)
\end{equation} 
or as half this quantity. The usual physical interpretation of formula \eqref{long} given in the field is that the moment of inertia provides a crude measure for the distribution of the bodies in space, with $I=0$ at total collision and $I$ large if at least one body is far away from the others. 
So not only that there is no match between Definition \ref{moment-of-inertia} and formula \eqref{long}, but the
celestial mechanics literature never hints at any connection
between the moment of inertia thus defined and the rotation of the bodies about an axis.

We thought that we might find a reason for this mismatch in the original works where formula \eqref{long} appeared. The moment of inertia for the classical $N$-body problem has been historically known for its presence in the Lagrange-Jacobi equation,
\begin{equation*}
\ddot I=(2\alpha+4)U+4h,
\end{equation*}
where $I$ is defined as in \eqref{long}, $U$ is the force function (i.e.\ the negative of the potential energy),
\begin{equation*}
U\colon\mathbb R^{3N}\to(0,\infty),\ \ U({\bf q}_1,\dots, {\bf q}_N)=\sum_{1\le i<j\le N}\frac{m_im_j}{|{\bf q}_i-{\bf q}_j|^\alpha},
\end{equation*}
$h$ is the energy constant, and $\alpha>0$ is also a constant. The physical units are chosen such that the gravitation constant is 1. Since the right hand-side of the Lagrange-Jacobi formula has a factor of 2, some researchers in celestial mechanics prefer to introduce the factor $\frac{1}{2}$
in the definition of $I$, but this detail is irrelevant. So a good place to start our attempt at answering the above question was the first work that contained the Lagrange-Jacobi equation.

\subsection{Jacobi's approach}

In the winter semester of 1842-43 at the University of K\"onigsberg in East Prussia, Carl Gustav Jacobi gave a lecture series on the $N$-body problem, which was very well received, so he published it as a book entitled ``Vorlesungen \"uber Dynamik'' (Lectures on Dynamics) in 1848, \cite{Jacobi}. On page 22,  the Lagrange-Jacobi equation appears for the first time. To write this relation he used the quantity $\sum m_ir_i^2,$ where he took $r_i^2=x_i^2+y_i^2+z_i^2.$ He never attached a name to this sum, as he did for other important concepts, such as kinetic energy, which he called ``lebendige Kraft'' (living force). Between pages 22 and 24 he referred to  $\sum m_ir_i^2$ as ``Ausdruck'' (expression), ``Summe'' (sum), or ``Gr\"osse'' (quantity), but  never hinted that it has anything to do with the moment of inertia defined in physics. Recall that this concept had been defined in 1765 and was already in dictionaries around 1830, so Jacobi should have been aware of it by the time of his lectures.

In the first paragraph on page 24, he mentioned that, at the origin of the coordinate system, $\sum m_ir_i^2$ reaches its minimum value and, when $\sum m_ir_i^2$ is constant, the bodies can be thought of lying on the same sphere. So he  formulated there our current physical interpretation of the moment of inertia in celestial mechanics as a crude measure of the bodies' distribution in space. And this is all he wrote relative to $\sum m_ir_i^2$. It is thus reasonable to think that he made no connection between this expression and the the rotation of the bodies about a fixed axis.

\subsection{Wintner's terminology}

A century later,  Aurel Wintner published the first edition of his influential book on the analytical foundations of celestial mechanics, updated in a second edition that appeared in 1947, \cite{Wintner}. On page 234, the quantity $J=\sum m_i\xi_i^2$ was introduced (with $\xi_i$ having the same meaning as Jacobi's $r_i$ mentioned above), which finally bears a name; he called it the polar inertia momentum. In modern parlance, the {\it polar moment of inertia}, or the {\it polar moment of area}, is a quantity used to predict an object's resistance to torsion. Physicists warn, however, that the polar moment of inertia should not be confused with the moment of inertia, which characterizes an object's angular acceleration due to torque. So though related, the concepts of torque and torsion mean different things.

\subsection{More recent developments}

Since the publication of Wintner's book, researchers in celestial mechanics got apparently mixed up in terminology. Though the two physical concepts are identical in the classical $N$-body problem as long as $I$ is defined in the plane $\mathbb R^2$, in $\mathbb R^3$ we must distinguish between the polar moment of inertia, \eqref{long}, and the moment of inertia, \eqref{short}. This remark is important to us for reasons related to the definition we will give for central configurations in spaces of constant curvature and to the fact that we can recover the classical definition when the curvature tends to zero.

In spite of a misleading terminology, the polar moment of inertia was understood in terms of a rotation when considered in the context of relative equilibria (orbits that maintain constant mutual distances between the bodies all along the motion) defined by central configurations, as we will explain in a later section. But the central configurations leading to relative equilibria  must be planar, (see \cite{Wintner}, p.\ 287). As there are no spatial relative equilibria in $\mathbb R^3$, the mixup between concepts was harmless. In the next section, we will provide and justify the correct definition of the moment of inertia for spheres and hyperbolic spheres, and later find another way to back up our findings.

\section{The moment of inertia in spaces of constant curvature}\label{mi}

In this section we will obtain the expression of the moment of inertia on spheres and hyperbolic spheres using the language of geometric mechanics, \cite{Abraham}, \cite{Marsden}, \cite{Ratiu}.  We will obtain the same expression for the moment of inertia for different types of rotations,  as expected from the considerations of the previous section. 

\subsection{The locked inertia tensor}

To define the moment of inertia in spaces of constant curvature, we will apply the more general concept of locked inertia tensor, introduced in \cite{Marsden}. For this purpose, consider as a configuration space a manifold $M$ endowed with an inner product   $\ll \,,\,\gg_{TM}$ on its tangent bundle $TM.$
Let $G$ be a Lie group that acts on $M$. Denote by $\mathfrak{g}$ the Lie algebra of $G$.  Each $\boldsymbol{\xi} \in \mathfrak g$ generates  a vector field on $M$ as follows: Write the action of $g\in G$ on a column vector $\q\in M$ simply as $g\q$; the vector at $\q$, denoted by $\boldsymbol{\xi}_M(\q)$,  is obtained  by differentiating $g\q$ with respect to $g$ in the direction of $\boldsymbol{\xi}$ at $g=e$. Explicitly, 
\begin{equation*}
\boldsymbol{\xi}_M (\q) := \frac{d}{dt}\Bigg|_{t=0} \left(\exp(\boldsymbol{\xi}t) \q   \right).
\end{equation*} 
Notice that the integral curves of this vector field are in fact group orbits of the $G$ action on $M.$ Denote by $\mathfrak g^*$ the linear space dual to $\mathfrak g$. 
For each $ \q\in M$,  the locked inertia tensor  is  the linear map
\begin{equation}\label{I}
\mathbb I(\q) \colon\mathfrak g\to\mathfrak g^*,\ \ 
\left<\I(\q) \boldsymbol{\xi},\boldsymbol{\eta} \right>_\mathfrak{g}= \ll \boldsymbol{\xi}_M (\q) \,,\boldsymbol{\eta}_M  (\q)  \gg_{ TM},  
\end{equation} 
where $\boldsymbol{\xi}, \boldsymbol{\eta}\in\mathfrak g$, and  $\left<\cdot, \cdot\right>_\mathfrak{g}$ is the natural pairing between $\mathfrak g$ and $\mathfrak g^*$.  Recall that the natural pairing between 
a vector space $V$ and its dual $V^*$ is a real number,
$\left<f,v\right>_V:=f(v)$, for each $f\in V^*$ and $v\in V$,
where $f(v)$ is defined in a natural way, specific to each vector space $V$.\\

For our purpose we consider the manifold $M$ embedded in a higher dimensional inner product vector space, take $G$ to be a matrix Lie group acting on the vector space, and understand  the action of the matrix Lie  group $G$ on $M$ as the induced action. We denote by $\q$ both the (column) vectors in the embedding inner-product space and  their representation in the embedded space $M$. Then it is easy to see that the vector field generated by $\boldsymbol{\xi} \in \mathfrak g$ at $\q$ is simply $\boldsymbol{\xi}\q$, i.e., the product of the matrix $\boldsymbol{\xi}$ with the column vector $\q$. 

The manifolds we are interested are embedded in either the standard Euclidean space, $\mathbb{R}^{4}$, or the Minkowski space,  $\mathbb{R}^{3,1}$. We regard these two spaces as $\R^4$, each endowed with its own inner product. More precisely, for some two vectors $\q_1=(x_1,y_1,z_1,w_1)^T$ and $\q_2=(x_2,y_2,z_2,w_2)^T$ in $\R^4$ or $\R^{3,1}$, the  inner products are given by
    \[ \q_1\cdot \q_2 = x_1x_2+y_1y_2+z_1z_2 +\sigma w_1w_2, \]
 where $\sigma=1$ for $\R^4$  and $\sigma=-1$  for $\R^{3,1}$. Then the family of manifolds  are 
 \[\M_\kappa^3:=
             \{(x,y,z,w)^T\in\mathbb R^4\ \! |\ \!
             x^2+y^2+z^2+\sigma w^2=\kappa,\ \kappa \ne 0\},\]
             with $w>0$ for $\kappa<0$.
For $\kappa>0$, the manifolds  are 3-spheres, which we denote by $\S_\kappa^3$, whereas for $\kappa<0$, the manifolds  are  hyperbolic 3-spheres, which we  denote by $\H_\kappa^3$. Let
\begin{align*}
T_{\q}\mathbb{M}_\kappa^3=\{\v=(v_x, v_y, v_z, v_w)^T\,|\,  x v_x +   y v_y + z v_z +\sigma w v_w =0  \}
\end{align*} 
be the tangent space to $\M_\kappa^3$ at  $\q=(x,y,z, w)^T \in \mathbb{M}_\kappa^3$, 
and let $m>0$ be the mass of a point particle at $\q$ moving with velocity $\v=(v_x, v_y, v_z,v_w)^T\in T_\q\mathbb{M}_\kappa^3$. We  introduce yet another inner product,   
\begin{equation}\label{dotprod-3-bis}
\ll \v \,,\v\gg := m\, \v\cdot  \v= m(v_x^2+v_y^2+v_z^2+\sigma v_w^2).
\end{equation} 

The matrix Lie groups acting on $\M_\kappa^3$ are the orthogonal groups $SO(4)$, for $\kappa>0$, and $SO(3,1)$, for $\kappa<0$. 
 An element of $SO(4)$  is of the form $PAP^{-1}$, with  $P\in SO(4)$ and 
 $$
 A= \begin{bmatrix}
  \cos\alpha \theta & -\sin \alpha \theta&0 &0\\
  \sin \alpha \theta&\cos\alpha \theta&0&0\\
  0&0&\cos\beta \theta & -\sin \beta \theta\\
   0&0&\sin \beta \theta&\cos\beta \theta
  \end{bmatrix},
  $$
 where $\a$, $\b\in \R$. We call these rotations positive elliptic-elliptic if  $\a \ne0$ and $\b\ne 0$, and positive elliptic if only one of them is zero. The above description is a generalization to $\S_\kappa^3$ of Euler's principle axis theorem for 2-spheres. Note that the reference to a fixed axis is, from the geometric point of view, far from suggestive in $\R^4$. 
 
  An element of  $SO(3,1)$ is of the form
  $PBP^{-1}$ or $PCP^{-1}$, with $P\in SO(3,1)$, 
$$
 B= \begin{bmatrix}
  \cos\a \theta & -\sin \a \theta&0 &0\\
  \sin \a \theta&\cos\a \theta&0&0\\
  0&0&\cosh\b \theta & \sinh \b \theta\\
   0&0&\sinh \b \theta&\cosh\b \theta
  \end{bmatrix},\ 
 C= \begin{bmatrix}
   1&0&0&0\\
   0&1&-\eta \theta&\eta \theta\\
   0&\eta \theta&1-\eta \theta^2/2&\eta \theta^2\\
   0&\eta \theta&-\eta \theta^2&1+\eta \theta^2/2
   \end{bmatrix},
   $$
where $\a$, $\b,\ \eta \in \R$. We call these rotations negative elliptic for $\a\ne 0$ and $\b=0$, negative hyperbolic for $\a= 0$ and $\b\ne 0$, negative elliptic-hyperbolic for $\a\ne 0$ and $\b\ne0$, and parabolic for $\eta\ne0$. The above description is a generalization to $\H_\kappa^3$ of the Euler's principle axis theorem for hyperbolic 2-spheres.
 
If $\mathfrak{so}(4)$ and $\mathfrak{so}(3,1)$ are the Lie algebras corresponding to the Lie groups $SO(4)$ and $SO(3,1)$, respectively, we can easily check that 
 $$A=\exp ( \boldsymbol{\xi}_1 \theta),\ \ \! B=\exp ( \boldsymbol{\xi}_2\theta),\ \ \!  C=\exp ( \boldsymbol{\xi}_3\theta), $$
 where $\boldsymbol{\xi}_1\in \mk{so}(4)$, $\boldsymbol{\xi}_2$, $\boldsymbol{\xi}_3$ $\in \mk{so}(3,1)$, with 
 \[    \boldsymbol{\xi}_1 = \begin{bmatrix}0&-\alpha&0&0\\
  \alpha&0&0&0\\ 0&0&0&-\beta\\0&0&\beta&0
  \end{bmatrix},  \ \  \boldsymbol{\xi}_2 = \begin{bmatrix}0&-\alpha&0&0\\
     \alpha&0&0&0\\ 0&0&0&\beta\\0&0&\beta&0
     \end{bmatrix},   \ \ \boldsymbol{\xi}_3 =  \begin{bmatrix}0&0&0&0\\
        0&0&-\eta &\eta \\ 0&\eta &0&0\\0&\eta &0&0
        \end{bmatrix}.  \]
 For our purpose, we compute the locked inertia tensor associated to the Lie subalgebras generated by $\boldsymbol{\xi}_1$ in $\mathfrak{so}(4)$ and by  $\boldsymbol{\xi}_2$ in $\mathfrak{so}(3,1)$. The reason why we do not compute it for $\boldsymbol{\xi}_3$ will become clear soon.

\subsection{Locked inertia tensor  in $\S^3_\kappa$}
Consider two elements, $a\boldsymbol{\xi}_1$ and $b\boldsymbol{\xi}_1$,  in the 1-dimensional Lie sub-algebra  $\mathfrak{so}(4)_{\boldsymbol{\xi}_1}\simeq\R$.  Then, obviously,
we have 
\begin{equation*}
a \boldsymbol{\xi}_1\q = 
\begin{bmatrix}
 a(-\a y, \a x, -\b w, \b z)
\end{bmatrix}^T. 
\end{equation*}
The locked inertia tensor \eqref{I} has now the form  
$$
\mathbb I_{\boldsymbol{\xi}_1}:\mathbb{S}_\kappa^3 \to {\mathcal{L}}(\mathfrak{so}(4)_{\boldsymbol{\xi}_1}, \mathfrak{so}(4)_{\boldsymbol{\xi}_1}^*), \ \ \
\left<\mathbb{I}_{\boldsymbol{\xi}_1}(\q) a\boldsymbol{\xi}_1, b\boldsymbol{\xi}_1 \right>_{\mathfrak{so}(4)_{\boldsymbol{\xi}_1}} =abm\boldsymbol{\xi}_1\q\cdot \boldsymbol{\xi}_1 \q,
$$
and  the natural pairing between  $\mathfrak{so}(4)$ and $\mathfrak{so}(4)^*$ is
$$
\left<\boldsymbol{\xi},\boldsymbol{\eta}\right>_{\mathfrak{so}(4)_{\boldsymbol{\xi}_1}}:=\frac{1}{2}{\rm tr}(\boldsymbol{\xi}^T \boldsymbol{\eta}).
$$
Then 
\begin{equation*}
\left<\mathbb{I}_{\boldsymbol{\xi}_1}(\q) a\boldsymbol{\xi}_1, b\boldsymbol{\xi}_1 \right>_{\mathfrak{so}(4)_{\boldsymbol{\xi}_1}} = ab\mathbb{I}_{\boldsymbol{\xi}_1}(\q)(\a^2+\b^2) , 
\end{equation*}
and 
\begin{equation*} 
  \begin{split}
  abm\boldsymbol{\xi}_1\q\cdot \boldsymbol{\xi}_1 \q &=ab m \left(\a^2(x^2+y^2) +\b^2(z^2+w^2)\right)\\
  &= ab m \left(\a^2(x^2+y^2) +\b^2\kappa ^{-1}- \b^2(x^2+y^2)  \right)\\
  &=ab m(\a^2-\b^2)(x^2+y^2) +ab m \b^2\kappa^{-1},
  \end{split}
  \end{equation*}
therefore the locked  inertia tensor associated with positive elliptic-elliptic rotations on $\mathbb{S}_\kappa^3$ is given by
\begin{equation}\label{I1}
\mathbb{I}_{\boldsymbol{\xi}_1}(\q)= \frac{m(\a^2-\b^2)(x^2+y^2)}{\a^2+\b^2 } +\frac{m \b^2\kappa^{-1}}{\a^2+\b^2 } .
\end{equation}
Notice that by letting $\b=0$, we get the locked inertia tensor associated with positive elliptic rotations,
\begin{equation*}
m(x^2+y^2),
\end{equation*}
which differs from the above only in a multiplicative coefficient and and an additive constant.

\subsection{Locked inertia tensor  in $\H^3_\kappa$}

Consider two elements, $a\boldsymbol{\xi}_2$ and $b\boldsymbol{\xi}_2$,  in the 1-dimensional Lie subalgebra  $\mathfrak{so}(3,1)_{\boldsymbol{\xi}_2}\simeq\R$.  Then obviously, 
\begin{equation*}
a \boldsymbol{\xi}_2\q = 
\begin{bmatrix}
 a(-\a y, \a x, \b w, \b z)
\end{bmatrix}^T. 
\end{equation*}
The locked inertia tensor \eqref{I} has now the form  
$$
\mathbb I_{\boldsymbol{\xi}_2}:\mathbb{S}_\kappa^3 \to {\mathcal{L}}(\mathfrak{so}(3,1)_{\boldsymbol{\xi}_2}, \mathfrak{so}(3,1)_{\boldsymbol{\xi}_2}^*), \ \ \
\left<\mathbb{I}_{\boldsymbol{\xi}_2}(\q) a\boldsymbol{\xi}_2, b\boldsymbol{\xi}_2 \right>_{\mathfrak{so}(3,1)_{\boldsymbol{\xi}_2}} =abm\boldsymbol{\xi}_2\q\cdot \boldsymbol{\xi}_2 \q,
$$
and  the natural pairing between  $\mathfrak{so}(3,1)$ and $\mathfrak{so}(3,1)^*$ is
$$
\left<\boldsymbol{\xi},\boldsymbol{\eta}\right>_{\mathfrak{so}(3,1)_{\boldsymbol{\xi}_2}}:=\frac{1}{2}{\rm tr}(\boldsymbol{\xi}^T \boldsymbol{\eta}).
$$
Then 
\begin{equation*}
\left<\mathbb{I}_{\boldsymbol{\xi}_2}(\q) a\boldsymbol{\xi}_2, b\boldsymbol{\xi}_2 \right>_{\mathfrak{so}(3,1)_{\boldsymbol{\xi}_2}} = ab\mathbb{I}_{\boldsymbol{\xi}_2}(\q)(\a^2+\b^2), 
\end{equation*}
and 
\begin{equation*} 
  \begin{split}
  abm\boldsymbol{\xi}_2\q\cdot \boldsymbol{\xi}_2 \q &=ab m \left(\a^2(x^2+y^2) +\b^2(-z^2+w^2)\right)\\
  &= ab m \left(\a^2(x^2+y^2) -\b^2\kappa ^{-1}+ \b^2(x^2+y^2)  \right)\\
  &=ab m(\a^2+\b^2)(x^2+y^2) -ab m \b^2\kappa^{-1},
  \end{split}
  \end{equation*}
therefore the locked  inertia tensor associated with negative elliptic-hyperbolic rotations in $\mathbb{H}_\kappa^3$ is given by
\begin{equation}\label{I2}
\mathbb{I}_{\boldsymbol{\xi}_2}(\q)= m(x^2+y^2) -\frac{m \b^2\kappa^{-1}}{\a^2+\b^2}.
\end{equation}
Notice that by letting $\b=0$, we get the locked inertia tensor associated with negative elliptic rotations,
\begin{equation*}
m(x^2+y^2),
\end{equation*}
and by letting $\a=0$, we get the locked inertia tensor associated with negative hyperbolic rotations,
\begin{equation*}
m(x^2+y^2)-m\kappa^{-1}, 
\end{equation*}
which differ from the above one, as in the case of the sphere, only in a multiplicative coefficient and an additive constant.

\subsection{Definition of the moment of inertia}

We can now end this section with the following natural definition of the moment of inertia for the $N$-body problem in spaces of constant Gaussian curvature.

\begin{definition}
Consider $N$ point masses, $m_1,\dots, m_N$, which move in $\mathbb M_\kappa^3$ under a law defined by a potential function, and assume that their configuration is given by the vectors
$ \q_i=(x_i, y_i, z_i,w_i)^T\in \M_\kappa^3, \ i=\overline{1,N}$.
Then the moment of inertia of the particle system 
 is the function
\begin{equation}\label{MI}
I(\q):=\sum_{i=1}^Nm_i(x_i^2+y_i^2).
\end{equation}
\end{definition}

The moment of inertia and the locked inertia tensor thus differ from each other only in a multiplicative coefficient and an additive constant. We distinguish them from the polar moment of inertia, which is usually defined in celestial mechanics in the Euclidean case. Of course, it isn't necessary to define the moment of inertia in other directions than the ones used above since, according to Euler's fixed axis theorem, we have already covered all possibilities. We will return to this concept in Section \ref{central-config} in the context of central configurations, where we will see that formula \eqref{MI} is essential for our purposes.

\section{Equations of motion}
In this section we introduce the $N$-body problem in spaces of constant nonzero curvature, which we will refer to as the \emph{curved $N$-body problem}, in contrast to its analogue in Euclidean space, which we will call the \emph{Newtonian  $N$-body problem}. As in \cite{Diacu03}, we set the curved $N$-body problems in the unit 3-sphere and the unit hyperbolic 3-sphere as  Hamiltonian systems in the Euclidean space $\R^4$ and in the Minkowski space $\R^{3,1}$,
respectively, with holonomic constraints that restrict the motion of the bodies to these manifolds. 
   
Recall that $\R^4$ and $\R^{3,1}$ are endowed with different inner products: for two vectors, $\q_1=(x_1,y_1,z_1,w_1)^T$ and $\q_2=(x_2,y_2,z_2,w_2)^T$, they are given by
    \[ \q_1\cdot \q_2 = x_1x_2+y_1y_2+z_1z_2 +\sigma w_1w_2, \]
    where $\sigma=1$ for the Euclidean space and $\sigma=-1$  for the Minkowski space. Then the unite sphere $\S^3$ and the unit hyperbolic sphere $\H^3$ are
   \begin{equation*} 
    \begin{split}
   \S^3&:=
            \{(x,y,z,w)^T\in\mathbb R^4\ \! |\ \!
            x^2+y^2+z^2+w^2=1\}\,\ \  {\rm and }\\
  \H^3&:=
                      \{(x,y,z,w)^T\in\mathbb R^{4}\ \! |\ \!
                      x^2+y^2+z^2-w^2=-1, \ w>0\},
    \end{split}
    \end{equation*}    
 respectively. We can merge these two manifolds into 
   \[  \M^3:=
             \{(x,y,z,w)^T\in\mathbb R^{4}\ \! |\ \!
             x^2+y^2+z^2+\sigma w^2=\sigma, \ {\rm with}\ w>0\ {\rm for}\ \sigma=-1\}. \]

   Given the positive masses $m_1,\dots, m_N$, whose positions are described by  the
   configuration $\q=(\q_1,\dots,\q_N)\in(\M^3)^N$,
   $\q_i=(x_i,y_i,z_i,w_i)^T,\ i=\overline{1,N}$, we define the
   singularity set 
    \begin{equation}
      \D=\cup_{1\le i<j\le N}\{\q\in (\M^3)^N\ \! ; \ \! \q_i\cdot\q_j=\pm \sigma\}.
      \notag\end{equation} 
If $d_{ij}$ is the geodesic distance between the point masses $m_i$ and $m_j$, we define the force function $U$ ($-U$ being the potential function) on $(\M^3)^N\setminus\D$ as
   $$
   U(\q):=\sum_{1\le i<j\le N}m_im_j\ctn d_{ij},
   $$
   where $\ctn (x)$ stands for $\cot (x)$ in $\S^3$ and $\coth (x)$ in $\H^3$. We also introduce two more notations, which unify the trigonometric and hyperbolic  functions, 
   \[ \sn (x)= \sin(x) \ {\rm or}\ \sinh(x), \ 
   \ \csn (x)= \cos(x) \ {\rm or}\ \cosh(x). \] 
  Then the distance $d_{ij}$ is given by the expression
   $$
   d_{ij}:=\arccsn
   (\sigma \q_i\cdot\q_j),
   $$
   where $\arccsn (x)$ is the inverse function of $\csn(x)$. We define the kinetic energy as
   \[T(\p)= \sum_{1\le i\le N}m_i \dot \q_i \cdot \dot \q_i=\sum_{1\le i\le N}m_i^{-1}\p_i\cdot \p_i,   \] 
   where $\p_i:=m_i\dot\q^T_i\in \R ^{4*}$, a row vector,  is the momentum of this system, and  $\R ^{4*}$ is endowed with two  inner products  induced from the two inner products in the  linear space dual to $\R^4$, i.e.\ for $\p_i=(u_{xi}, u_{yi},u_{zi}, u_{wi})$, $i=1,2$, 
   \[ \p_1 \cdot \p_2=u_{x1}u_{x2}+u_{y1}u_{y2} +u_{z1}u_{z2}+\sigma u_{w1}u_{w2}.\]
   We also denoted the momentum of the particle system by
   $$
   \p=(\p_1,\dots,\p_N).
   $$
   
   Then the curved $N$-body problem is given by the Hamiltonian system on $T^*((\M^3)^N\setminus\D)$, with 
   \[  \ H(\q,\p):=T(\q,\p)-U(\q).   \]

  Let us derive the equations of motion for the Hamiltonian system on $\S^3$.  The Hamiltonian is
  \[ H= \sum_{1\le i\le N}m_i^{-1}\p_i\cdot \p_i- \sum_{1\le i<j\le N}m_im_j\cot d_{ij}.\]    
  Here $U$ is defined on $(\S^3)^N\setminus\D$, with the set of singularities  $\D=\D^-\cup \D^+$, where
  \[\D^-:= \cup_{1\le i<j\le N}\{\q\in(\S^3)^N: \ \! \q_i\cdot\q_j=-1\}, \]  \[\D^+:= \cup_{1\le i<j\le N}\{\q\in(\S^3)^N: \ \! \q_i\cdot\q_j=1\}.\] 
   Using constrained Lagrangian dynamics, we get the equations describing the motion of the bodies, 
   \begin{equation*}
   \begin{cases}
   \dot\q_i=m_i^{-1}\p^T_i\cr
   \dot\p^T_i=\nabla_{\q_i} U-m_i^{-1}(\p_i\cdot \p_i)\q_i=\nabla_{\q_i} U-m_i(\dot{\q}_i\cdot \dot{\q}_i)\q_i\cr
   \q_i\cdot\q_i=1,\ \ \p_i\q_i=0, \ \ i=\overline{1,N},
   \end{cases}
   \end{equation*}
  where  $\p_i \q_i$ stands for  the matrix multiplication of the $1\times 4$ matrix and the  $4\times 1$  matrix, and  $\nabla_{\q_i} U$  stands for  the gradient of  $U$ on the manifold $(\S^3)^N $. Note that  the gradient  can be interpreted as the attractive force on $\q_i$ produced by all the other particles and $-m_i^{-1}(\p_i\cdot \p_i)\q_i$ can be viewed as the constraint force keeping the particles on the sphere. Thus we  denote 
  $\nabla_{\q_i} U$  and  $\nabla_{\q_i} m_im_j \cot d_{ij}$ by $\F_i$ and  $\F_{ij}$, respectively, so we have
  \begin{equation*}
    \F_{ij}=\frac{-m_im_j}{\sin^2 d_{ij}}\nabla_{\q_i} d_{ij}=\frac{-m_im_j}{\sin^2 d_{ij}}\nabla_{\q_i} \cos ^{-1}\q_i\cdot \q_j = \frac{m_im_j}{\sin^3 d_{ij}}\nabla_{\q_i} \q_i\cdot \q_j.
  \end{equation*}   
The gradient of $\q_i\cdot \q_j$ on the manifold $(\S^3)^N$ can be computed as follows. We extend any function $f\colon(\S^3)^N \to \R $ to the ambient space with the help of a function $\bar{f}\colon(\R^4)^N\to\R$ with the property  $\bar{f}(\lambda \q)=\bar{f}(\q)$, for $\lambda >0$,   
$$
\bar{f}(\q)= f\left(\frac{\q_1}{\sqrt{\q_1\cdot \q_1}}, \cdots, \frac{\q_N}{\sqrt{\q_N\cdot \q_N}}\right), 
$$ 
which is a homogeneous function of degree zero. Let $\t{\nabla}$ be the gradient in the ambient space and $\frac{\partial}{\partial n_i}$ the unit  normal vector of the $i$-th unit sphere.  Since $\frac{\partial \bar{f}}{\partial r_i}=0$, we obtain  $(\t{\nabla}_{\q_i}\bar{f})|_{(\S^3)^N}=\nabla_{\q_i} f+ \frac{\partial \bar{f}}{\partial r_i}\frac{\partial}{\partial n_i}=\nabla_{\q_i} f$.   Thus 
     \begin{equation*} \begin{split} 
      \F_{ij}&= \frac{m_im_j}{\sin^3 d_{ij}}\t{\nabla}_{\q_i} \frac{\q_i\cdot \q_j}{\sqrt{\q_i\cdot \q_i}\sqrt{\q_j\cdot \q_j}}
      =\frac{m_im_j}{\sin^3 d_{ij}} \frac{\sqrt{\q_i\cdot \q_i}\sqrt{\q_j\cdot \q_j} \q_j- \q_i\cdot \q_j \frac{\sqrt{\q_j\cdot \q_j}}{\sqrt{\q_i\cdot \q_i}} \q_i}{(\sqrt{\q_i\cdot \q_i}\sqrt{\q_j\cdot \q_j})^2}\\
      &= \frac{m_im_j [\q_j-\cos d_{ij}\q_i]}{\sin^3 d_{ij}}.
   \end{split} \end{equation*}  
Thus the equations of motion for the curved $N$-body problem on $\S^3$ are
  \begin{equation*}
     \begin{cases}
     \dot\q_i=m_i^{-1}\p^T_i\cr
     \dot\p^T_i=\sum_{j=1,j\ne i}^N\frac{m_im_j [\q_j-\cos d_{ij}\q_i]}{\sin^3 d_{ij}} -m_i(\dot{\q}_i\cdot \dot{\q}_i)\q_i\cr
     \q_i\cdot\q_i=1,\ \ \p_i\q_i=0, \ \ i=\overline{1,N}.
     \end{cases}
     \end{equation*}
 
     
 \noindent{\bf {Gravitation law} in $\S^3$}.
 A mass $m_2$ at $\q_2 \in \S^3$ attracts another mass $m_1$ at $\q_1\in \S^3$ ($\q_1\ne \pm \q_2$) along the minimal geodesic connecting the two points with a force whose magnitude is $\frac{m_1m_2}{\sin^2 d_{12}}$. More precisely, 
 $$\F_{12}= \frac{m_1m_2 [\q_2-\cos d_{12}\q_1]}{\sin^3 d_{12}}.$$

  Similarly, we can derive the  equations of motion for the Hamiltonian system  on $\H^3$. The Hamiltonian is  
  \[ H=T(\q,\p)-U(\q)= \sum_{1\le i\le N}m_i^{-1}\p_i\cdot \p_i- \sum_{1\le i<j\le N}m_im_j\coth d_{ij}.\]    
    Here $U$ is defined on $(\H^3)^N\setminus\D$, and the set of singularities is 
      \[\D:= \cup_{1\le i<j\le N}\{\q\in(\H^3)^N: \ \! \q_i\cdot\q_j=1\}.\] 
   We interpret  $\nabla_{\q_i} U$  and  $\nabla_{\q_i} m_im_j \coth d_{ij}$ as $\F_i$ and  $\F_{ij}$ respectively. Similar computations lead to
    \begin{equation*}
      {\bf F}_{ij}= \frac{m_im_j [\q_j-\cosh d_{ij}\q_i]}{\sinh^3 d_{ij}},
      \end{equation*} 
    and  the equations of motion the curved $N$-body problem on $\H^3$ are
    \begin{equation*}
       \begin{cases}
       \dot\q_i=m_i^{-1}\p^T_i\cr
       \dot\p^T_i=\sum_{j=1,j\ne i}^N\frac{m_im_j [\q_j-\cosh d_{ij}\q_i]}{\sinh^3 d_{ij}} +m_i(\dot{\q}_i\cdot \dot{\q}_i)\q_i\cr
       \q_i\cdot\q_i=-1,\ \ \p_i\q_i=0, \ \ i=\overline{1,N}.
       \end{cases}
       \end{equation*}
       

\noindent{\bf{Gravitation law in $\H^3$}}.
 A mass $m_2$ at $\q_2 \in \H^3$ attracts another mass $m_1$ at $\q_1\in \H^3$ ($\q_1\ne \q_2$)   along the minimal geodesic connecting the two points with  a force  whose magnitude is $\frac{m_1m_2}{\sinh^2 d_{12}}$. More precisely, 
 $$\F_{12}= \frac{m_1m_2 [\q_2-\cosh d_{12}\q_1]}{\sinh^3 d_{12}}.$$  

Using the functions $\sn (x)$ and $\csn(x)$ introduced earlier, we  can blend the two systems of equations into one system in $(\M^3)^N\setminus\D$,
  \begin{equation}\label{equation}
   \begin{cases}
   \dot\q_i=m_i^{-1}\p^T_i\cr
   \dot\p^T_i=\sum_{j=1,j\ne i}^N\frac{m_im_j [\q_j-\csn d_{ij}\q_i]}{\sn^3 d_{ij}}-\sigma m_i(\dot{\q}_i\cdot \dot{\q}_i)\q_i\cr
   {\bf q}_i\cdot{\bf q_i}=\sigma, \ \ \p_i\q_i=0, \ \ i=\overline{1,N}.
   \end{cases}
   \end{equation}

\begin{remark}  
Some researchers studied the curved $N$-body problem on spheres and hyperbolic spheres with curvature $\kappa\ne \pm1$ \cite{Kilin}, i.e., in 
\[
\S_\kappa^3=
\{(x,y,z,w)\in \R^4\ \! |\ \!
x^2+y^2+z^2+w^2=\kappa^{-1}\}\, \ \ \kappa >0,  \]
 \[
 \H_\kappa^3=
 \{(x,y,z,w)\in \R^4\ \! |\ \!
 x^2+y^2+z^2-w^2=\kappa^{-1}, w>0\}\, \ \ \kappa <0.  \]
This is not necessary since it has been shown in \cite{Diacu03} that there are coordinate and time-rescaling transformations, 
     \begin{equation*}
       \q_i = |\kappa|^{-1/2}{\bf r}_i, \ i=\overline{1,N}\ \  {\rm  and}\ \     \tau = |\kappa|^{3/4} t,
       \end{equation*}
which bring the systems from $\S_\kappa^3$ and $\H_\kappa^3$ to systems to $\S^3$ and $\H^3$, respectively.    
 \end{remark}

\subsection{Total angular momentum integrals}
The Hamiltonian function is invariant under the action of $SO(4)$ and $SO(3,1)$ for motions in $\S^3$ and  $\H^3,$ respectively. These symmetries lead to six integrals, which stand for the generalized version of the usual total angular momentum conservation laws in $\mathbb{R}^3,$
\begin{align*}
\omega_{xy}&=\sum_{i=1}^Nm_i(\dot x_iy_i-x_i\dot y_i), &
\omega_{xz}&=\sum_{i=1}^Nm_i(\dot x_iz_i-x_i\dot z_i),\displaybreak[0]\\
\omega_{xw}&=\sum_{i=1}^Nm_i(\dot x_iw_i-x_i\dot w_i), & 
\omega_{yz}&=\sum_{i=1}^Nm_i(\dot y_iz_i-y_i\dot z_i),\displaybreak[0]\\
\omega_{yw}&=\sum_{i=1}^Nm_i(\dot y_iw_i-y_i\dot w_i), & 
\omega_{zw}&=\sum_{i=1}^Nm_i(\dot z_iw_i-z_i\dot w_i),
\end{align*}
as shown in \cite{Diacu03} and \cite{Diacu05}.  We will refer to them as \emph{angular momentum integrals.}

\section{Relative equilibria}\label{relativeequilibria}

In this section we introduce the relative equilibria of the curved $N$-body problem and classify these solutions into several classes. Although this notion was considered and analyzed in our previous work (see \cite{Diacu03}, \cite{Diacu05}), we
stir our presentation towards showing some patterns not seen before, which will allow us to later define the concept of central configuration. Relative equilibria are orbits of the form of $Q(t) \q(0)$, where $Q(t)$ is a one-parameter subgroup of the isometry group of the system. 
We need first to take a closer look at the isometric transformations of $\M^3$, the Lie groups $SO(4)$ and $SO(3,1)$, such that we can find a suitable definition for these solutions.

A one-parameter subgroup of $SO(4)$  is of the form $PA_{\alpha, \beta}(t)P^{-1}$, with  $P\in SO(4)$ and 
$$
A_{\alpha, \beta}(t)= \begin{bmatrix}
 \cos\alpha t & -\sin \alpha t&0 &0\\
 \sin \alpha t&\cos\alpha t&0&0\\
 0&0&\cos\beta t & -\sin \beta t\\
  0&0&\sin \beta t&\cos\beta t
 \end{bmatrix},
 $$
where $\a$, $\b\in \R$. As remarked in Section \ref{mi},  we call these rotations positive elliptic-elliptic if  $\a \ne0$ and $\b\ne 0$, and positive elliptic if only one of them is zero. 

 A one-parameter subgroup of  $SO(3,1)$ is of the form
 $PB_{\alpha, \beta}(t)P^{-1}$ or $PC_\eta (t)P^{-1}$, with $P\in SO(3,1)$,  and
$$
B_{\alpha, \beta}(t)= \begin{bmatrix}
 \cos\a t & -\sin \a t&0 &0\\
 \sin \a t&\cos\a t&0&0\\
 0&0&\cosh\b t & \sinh \b t\\
  0&0&\sinh \b t&\cosh\b t
 \end{bmatrix},$$
$$
C_\eta (t)= \begin{bmatrix}
  1&0&0&0\\
  0&1&-\eta t&\eta t\\
  0&\eta t&1-\eta t^2/2&\eta t^2\\
  0&\eta t&-\eta t^2&1+\eta t^2/2
  \end{bmatrix},
$$
where $\a$, $\b,\ \eta \in \R$. As remarked in Section \ref{mi},  the negative elliptic, negative hyperbolic, negative elliptic-hyperbolic and parabolic transformations correspond to $\a\ne 0$ and $\b=0$, $\a= 0$ and $\b\ne 0$, $\a\ne 0$ and $\b\ne0$, and $\eta\ne0$, respectively. We can easily check that 
$$A_{\alpha, \beta}(t)=\exp ( \boldsymbol{\xi}_1t),\ \ \! B_{\alpha, \beta}(t)=\exp ( \boldsymbol{\xi}_2t),\ \ \!  C_\eta (t)=\exp ( \boldsymbol{\xi}_3t), $$
where $\boldsymbol{\xi}_1\in \mk{so}(4)$, $\boldsymbol{\xi}_2$, $\boldsymbol{\xi}_3$ $\in \mk{so}(3,1)$, and 
\[    \boldsymbol{\xi}_1 = \begin{bmatrix}0&-\alpha&0&0\\
 \alpha&0&0&0\\ 0&0&0&-\beta\\0&0&\beta&0
 \end{bmatrix},  \ \  \boldsymbol{\xi}_2 = \begin{bmatrix}0&-\alpha&0&0\\
    \alpha&0&0&0\\ 0&0&0&\beta\\0&0&\beta&0
    \end{bmatrix},   \ \ \boldsymbol{\xi}_3 =  \begin{bmatrix}0&0&0&0\\
       0&0&-\eta &\eta \\ 0&\eta &0&0\\0&\eta &0&0
       \end{bmatrix}.  \]
       
It is easy to see that the curved $N$-body problem is invariant under the isometry group of $\M^3$,  which implies that 
for any $\phi$ in the isometry group,  $\left( \q(t),\p(t)\right)$ solves the curved $N$-body problem if and only if $\phi \left( \q(t),\p(t)\right)$ does.
Thus we do not lose anything  by  defining the concept of relative equilibrium for the curved $N$-body problem  with the three normal forms of orthogonal matrices.
To simplify the notation, we will denote initial positions without any argument and attach the argument $t$ to functions depending on time. 
\begin{definition}
Let $\q=(\q_1,\dots,\q_N)$ be a nonsingular initial configuration of the point particles of masses $m_1,\dots,m_N$, $N\geq 2$, on the manifold $\mathbb{M}^3$, where the initial position vectors are $\q_i=(x_i, y_i, z_i, w_i)^T$, $i=\overline{1,N}$. Then a solution of the form 
$$
\q(t)=Q(t)\q:= (Q(t)\q_1,\dots,Q(t)\q_N)
$$ 
of system (\ref{equation}), with $Q(t)$ being $A_{\alpha, \beta}(t)$, $B_{\alpha, \beta}(t)$, or $C_\eta(t)$, is called a  relative equilibrium. 
\end{definition}

It was shown in \cite{Diacu03} and \cite{Diacu05}  that there exist five types of relative equilibria: \emph{positive elliptic} ($\alpha=0$ or $\beta=0$, but not both, in $A_{\alpha, \beta}(t)$), \emph{positive elliptic-elliptic} ($\alpha \neq 0$,  $\beta \neq0$ in $A_{\alpha, \beta}(t)$), \emph{negative  elliptic} ($\alpha \neq 0$,  $\beta = 0$ in $B_{\alpha, \beta}(t)$),  \emph{negative hyperbolic} ($\alpha = 0$,  $\beta \neq 0$ in $B_{\alpha, \beta}(t)$ ), \emph{negative  elliptic-hyperbolic} ($\alpha \neq 0$,  
$\beta \neq 0$ in $B_{\alpha, \beta}(t)$), and there are no relative equilibria corresponding to $C_\eta(t)$. Consequently,  from now on we will
ignore the elements of the form $C_\eta(t)$ in $SO(3,1)$ and deal only with those of the form $B_{\alpha, \beta}(t)$. This is why we did not compute the locked inertia tensor associated to $\boldsymbol{\xi}_3$. 

\subsection{Examples of relative equilibria}

The existence of several classes of relative equilibria is proved in \cite{Diacu03} and \cite{Diacu05}.  We present here some particular examples. Some straightforward computations, or the criteria we give in Section \ref{criteria}, confirm their existence. The reason for presenting these examples will soon become clear.
\begin{example}
 On $\S^3$, let us place three equal masses $m_1=m_2=m_3=\frac{13\sqrt{39}}{512}$ at $\q=(\q_1,\q_2, \q_3),\ \q_j=(x_j,y_j,z_j,w_j)^T,\ j=1,2,3$, where
$$
x_j=\frac{1}{2}\cos  \beta_j,\ \ y_j=\frac{1}{2}\sin \beta_j,\ \ z_j=\frac{\sqrt{3}}{2},\ \ w_j=0,\ \  \beta_j = \frac{2\pi j}{3}.
$$
Then the computations show that $\q(t)=A_{1,0}(t)\q$ is a positive  elliptic relative equilibrium and  $\q(t)=A_{\sqrt{2},1}(t) \q$ is a positive  elliptic-elliptic relative equilibrium.
\end{example}
\begin{example}
 On  $\H^3$, let us place three equal masses $m_1=m_2=m_3=\frac{8\sqrt{2}}{9}$ at $\q=(\q_1,\q_2,\q_3),\ \q_i=(x_i,y_i,z_i,w_i)^T,\ i=1,2,3$, where
\begin{align*}
x_1&=0, & y_1&=0, & z_1&=0, & w_1&=1\\
x_2&=1, & y_2&=0, & z_2&=0, & w_2&=\sqrt{2}\\
x_3&=-1, & y_3&=0, & z_3&=0, & w_3&=\sqrt{2}.
\end{align*}
Then the computations show that $\q(t)=B_{1,0}(t)\q$ is a negative  elliptic relative equilibrium, $\q(t)=B(t)_{0,1}\q$ is a negative hyperbolic relative equilibrium, and $\q(t)=B_{1/2,\sqrt{3}/2}(t)\q$ is a negative  elliptic-hyperbolic relative equilibrium.
\end{example}
Notice that in each example the relative equilibria can be generated from the same initial configuration. Far from being a coincidence, this fact will be clarified in Section \ref{central-config}.

\section{Existence criteria for relative equilibria}\label{criteria}
In this section, we derive criteria for the existence of relative equilibria. They are equivalent with the criteria we gave in \cite{Diacu03} and \cite{Diacu05}, but differ significantly in form, and will be essential in defining the concept of central configuration. 

Let 
$$
\q=(\q_1,...,\q_N),\ \ \q_i=(x_i,y_i,z_i,w_i)^T,\ \ i=\overline{1,N},
$$ 
be a nonsingular configuration in $(\M^3)^N$ and $Q(t)\q$ a relative equilibrium, where $Q(t)$ is $A_{\alpha, \beta}(t)$ or  $B_{\alpha, \beta}(t)$. Again, to simplify the notation, we will denote initial positions and velocities without any argument and attach the argument $t$ to functions depending on time.

To rewrite the criteria proved in \cite{Diacu03} and \cite{Diacu05} in matrix form, we first substitute  $\q_i(t)=Q(t)\mathbf{q}_i,\ i=\overline{1,N}$, into equations (\ref{equation}) and obtain 
\begin{equation*} 
m_i\ddot{Q}(t)\q_i=\nabla_{\q_i} U(t)-\sigma m_i[\dot{Q}(t)\q_i\cdot\dot{Q}(t)\q_i]Q(t)\q_i, \ i=\overline{1,N}.
\end{equation*}
Since $d_{ij}$ is kept cosnstant during the motion, we have
$$
\nabla_{\q_i} U(t)= \sum_{j=1, j\neq i}^N Q(t)\frac{m_im_j [\q_j-\csn d_{ij}\q_i]}{\sn^3 d_{ij}}, 
\ \ i=\overline{1,N}.
$$
Multiplying to the left by $Q^{-1}(t)$ yields
\begin{equation}\label{rigidmotion}
m_iQ^{-1}(t)\ddot{Q}(t)\q_i=\nabla_{\q_i} U-\sigma m_i[\dot{Q}(t)\q_i\cdot\dot{Q}(t)\q_i]\q_i.
\end{equation}

\subsection{Criterion for relative equilibria in $\S^3$}

We can now prove the following criterion for the existence of relative equilibria in $\S^3$.

\begin{criterion} \label{criterion+}
Let $\q=(\q_1,\dots,\q_N), \q_i=(x_i,y_i,z_i,w_i)^T,\ \!  i=\overline{1,N}$,  be a nonsingular  configuration in $\S^3$. Then $A_{\alpha, \beta}(t)\q$ is a relative equilibrium if and only if this configuration satisfies the equations
 \begin{equation}\label{rel-crit+}  
 m_i(\beta^2-\alpha^2)
 \begin{bmatrix}
 x_i(w_i^2 +z_i^2)\\
 y_i(w_i^2 +z_i^2)\\
 -z_i(x_i^2 +y_i^2)\\
 -w_i(x_i^2 +y_i^2)
\end{bmatrix}=\nabla_{\q_i} U,\  i=\overline{1,N}.\end{equation}
 More precisely, the
relative equilibria we obtain are

(i) positive elliptic if $\alpha\ne 0$
and $\beta=0$ or if $\alpha=0$ and $\beta\ne 0$;

(ii) positive elliptic-elliptic if $\alpha\ne 0$ and $\beta\ne 0$.
\end{criterion}
\begin{proof}
Using the fact that $A_{\alpha, \beta}(t)=\exp ( \boldsymbol{\xi}_1t)$ and that $\exp(\boldsymbol{\xi}_1 t)$ and $\boldsymbol{\xi}_1$ commute, straightforward computations show that 
\begin{equation*}
\begin{split}
A^{-1}_{\alpha, \beta}(t)\ddot{A}_{\alpha, \beta}(t)&=
{\rm diag}(-\alpha^2, -\alpha^2, -\beta^2, -\beta^2),\\
\dot{A}_{\alpha, \beta}(t)\q_i\cdot\dot{A}_{\alpha, \beta}(t)\q_i 
&=\alpha^2(x_i^2+y_i^2)+\beta^2(z_i^2+w_i^2).
\end{split}
\end{equation*}
Substituting these expressions into equations (\ref{rigidmotion}), we obtain that
\begin{equation*} 
m_i
\begin{bmatrix}
-\alpha^2x_i\\
-\alpha^2y_i\\
-\beta^2z_i\\
-\beta^2w_i
\end{bmatrix}
=\nabla_{\q_i} U - m_i[\alpha^2(x_i^2+y_i^2) +\beta^2(z_i^2+w_i^2)] \begin{bmatrix}
  x_i
  \\y_i
  \\z_i
  \\w_i
\end{bmatrix},\ \ i=\overline{1,N}. 
\end{equation*} 
Using in the above equations the identity $\q_i \cdot \q_i=1$, we can conclude that
\begin{align*}
x_i\left[-\alpha^2 + \alpha^2(x_i^2+y_i^2) +\beta^2(z_i^2+w_i^2)\right] &= x_i(\beta^2-\alpha^2)(z_i^2+w_i^2),\\
y_i\left[-\alpha^2 + \alpha^2(x_i^2+y_i^2) +\beta^2(z_i^2+w_i^2)\right]&= y_i(\beta^2-\alpha^2)(z_i^2+w_i^2),\\
z_i\left[-\beta^2 + \alpha^2(x_i^2+y_i^2) +\beta^2(z_i^2+w_i^2)\right]&=  -z_i(\beta^2-\alpha^2)(x_i^2+y_i^2),\\
w_i\left[-\beta^2 + \alpha^2(x_i^2+y_i^2) +\beta^2(z_i^2+w_i^2)\right]&=  -w_i(\beta^2-\alpha^2)(x_i^2+y_i^2).
\end{align*}
Then we are led to equations \eqref{rel-crit+}, a remark that completes the proof.
\end{proof}

\subsection{Criterion for relative equilibria in $\H^3$}

We can now provide the following criterion for the existence of relative equilibria in $\H^3$.

\begin{criterion} \label{criterion-}
Let $\q=(\q_1,\dots,\q_N)$, $\q_i=(x_i,y_i,z_i,w_i)^T, \ i=\overline{1,N},$ be a nonsingular configuration in $\H^3$. Then $B_{\alpha, \beta}(t)\q$ is a relative equilibrium if and only if this configuration satisfies the equations
\begin{equation} \label{rel-crit-}
    -m_i(\alpha^2+\beta^2)
    \begin{bmatrix}
    x_i(w_i^2 -z_i^2)\\
    y_i(w_i^2 -z_i^2)\\
    z_i(x_i^2 +y_i^2)\\
    w_i(x_i^2 +y_i^2)
\end{bmatrix}=\nabla_{\q_i} U, \ \ i=\overline{1,N}. 
\end{equation} 
 More precisely, the relative equilibria we obtain are

(i) negative elliptic if $\alpha\ne 0$ and $\beta=0$; 

(ii) negative hyperbolic if $\alpha=0$ and $\beta\ne 0$;

(iii) negative elliptic-hyperbolic if $\alpha\ne 0$ and $\beta\ne 0$. 
\end{criterion}    
\begin{proof}
Using the fact that $B_{\alpha, \beta}(t)=\exp ( \boldsymbol{\xi}_2t)$ and that $\exp(\boldsymbol{\xi}_2 t)$ and $\boldsymbol{\xi}_2$ commute, straightforward computations show that
\begin{equation*}
\begin{split}
B^{-1}_{\alpha, \beta}(t)\ddot{B}_{\alpha, \beta}(t)
&=
{\rm diag}(-\alpha^2, -\alpha^2, \beta^2, \beta^2),\\
\dot{B}_{\alpha, \beta}(t)\q_i\cdot\dot{B}_{\alpha, \beta}(t)\q_i 
&=\alpha^2(x_i^2+y_i^2)-\beta^2(z_i^2-w_i^2).
\end{split}
\end{equation*}
Substituting these results into equations (\ref{rigidmotion}), we obtain
\begin{equation*}  
m_i
\begin{bmatrix}
-\alpha^2x_i\\
-\alpha^2y_i\\
\beta^2z_i\\
\beta^2w_i
\end{bmatrix}
=\nabla_{\q_i} U + m_i[\alpha^2(x_i^2+y_i^2) -\beta^2(z_i^2-w_i^2)] \begin{bmatrix}
  x_i
  \\y_i
  \\z_i
  \\w_i
\end{bmatrix},\ \ i=\overline{1,N}. 
\end{equation*} 
Using in the above equations the identity $\q_i \cdot \q_i=-1$, we can conclude that
\begin{align*}
x_i\left[-\alpha^2 - \alpha^2(x_i^2+y_i^2) +\beta^2(z_i^2-w_i^2)\right] &=x_i(\alpha^2+\beta^2)(z_i^2-w_i^2),\\
y_i\left[-\alpha^2 - \alpha^2(x_i^2+y_i^2) +\beta^2(z_i^2-w_i^2)\right]&=y_i(\alpha^2+\beta^2)(z_i^2-w_i^2),\\
z_i\left[\beta^2 - \alpha^2(x_i^2+y_i^2) +\beta^2(z_i^2-w_i^2)\right]&=-z_i(\alpha^2+\beta^2)(x_i^2+y_i^2),\\
w_i\left[\beta^2 - \alpha^2(x_i^2+y_i^2) +\beta^2(z_i^2-w_i^2)\right]&=-w_i(\alpha^2+\beta^2)(x_i^2+y_i^2).
\end{align*}
Then we are led to equations \eqref{rel-crit-}, a remark that completes the proof.
\end{proof}

\subsection{Relative equilibria and the locked inertia tensor}
In this subsection we take a slightly different point of view. We will apply a general theorem about mechanical systems with symmetry to obtain a new criterion for the existence of relative equilibria for the curved $N$-body problem and then show that this new criterion agrees with the ones proved in the last two subsections. 

Consider a mechanical system of the form $K+V$ ``kinetic plus potential energy"  on some  manifold $M$, where the kinetic energy $K$ is generated by the inner product  $\ll \,,\,\gg_{TM}$ on the tangent bundle of $M$. Let a Lie group  $G$ acting on $M$ preserve the kinetic energy $K$ and the potential  $V$. As mentioned earlier, for each $\boldsymbol{\xi}$ belonging to the Lie algebra $\mathfrak g$ of $G$, there is a vector field $\boldsymbol{\xi}_M$.  Denote by  $\boldsymbol{\xi}_M(\q)$  the vector at $\q\in M$. Then relative equilibria of the mechanical system are solutions of the equations of motion of the system, in the form of  $\exp(\boldsymbol{\xi}t)\q$, where $g\q$ means the action of $g\in G$ on $\q$. In other words, the relative equilibria are both solutions of the system and integral curves of the vector field $\boldsymbol{\xi}_M$. Then there is a theorem due to Smale, which states that relative equilibria can be found by determining, for each $\boldsymbol{\xi} \in \mathfrak{g}$ fixed, the critical points $\q(\boldsymbol{\xi})$ of the so-called effective (or augmented) potential (\cite{Marsden}, p.\ 80, \cite{Smale70-1}), namely the function
\begin{equation}
\label{effective-potential}
V_{\boldsymbol{\xi}}(\q):=V(\q)- \frac{1}{2} \ll \boldsymbol{\xi}_M(\q)\,, \boldsymbol{\xi}_M(\q) \gg_{TM} \, \equiv \,V(\q)- \frac{1}{2} \left< \mathbb{I} (\q)\, \boldsymbol{\xi}\,,  \boldsymbol{\xi} \right>_\mathfrak{g},
\end{equation}
where $\mathbb{I}$ is the locked inertia tensor defined in Section \ref{mi}.  Once such a critical point $\q(\boldsymbol{\xi})$ is found, the relative equilibrium is given by $\q(t) =  \exp(\boldsymbol{\xi} t)\q(\boldsymbol{\xi}).$ \\

The curved $N$-body problem is a $(K+V)$-type mechanical system on $(\M^3)^N$, with $V(\q)=-U(\q)$, and the kinetic energy is generated by the inner product
 \begin{equation*}
  \ll \u \,,\v\gg =\sum _{i=1}^N m_i \u_i \cdot \v_i = \sum _{i=1}^N m_i(u_{x_i}v_{x_i}+u_{y_i}v_{y_i}+u_{z_i}v_{z_i}+\sigma u_{w_i}v_{w_i}),
  \end{equation*}  
where $\u $, $\v\in T_\q(\M^3)^N$, $\u=(\u_1,\cdots, \u_N)$, $\v=(\v_1,\cdots, \v_N)$, and $\u_i \cdot \v_i$ stands for  the dot products in $\R^4$. The matrix Lie group, either $SO(4)$ ($\mathfrak{so} (4)$ being the Lie algebra) or $SO(3,1)$ ($\mathfrak{so} (3,1)$ being the Lie algebra), acts on $(\R^4)^N$ diagonally, where we understand the action on $(\M^3)^N$ of the groups as the induced action. It is easy to see that these actions preserve the kinetic and potential functions. As mentioned, without loss of generality, we can put relative equilibria in the form 
\[\exp(\boldsymbol{\xi}_1t)\q, \ \  \exp(\boldsymbol{\xi}_2t)\q,\]
 where 
$\boldsymbol{\xi}_1$ and $\boldsymbol{\xi}_2$ are defined in Section \ref{relativeequilibria}. Then the vector fields generated by $\boldsymbol{\xi}_1$  and $\boldsymbol{\xi}_2$ on $(\S^3)^N$ and  $(\H^3)^N$ are simply $\boldsymbol{\xi}_1\q$ and $\boldsymbol{\xi}_2\q$, respectively, where $\boldsymbol{\xi}_1$
and $\boldsymbol{\xi}_2$  are $4N\times 4N$  block diagonal matrices. Now we can compute the effective potential to get a new criterion for the existence of relative equilibria.

Recall that the definition of the moment of inertia for a configuration $\q=(\q_1, \cdots, \q_N)$,  $\q_i\in \M^3$ is
\begin{equation*}
I(\q):=\sum_{i=1}^Nm_i(x_i^2+y_i^2).
\end{equation*} 
\begin{criterion}\label{criterion}
Let $\q=(\q_1,\dots,\q_N)$, $\q_i=(x_i,y_i,z_i,w_i)^T, \ i=\overline{1,N},$ be a nonsingular configuration in $\S^3$. Then $\exp(\boldsymbol{\xi}_1t)\q=A_{\alpha, \beta}(t)\q$ is a relative equilibrium if and only if this configuration satisfies the equations
\[ \frac{\b^2-\a^2}{2}\nabla_{\q_i} I(\q)=\nabla_{\q_i} U(\q). \]
Let $\q=(\q_1,\dots,\q_N)$, $\q_i=(x_i,y_i,z_i,w_i)^T, \ i=\overline{1,N},$ be a nonsingular configuration in $\H^3$. Then $\exp(\boldsymbol{\xi}_2t)\q=B_{\alpha, \beta}(t)\q$ is a relative equilibrium if and only if this configuration satisfies the equations
\[-\frac{\a^2+\b^2}{2}\nabla_{\q_i} I(\q)=\nabla_{\q_i} U(\q). \]
\end{criterion}

\begin{proof}
Recall that for one point mass $m$ at $\q=(x,y,z,w)^T$ in $\S^3$ or $\H^3$, equations \eqref{I1} and  \eqref{I2} imply that   
\begin{equation*}
\begin{split}
\ll \boldsymbol{\xi}_1  \q\,, \boldsymbol{\xi}_1  \q \gg_{T\S^3}&=m(\a^2-\b^2)(x^2+y^2) + m \b^2,\\
\ll \boldsymbol{\xi}_2 \q\,, \boldsymbol{\xi}_2  \q \gg_{T\H^3}&=m(\a^2+\b^2)(x^2+y^2) + m \b^2,
\end{split}
\end{equation*}
respectively.  Then for a configuration $\q=(\q_1,\dots,\q_N)\in (\M^3)^N$, we obtain that 
\begin{equation*}
\begin{split}
\ll \boldsymbol{\xi}_1  \q\,, \boldsymbol{\xi}_1  \q \gg_{T(\S^3)^N}&=\sum _{i=1}^N m_i(\a^2-\b^2)(x_i^2+y_i^2) + \sum _{i=1}^N m_i \b^2,\\
\ll \boldsymbol{\xi}_2 \q\,, \boldsymbol{\xi}_2 \q \gg_{T(\H^3)^N}&=\sum _{i=1}^N m_i(\a^2+\b^2)(x_i^2+y_i^2) + \sum _{i=1}^N m_i \b^2.
\end{split}
\end{equation*}
Thus the effective potentials \eqref{effective-potential} with respect to $\boldsymbol{\xi}_1$ and $\boldsymbol{\xi}_2$ are
 \begin{equation*} 
 \begin{split}
 V_{\boldsymbol{\xi}_1} (\q)&=-U(\q)- \sum_{i=1}^ N \frac{m_i}{2} (\a^2-\b^2)(x_i^2+y_i^2),\\
 V_{\boldsymbol{\xi}_2} (\q)&=-U(\q)- \sum_{i=1}^ N \frac{m_i}{2} (\a^2+\b^2)(x_i^2+y_i^2),\\
 \end{split}
 \end{equation*} 
 where we have ignored the constants. Thus  $\exp(\boldsymbol{\xi}_i t)\q$ is a relative equilibrium if and only if $\q$ is a critical point of these effective potentials, which is equivalent to  the two equations as stated in the criterion. 
This remark completes the proof. 
\end{proof}

We claim that the gradient of the moment of inertia $I$ matches the left hand side of the equations in criteria \eqref{rel-crit+} and \eqref{rel-crit-}. Indeed, define $f(x,y,z,w)=x^2+y^2$  as a function from $\M^3$ to $\R$. We employ the trick used to derive the equations of motion. Extend  $f$ to a  homogeneous function $\bar f$ of degree zero, defined in  the ambient space $\R^4$, 
  \[ \bar{f}(x,y,z,w):= \frac{x^2+y^2}{\sigma (x^2+y^2+z^2)+w^2}, \] 
let $\t{\nabla}$ be the gradient in the ambient space, and $\frac{\partial}{\partial n}$ be the unit  normal vector of the unit sphere.  Since $\frac{\partial\bar{f}}{\partial r}=0$, we obtain  $(\t{\nabla}\bar{f})|_{\M^3}=\nabla f+ \frac{\partial \bar{f}}{\partial r}\frac{\partial}{\partial n}=\nabla f$.   Thus straightforward computations show that 
     \begin{equation*}\begin{split}
    \nabla f(x,y,z,w)&= 
   2[x(w^2+  z^2),
             y(w^2+  z^2),
            - z(x^2+ y^2),
             -  w(x^2+ y^2) ]^T \ {\rm in}\ T\S^3, \\
    \nabla f(x,y,z,w) &=2[x(w^2-  z^2),
                 y(w^2- z^2),
                 z(x^2+ y^2),
                   w(x^2+ y^2) ]^T\ {\rm in}\ T\H^3.   
      \end{split}\end{equation*} 
Hence we can conclude that $\nabla_{\q_i} I(\q)$ is given by
\begin{equation*}
   2m_i 
        \begin{bmatrix}
         x_i(w_i^2+ z_i^2)\\
         y_i(w_i^2+  z_i^2)\\
        -  z_i(x_i^2+ y_i^2)\\
         -  w_i(x_i^2+ y_i^2) 
         \end{bmatrix}{\rm in}\ \ \!  T(\S^3)^N\ \ {\rm and}\ \  2m_i 
                 \begin{bmatrix}
                  x_i(w_i^2- z_i^2)\\
                  y_i(w_i^2- z_i^2)\\
                  z_i(x_i^2+ y_i^2)\\
                   w_i(x_i^2+ y_i^2) 
                  \end{bmatrix} {\rm in}\ \ \!  T(\H^3)^N. 
\end{equation*}
Thus Criterion \ref{criterion}
agrees with the criteria obtained in the last two subsections. 

\section{Central configurations}\label{central-config}
In this section we will introduce central configurations in $\S^3$ and $\H^3$ and show their connection with relative equilibria. We will also isolate a particular class of central configurations that correspond to fixed-point solutions in $\S^3$, but which don't exist in $\H^3$, and also introduce various other types of central configurations. Finally we will provide their physical description.

\subsection{Definition of central configurations}
Recall that the central configurations of the Newtonian $N$-body problem are of the form $\q=(\q_1,\dots,\q_N)$, $\q_i=(x_i,y_i,z_i)^T$, with
\[ \nabla_{\q_i}U= \lambda \nabla_{\q_i} \sum_{1\le i\le N} m_i(x_i^2+y_i^2+z_i^2), \ \  i=\overline{1,N},\]
 where $\lambda \in \R$ is a constant and $U$ is the Newtonian force function. The resemblance between these conditions and the   equations occurring in Criterion \ref{criterion} for relative equilibria in $\M^3$ suggests a way to define central configuration of the curved $N$-body problem.
 
\begin{definition}
Assume that the point masses $m_1,\dots, m_N$ in $\M^3$ have the nonsingular positions given by the vector
$$
\q=(\q_1,\dots,\q_N),\ \q_i=(x_i,y_i,z_i,w_i)^T, \ i=\overline{1,N}.
$$
Then $\q$ is a central configuration of the curved $N$-body problem in $\M^3$ if it satisfies the equations
\begin{equation}\label{CCE}
\nabla_{\q_i}U(\q)=\lambda\nabla_{\q_i} I(\q),\ i=\overline{1,N},
\end{equation}
where $\lambda \in \R$ is a constant and $I$ is the moment of inertia. We will further refer to these conditions as the first central configuration equation.
\end{definition} 
 The following class of central configurations  exist in $\S^3$ only, \cite{Diacu03}, \cite{Diacu05}.

\begin{definition}
Consider the masses $m_1,\dots, m_N>0$ in $\S^3$. Then
a configuration
$$
\q=(\q_1,\dots,\q_N),\ \q_i=(x_i,y_i,z_i,w_i)^T, \ i=\overline{1,N},
$$
is called a special central configuration if it is a critical point of the force function $U$, i.e.
\begin{equation}\label{SCCE}
\nabla_{\q_i}U(\q)=0,\ i=\overline{1,N}.
\end{equation}
To avoid any confusion, we will call ordinary central configurations those central configurations that are not special.  
\end{definition}

Special central configurations are obviously central configurations since they satisfy equations \eqref{CCE} with either $\lambda=0$ or $\nabla_{\q_i} I(\q)=0$ for all $i=\overline{1,N}$. We will further see that special central configurations differ from ordinary central configurations in many ways.  

Here is one remark on terminology. These special central configurations were introduced in \cite{Diacu03} under the name of  \emph{fixed points}. Given such a configuration $\q$, we see with the help of Criterion \ref{criterion} that $A_{0,0}(t) \q$ is an associated relative equilibrium, which is a \emph{fixed-point solution}: $\q(t)=\q$, $\p(t)=0$.  This explains the old terminology. 
Let us introduce some new terminology as well.
  
\begin{definition}
A central configuration $\q$ of the curved $N$-body problem is
called 
\begin{itemize}
\item a geodesic central configuration if it is lying on a geodesic;
\item  an $\S^2$ central configuration if it is lying on a great 2-sphere; 
\item   an $\H^2$ central configuration if it is lying on a great hyperbolic 2-sphere;
\item  an $\S^3$ central configuration if it is not lying on any great 2-sphere;
\item   an $\H^3$ central configuration if it is not  lying on any great hyperbolic 2-sphere.
\end{itemize}
$\S^2$ central configurations and  $\H^2$ central configurations
will also be called $\M^2$ central configurations.
\end{definition}
  
 Central configurations will play an important role in the study of the curved $N$-body problem. For example, they influence the topology of the integral manifolds \cite{Marsden,Smale70-1}, 
 and they  are closely related to the relative equilibria.  
   \subsection{Central configurations and solutions of the curved $N$-body problem}
     
     Each central configuration gives rise to a family of relative equilibria, which we call \emph{relative equilibria associated to a central configuration}. Thus it is  advantageous to seek central configurations instead of relative equilibria. We first introduce the following notations
     \begin{equation*} 
      \begin{split}
      \S^1_{xy}:=&\{(x,y,z,w)^T\in \R^4| x^2+y^2=1, z=w=0\},\\
           \S^1_{zw}:=&\{(x,y,z,w)^T\in \R^4| z^2+w^2=1, x=y=0\}, \\
         \H^1_{zw}:=&\{(x,y,z,w)^T\in \R^4| z^2-w^2=-1, x=y=0\}.
      \end{split}
      \end{equation*}
      
\begin{proposition}\label{vanish_of_I}
On $(\S^3)^N$, \[  \nabla_{\q_i}I=0 \ {\rm if  \ and \ only\  if}\ \q_i\in \S^1_{xy}\cup\S^1_{zw}, \]
On $(\H^3)^N$, 
 \[  \nabla_{\q_i}I=0 \ {\rm if  \ and \ only\  if}\ \q_i\in    \H^1_{zw}. \]
\end{proposition}
\begin{proof}
On $(\S^3)^N$, recall that 
  $$
  \nabla_{\q_i} I = 2m_i( x_i(w_i^2+ z_i^2),
           y_i(w_i^2+  z_i^2),
          -  z_i(x_i^2+ y_i^2),
           -  w_i(x_i^2+ y_i^2) )^T.
  $$
 On one hand,  if $\nabla_{\q_i} I$  is a zero vector,  then 
  $$
  (x_i(w_i^2+ z_i^2))^2+(y_i(w_i^2+ z_i^2))^2=(x_i^2+ y_i^2)(w_i^2+ z_i^2)^2=0, 
  $$
  which means that $\q_i\in \S^1_{xy}$ or $\S^1_{zw}$. On the other hand, it is easy to see that if  $\q_i\in \S^1_{xy}\cup \S^1_{zw}$, then $\nabla_{\q_i} I=0$. 

 On $(\H^3)^N$, recall that 
  $$
   \nabla_{\q_i} I = 2m_i( x_i(w_i^2- z_i^2),
            y_i(w_i^2- z_i^2),
             z_i(x_i^2+ y_i^2),
              w_i(x_i^2+ y_i^2) )^T.
   $$
Again, on one hand,  if $\nabla_{\q_i} I$  is a zero vector,  then 
  $$
  (x_i(w_i^2- z_i^2))^2+(y_i(w_i^2- z_i^2))^2=(x_i^2+ y_i^2)(w_i^2- z_i^2)^2=0, 
  $$
  which means that $x_i=y_i=0$, since $w_i^2- z_i^2 = 1+x_i^2+y_i^2\ne 0$. Thus we notice that  $\q_i\in \H^1_{zw}$. On the other hand, it is easy to see that if  $\q_i\in \H^1_{zw}$, then $\nabla_{\q_i} I=0$. This remark completes the proof. 
\end{proof}

A direct consequence of the central configuration equation defined in \eqref{CCE}, \eqref{SCCE}, and discussed in Criterion \ref{criterion} is the following result.
 
 \begin{corollary}\label{CC_RE}
 Consider a central configuration $\q=(\q_1,\dots,\q_N)$, $\q_i=(x_i,y_i,z_i,w_i)^T,$ $i=\overline{1,N},$  in $\M^3$. Let $\lambda$ be the constant in the central configuration equation $\nabla_{\q_i}U(\q)=\lambda\nabla_{\q_i} I(\q)$. 
 \begin{itemize}
 \item If $\q$ is an ordinary central configuration  in  $\S^3$, then it  gives rise to   a one-parameter family of relative equilibria: $A_{\alpha, \beta}(t)\q$ with $\lambda=\frac{\b^2-\a^2}{2}$. 
 \item If $\q$ is in  $\H^3$, then it  gives rise to a one-parameter family of relative equilibria: $B_{\alpha, \beta}(t)\q$ with $\lambda=-\frac{\a^2+\b^2}{2}$. 
 \item If $\q$ is a special central configuration in $\S^3$ and not all  the particles are in $\S^1_{xy}\cup \S^1_{zw}$,  then it  gives rise to a one-parameter family of relative equilibria: $A_{\alpha, \beta}(t)\q$ with $0=\b^2-\a^2$. 
 \item If $\q$ is a special central configuration in  $\S^3$ and all  the particles are in $\S^1_{xy}\cup \S^1_{zw}$,  then it  gives rise to a two-parameter family of relative equilibria: $A_{\alpha, \beta}(t)\q$ with $\a ,\b \in \R$. 
 \end{itemize}
 \end{corollary}   
Before proving this result, let us make the following remark on terminology.
In the literature, the concept of relative equilibrium stands for both the central configurations and the rigid motions associated to them, \cite{Marsden, Smale}. In this paper, however, we use the term  relative equilibrium only for the rigid motions.
    \begin{proof}
    The first two claims are obvious.  If $\q$ is a special central configuration in $\S^3$, then by Criterion \ref{criterion}, $A_{\alpha, \beta}(t)\q$ is an associated relative equilibrium if and only if $\frac{\b^2-\a^2}{2}\nabla_{\q_i}I=0$ for all $i$. 
    
    There are two possibilities: first, if there exists some $\q_i$ with 
    $\nabla_{\q_i}I\ne0$, that is, there is some $\q_i\notin\S^1_{xy}\cup \S^1_{zw},$ then $0=\b^2-\a^2$, i.e., $\q$ gives 
    rise to a one-parameter family of relative equilibria:  $A_{\alpha, \beta}(t)\q$ with $0=\b^2-\a^2$; second, if $\nabla_{\q_i}I=0$ for all $i$,  that is,  $\q_i\in\S^1_{xy}\cup \S^1_{zw}$ for all $i$,  then there is no limitation for $\a,\b$, i.e., $\q$ gives rise to a two-parameter family of relative equilibria: $A_{\alpha, \beta}(t)\q$ with $\a ,\b \in \R$. 
   This remark completes the proof of the statements relative to $\S^3$. 
    \end{proof}
 \begin{remark}
 The reader may notice a gap in the proof. For a central configuration in $\H^3$, we don't have a one-parameter family of relative equilibria, as  claimed, unless  we can show that the value of $\lambda$ is always negative. This fact will be proved in Section \ref{lambda}, where the value of $\lambda$ will be explicitly computed. 
 \end{remark} 
   
Let us notice that while spatial central configurations of  the Newtonian $N$-body problem  do not have associated relative equilibria, all central configurations of the curved $N$-body problem have associated relative equilibria.  

Now it is easy to explain what happens in Examples 1 and 2 of Section \ref{relativeequilibria}. In Example 1, we can check that the given configuration $\q$ is a central configuration in $\S^3$ with $\lambda=-\frac{1}{2}$. Then we obtain the positive elliptic and positive elliptic-elliptic relative equilibria from it. Similarly, in Example 2,  the given configuration $\q$ is a central configuration in $\H^3$ with $\lambda=-\frac{1}{2}$, and we obtain the negative elliptic, negative hyperbolic, and negative elliptic-hyperbolic relative equilibria from it. 
    
In the family of relative equilibria associated to one central configuration, there are motions of different characteristics. In $\S^3$, the relative equilibria can be positive elliptic and positive elliptic-elliptic. In  $\H^3$, they can be negative elliptic, negative hyperbolic, and negative elliptic-hyperbolic. Furthermore, in $\S^3$, these rigid motions can be periodic or quasi-periodic unless they stem from a special central configuration with a body not on $\S^1_{xy}\cup \S^1_{zw}$. Since
 $\a/\b=\sqrt{1-2\lambda /\b^2}$,  the associated relative equilibria are periodic if we choose $\b$ such that $\sqrt{1-2\lambda /\b^2}$
 is a rational number, and they are quasi-periodic if we choose $\b$ such that $\sqrt{1-2\lambda /\b^2} $ is an irrational number.

However, unlike in the Newtonian $N$-body problem,  central configurations do not provide us with homothetic solutions, which occur only in vector spaces, since they require similarity, \cite{Wintner}. Since $\S^3$ and $\H^3$ are not vector spaces, we cannot derive such orbits from central configurations. 
    
 We end this subsection with stating the following property, which is a direct consequence from Proposition \ref{vanish_of_I}.
 \begin{corollary}\label{SCC_on_S1&S1}
A central  configuration $\q=(\q_1,\dots,\q_N),\ \q_i=(x_i,y_i,z_i,w_i)^T \in \S^1_{xy}\cup \S^1_{zw},   \ i=\overline{1,N}$,  is a special central configuration.  
 \end{corollary}   
     
   \subsection{A physical description of central configurations}
In this subsection, we study the vector field $\nabla_{\q_i}I(\q)$, which provides us with  a physical description of central configurations and brings  restrictions to the conditions imposed on ordinary  $\S^2$ and  $\H^2$ central configurations.

Let us start with a result that gives a geometric interpretation of the moment of inertia $I$. 
\begin{lemma}\label{inertia}
In $(\S^3)^N$, 
\[I=\sum_{1\le i\le N} m_i(x_i^2+y_i^2)=\sum_{1\le i\le N} m_i \sin^2 d(\q_i, \S^1_{zw}),  \]
and in $(\H^3)^N$,  
\[I=\sum_{1\le i\le N} m_i(x_i^2+y_i^2)=\sum_{1\le i\le N} m_i \sinh^2 d(\q_i, \H^1_{zw}), \]
where  $d(A,M):= \min_{B\in M}d(A,B)$, with $A,B$ representing points and $M$ being a set. 
\end{lemma}
  
  \begin{proof}
It suffices to show that for a point $A=(x,y,z,w)^T$ in $\S^3$ or  $\H^3$, 
\[ \sigma z^2+w^2=\cos ^2 d(A, \S^1_{zw}) \ \ {\rm and}\ \ \cosh ^2 d(A, \H^1_{zw}), \ {\rm respectively}. \]   

Consider in $\R^4$ the vectors $A$, $e_z=(0,0,1,0)$, and $e_w=(0,0,0,1)$. Let 
$\R^3_A$ be the 3- (or 2-) dimensional subspace spanned by the these three vectors and $\R^2_{zw}$ the 2-dimensional subspace spanned by $e_z$ and $e_w$. 

In $\S^3$, the minimal geodesic connecting $A$ and $\S^1_{zw}$ is on the great 2-sphere $\S^2_A=\R^3_A\cap \S^3$. Let $\th=d(A, \S^1_{zw})$, then $A=A_v+A_h\in (\R^2_{zw})^\bot\oplus\R^2_{zw}$ with $||A_v||=\sin \th$ and $||A_h||=\cos \th$. Hence, we obtain
$$\cos^2 d(A,\S^1_{zw})=||A_h||^2=|| (A\cdot e_z ) e_z+(A\cdot e_w)  e_w ||^2=|| z  e_z+w  e_w ||^2=z^2+w^2.$$

In $\H^3$, the minimal geodesic connecting $A$ and $\H^1_{zw}$ is on the great hyperbolic 2-sphere $\H^2_A=\R^3_A\cap \H^3$. Let $\th=d(A, \H^1_{zw})$, then $A=A_v+A_h\in (\R^2_{zw})^\bot\oplus\R^2_{zw}$ with $||A_v||=\sinh \th$ and $||A_h||=\cosh \th$. Hence, we obtain
\begin{equation*} 
 \begin{split}
\cosh^2 d(A,\S^1_{zw})&=||A_h||^2=|| \frac{A\cdot e_z}{e_z\cdot e_z}  e_z+\frac{A\cdot e_w}{e_w\cdot e_w}  e_w ||^2\\
      &=|| z  e_z-(-w)  e_w ||^2=|(z  e_z+w  e_w)\cdot(z  e_z+w  e_w)  |\\
      &=|z^2-w^2|=-z^2+w^2,
 \end{split}
 \end{equation*}
since  $z^2-w^2=-1-x^2-y^2<0$. This remark competes the proof.
  \end{proof}

\begin{theorem}[The second central configuration equation]
A nonsingular configuration $\q=(\q_1,\dots,\q_N),\ \q_i=(x_i,y_i,z_i,w_i)^T,   \ i=\overline{1,N}$, in $\M^3$ is a central configuration if and only if  
\begin{equation}\label{CCE2}
 \begin{split}
\nabla_{\q_i}U(\q)&=\lambda m_i \sin [2 d(\q_i, \S^1_{zw})]\nabla_{\q_i}d(\q_i, \S^1_{zw}),\ i=\overline{1,N}, \ {\rm in}\ \S^3,\\
 \nabla_{\q_i}U(\q)&=\lambda m_i \sinh [2 d(\q_i, \H^1_{zw})]\nabla_{\q_i}d(\q_i, \H^1_{zw}),\ i=\overline{1,N}, \ {\rm in}\ \H^3,
 \end{split}
 \end{equation}
 where $\lambda \in \R$ is a constant.
\end{theorem}
\begin{proof}
The previous proposition yields 
\[ \nabla_{\q_i} I= m_i \sin [2 d(\q_i, \S^1_{zw})]\nabla_{\q_i}d(\q_i, \S^1_{zw}), \]
 \[ \nabla_{\q_i} I= m_i\sinh [2 d(a, \H^1_{zw})]\nabla_{\q_i}d(\q_i, \H^1_{zw}). \]
 Then the first central configuration equation is equivalent to \eqref{CCE2}. 
\end{proof}

Let us give some physical interpretation for an ordinary central configuration. Denote by $\M^1_{zw}$ either of the geodesics $\S^1_{zw}$ and $\H^1_{zw}$. Then an ordinary central configuration is a special position of the particles in $\M^3$  with the property that the gravitational acceleration vector produced on each particle by all the others particles points toward the geodesic $\M^1_{zw}$ and is proportional to $\sn [2 d(\q_i, \M^1_{zw})]$. By the definition of special central configurations, we see that they are  special arrangements of the particles such that $\F_i=\sum_{j=1, j\ne i}^{N} \F_{ij} =0$ for each $i=\overline{1,N}$.  

This geometrical description of the vector field $\nabla (x^2+y^2)$ brings restrictions to the condition for  ordinary $\M^2$ central configurations.  Recall that  they  are ordinary central configurations lying on a great sphere or a great hyperbolic sphere $\M^2$.  Let $\q$ be such a central configuration on $\M^2$. $\nabla_{\q_i}  U$ is always  tangential   to $(\M^2)^N$ from how the gravitational law is defined in $\M^3$. But the minimal geodesic connecting $\q_i$ and $\S^1_{zw}$ $(\H^1_{zw})$ may not lie on that particular $\M^2$, thus $\nabla_{\q_i} I$ might not be tangential  to  
$(\M^2)^N$, which means that the configuration cannot be central. Although we can find restrictions to $\M^2$ following this  geometric approach, we will further use an analytic argument, which is easier to explain.

\begin{proposition} \label{2d_CC1}
    Consider the nonempty set 
\[  \M^2:=  \{ (x,y,z,w)^T\in \M^3\ \! |\ \! ax+by+cz+dw=0\},  \]
and suppose that there exists an ordinary central configuration in $\M^2$.  Then  $(a,b)=(0,0)$ or $(c,d)=(0,0)$.   
     \end{proposition} 
 \begin{proof}
 Recalling that $\nabla _{\q_i}U$ is a linear combination of the position vectors,  we see that $\nabla _{\q_i}U$ lies  on the 3-dimensional space  $ax+by+cz+dw=0$ in $\R^4$. Thus to have an ordinary central configuration, it is necessary that
 \[ \nabla (x^2+y^2)=(x(w^2+ \sigma z^2),y(w^2+  \sigma z^2), -\sigma  z(x^2+ y^2), - \sigma  w(x^2+ y^2) )^T\]
also lies in this 3-dimensional space, i.e., 
 \begin{equation*} 
  \begin{split}
  0&=ax(w^2+ \sigma z^2)+by(w^2+  \sigma z^2) -c\sigma  z(x^2+ y^2) -d \sigma  w(x^2+ y^2)\\
  &=(ax+by)(w^2+ \sigma z^2)-(cz+dw)(x^2+ y^2).\\
  \end{split}
  \end{equation*}
Using that $ax+by+cz+dw=0$ and $(x^2+ y^2)=\sigma-\sigma(w^2+  \sigma z^2)$, we obtain
\[ ax+by=0, \ \ cz+dw=0.  \] 
Thus the original  3-dimensional space $ax+by+cz+dw=0$ is equivalent to the intersection of the spaces,
\[  \{ax+by=0, \ \ cz+dw=0, \ \ ax+by+cz+dw=0\}. \]
So the linear space spanned by the three vectors $(a,b,0,0)$, $(0,0,c,d)$, and $(a,b,c,d)$ is the 1-dimensional space spanned by $(a,b,c,d)$, which implies that we have either $(a,b)=(0,0)$ or $(c,d)=(0,0)$.   
 \end{proof}

\section{Equivalent central configurations}\label{count}
In this section we find a way to count central configurations.
With a convention to be introduced soon, we will see that that there are infinitely many central configurations for three equal masses on $\S^3$.  Then we will show that any $\M^2$ central configuration is equivalent to some central configuration on one of four particular great 2-spheres, and that any geodesic central configuration is equivalent to some central configuration on one of two particular geodesics.

Recall that the central configuration equation for the Newtonian $N$-body problem in Euclidean space is given by the system 
$$ \nabla_{\q_i}U(\q)=\lambda \nabla_{\q_i} \sum_{i=1}^N m_i(x_i^2+y_i^2+z_i^2),\ i=\overline{1,N},$$
which is invariant under the Euclidean similarities of $\R^3-$ dilations and the isometry group $O(3)$. Thus we call two central configurations $\q$ and $\q'$ \emph{equivalent} if there is a 
 constant $k\in \R$ and a $3\times 3$ orthogonal matrix $Q$ such that $\q_i'=kQ\q_i,$  $i=\overline{1,N}$. So we count central configurations by counting the corresponding equivalence classes.

Things are different for central configurations of the curved $N$-body problem is different. First, we do not have dilations  since the space $\M^3$ is not linear; second, while the  special central configuration equation,  $\nabla_{\q_i}U=0 $,  is obviously invariant under the group  $O(4)$,  the ordinary central configuration equation, $\nabla_{\q_i}U=\lambda \nabla_{\q_i}I$, is not invariant under the isometry group, a fact implied by Proposition \ref{2d_CC1}.  Nevertheless, we can define equivalent classes of central configurations using the subgroup that preserves the central configuration equation. 
\begin{definition}\label{equi-central configuration}
Let  
$$\q=(\q_1,\dots,\q_N),\ \q_i=(x_i,y_i,z_i,w_i)^T,\   \ i=\overline{1,N},$$
$$\q'=(\q'_1,\dots,\q'_N),\ \q'_i=(x'_i,y'_i,z'_i,w'_i)^T,\   \ i=\overline{1,N},
$$ 
be two central configurations in $\M^3$.
If $\q$ and $\q'$ are  special central configurations,  then we call them equivalent if there is  $\phi\in SO(4)$, such that $\q_i=\phi\q_i', i=\overline{1,N}$.
If $\q$ and $\q'$ are ordinary central configurations,  then we call them equivalent if there is  $\phi=(\phi_1, \phi_2)\in SO(2)\times SO(2)$  $(SO(2)\times SO(1,1))$, such that $\q_i=\phi\q_i'$ for each $i=\overline{1,N}$, where the action of $\phi$ is understood such that $\phi_1$ acts on the $xy$-plane and $\phi_2$ on the $zw$-plane. 
\end{definition}  
We would like to point out that though the central configuration equation is actually invariant under the action of $O(2)\times O(2)$ or $O(2)\times O(1,1)$, we adopt this definition to keep consistency with the critical point formulation, which will be introduced in Section \ref{criticalpoint}. 
 
Let us now justify the definition, i.e.\ show that the central configuration equation is only invariant under the subgroup $O(2)\times O(2)$ or $O(2)\times O(1,1)$. This is intuitively 
easy to see since only this subgroup keeps $I$, the moment of inertia, invariant. More precisely, let $X(\q)=\nabla _{\q_i}U(\q) -\lambda \nabla _{\q_i}I(\q)$ be the vector field on $(M^3)^N\setminus \D$ defined by the central configuration equation, and $\phi$ an element of the isometry group. We need to show that 
 \begin{center}
 $\phi_* X(\q)= X(\phi \q)$ if and only if $\phi \in O(2)\times O(2)$ or $O(2)\times O(1,1)$, 
 \end{center}
 where $\phi_*$ is the tangent map associated to  $\phi$ and it is just $\phi$ since $\phi$ is a linear map.  
First, note that the force function $U$ depends on the mutual distances between bodies, thus  $U(\q)=U(\phi \q)$. Notice that $\nabla_{\q'_i} U(\q)$ and $\nabla_{\q_i} U(\q)$ are two $(1,0)$ tensors, and we have 
\[ \nabla_{\q'_i} U(\phi \q) =\nabla_{\q'_i} U(\q) = \phi \nabla_{\q_i} U(\q).\]
Second, notice that $I(\phi \q)= \sum_{1\le i\le N} m_i \sn^2 d(\phi \q_i, \M^1_{zw}),$ then
\[  \nabla_{\q'_i} I(\phi \q) =\phi\nabla_{\q_i} I(\phi \q)  = \phi m_i \sn [2 d(\phi \q_i, \M^1_{zw})]\nabla_{\q_i}d(\phi \q_i, \M^1_{zw}). \] 
Thus we obtain that
$$X(\phi \q)=\phi \left(\nabla_{\q_i} U(\q) - \lambda m_i \sn [2 d(\phi \q_i, \M^1_{zw})]\nabla_{\q_i}d(\phi \q_i, \M^1_{zw})\right). $$
Then $\phi X(\q)= X(\phi \q)$ implies that
\[ \sn [2 d(\phi \q_i, \M^1_{zw})]\nabla_{\q_i}d(\phi \q_i, \M^1_{zw}) = \sn [2 d( \q_i, \M^1_{zw})]\nabla_{\q_i}d( \q_i, \M^1_{zw}), \]
which holds if and only if $\phi$ preserves the function $d(\q_i, \M^1_{zw})$, so then it must belong to the claimed subgroup. \\

\subsection{A counting example}
Let us now illustrate this convention by counting the
number of  central configurations in the following  example of Lagrangian central configurations on 
 $$\S_{xyz}^2:=\{(x,y,z,w)^T\in \R^4|x^2+y^2+z^2+w^2=1, w=0\},$$
see Figure \ref{fig:lag+}.
These central configurations have already appeared in Section \ref{relativeequilibria} in connection with our first example of relative equilibria and we will discuss them in more detail in Section \ref{example}. 
 So let us place three equal masses $m_1=m_2=m_3=1$ at 
 $$\q=(\q_1,\q_2, \q_3),\ \q_j=(x_j,y_j,z_j,w_j)^T,\ j=1,2,3,$$
 $$
 x_j=\sqrt{1-c^2}\cos  \beta_j,\ \ y_j=\sqrt{1-c^2}\sin \beta_j,\ \ z_j=c,\ \ w_j=0,\ \  \beta_j = \frac{2\pi j}{3},
 $$
where $c$ could have any value between  $-1$ and $1$. It is easy to verify that these are all central configurations. By the convention we introduced,  rotating the central configurations in the $xy$-plane does not lead to new central configurations, and the rotated ones still remain on the original 2-sphere; rotating them in the $zw$-plane does not lead to new central configurations either, although they will not remain on the original 2-sphere. Though all these central configurations are similar in some sense, there does not exist an element in $SO(2)\times SO(2)$ to relate any two of them. Thus we see that there is an infinite number of equivalent classes of central configurations for the three equal masses. In other words, the set they form has the power of the continuum.  

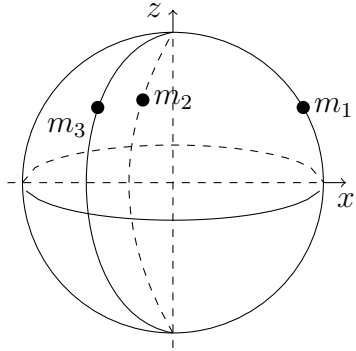
\begin{figure}[!h]
      	\centering
         \begin{tikzpicture}
         \draw (0,  0) circle (2);
         \draw [dashed] (0, -2.2)--(0, 2);
        \draw [->] (0, 2)--(0, 2.3) node (zaxis) [left] {$z$};
         \draw [dashed] ( -2.2,0)--(2,0);
        \draw [->] (2, 0)--( 2.3,0) node (yaxis) [below] {$x$};
         \draw [dashed] [domain=-2:2] plot (\x, {sqrt(4-\x*\x)/4}); 
                \draw  [domain=-1.95:1.95] plot (\x, {-sqrt(4-\x*\x)/4});

          \draw[dashed]  (0,-2) to [out=120,in=240] (0,2);
          \draw  (0,-2) to [out=170,in=190] (0,2);

         \fill (30:2) circle (2.5pt) node[right] {$m_1$};
          \fill ( -0.4,1.1) circle (2.5pt) node[ right] {$m_2$}; 
           \fill ( -1,1) circle (2.5pt) node[below left] {$m_3$};  
          
          \end{tikzpicture}
     \caption{Lagrangian  central configurations on $\S^2_{xyz}$    }
       \label{fig:lag+}  
       \end{figure}    
We will see that this property is common for all given masses, and we will return to this topic after we prove the existence of central configurations for any given masses in Section \ref{criticalpoint}.\\

\subsection{Reduction results}
 We will further show how some types of central configurations can be studied in simpler settings. Let us denote
 \begin{equation*} 
 \begin{split}
 \S_{xyz}^2:&=\{(x,y,z,w)^T\in \R^4\ \! |\ \!  x^2+y^2+z^2+w^2=1, \ w=0\},\\
 \S_{xzw}^2:&=\{(x,y,z,w)^T\in \R^4 \ \! |\ \!   x^2+y^2+z^2+w^2=1, \ y=0\},\\
 \H_{xyw}^2:&=\{(x,y,z,w)^T\in \R^4 \ \! |\ \!   x^2+y^2+z^2-w^2=-1, \ z=0\},\\
 \H_{xzw}^2:&=\{(x,y,z,w)^T\in \R^4 \ \! |\ \!   x^2+y^2+z^2-w^2=-1, \ y=0\}.
 \end{split}
 \end{equation*}
With the above  definition of equivalent central configurations,  the result in Proposition \ref{2d_CC1} can be restated as follows. 
 \begin{theorem}\label{2d_CC2}
 Any $\H^2$ central configuration  is equivalent to some  central configuration on $\H_{xyw}^2$ or $\H_{xzw}^2$.  Any $\S^2$ central configuration  is equivalent to some  central configuration on $\S_{xyz}^2$ or $\S_{xzw}^2$. Furthermore, there is a one-to-one correspondence between the central configurations on  $\S_{xyz}^2$ and the central configurations on  $\S_{xzw}^2$.
 \end{theorem}
 \begin{proof}
 The above statement is obvious for special central configurations. For an ordinary $\M^2$ central configuration $\q$, by Proposition \ref{2d_CC1},  we assume that it lies on  
 \[ \{ (x,y,z,w)^T\in \M^3\ \!|\ \! \cos \th_1x+\sin \th_1y=0\}. \]
 Let $\phi=(\phi_1, id)$, and
  \[ \phi_1=\begin{bmatrix}
   \cos(\pi/2-\th_1) & -\sin (\pi/2-\th_1) \\
   \sin (\pi/2-\th_1) &\cos(\pi/2-\th_1) \end{bmatrix}=\begin{bmatrix}
     \sin \th_1 & -\cos \th_1 \\
     \cos \th_1  &\sin \th_1  \end{bmatrix}\in SO(2).\]
 Then $\phi (x,y,z,w)^T=(\sin \th_1 x-\cos \th_1 y, 0, z,w)$. Hence $\q$ is equivalent to $\phi \q$, which is on $\S_{xzw}^2$  $(\H_{xzw}^2)$. 
 
 Similarly, we can show that an  $\M^2$ central configuration $\q$  on  
  \[ \{ (x,y,z,w)^T\in \M^3 \ \! |\ \!   cz+dw=0\} \]
   is  equivalent to some central configuration on $\S_{xyz}^2$  $(\H_{xyw}^2)$. 
   
   Let $\q$ be  a central configuration  on $\S^2_{xyz}$, i.e.\ $\nabla _{\q_i}U(\q) -\lambda \nabla _{\q_i}I(\q)=0, i=\overline{1,N}$.  Consider the orthogonal transformation given by  
   \[ \varphi (x_i,y_i,z_i,w_i)=(z_i,w_i,x_i,y_i). \]
  Then $\q_i'=\varphi \q_i$ is on  $\S^2_{xzw}$ and 
  \[ I(\q')=\sum_{i=1}^N m_i({x'}_i^2+{y'}_i^2)=\sum_{i=1}^N m_i(1-x_i^2-y_i^2)= \sum_{i=1}^N m_i -I(\q). \]
 Recall that $\nabla _{\q'_i}f=\varphi \nabla _{\q_i}f$ for any smooth function $f$. We then obtain that
 \[\nabla _{\q'_i}U(\varphi\q) +\lambda \nabla _{\q'_i}I(\varphi\q)= \varphi\left(\nabla _{\q_i}U(\q) +\lambda \nabla _{\q_i}[-I(\q)]\right)=0.  \] 
  Then the configuration $\q'=\varphi \q$ is a central configuration on $\S^2_{xzw}$, a remark that completes the proof.
 \end{proof}
 
Let us now study the vector fields $\nabla (x^2+y^2)$ on great spheres  and great hyperbolic spheres.  Our goal is to find all the geodesics to which the vector field is tangential. 
 \begin{proposition}\label{3}
Assume that a geodesic $\M^1$ on $\S_{xyz}^2$, $\S_{xyw}^2$,   $\H_{xyw}^2$, or $\H_{xzw}^2$ is given by the nonempty set   
  \[ \{ (x,y,z,w)^T\in \M^2 \ \! |\ \!   ax+by+cz+dw=0 \}. 
  \]
  Then the vector field $\nabla (x^2+y^2)$ is tangential to this geodesic if and only if  $(a,b)=(0,0)$ or $(c,d)=(0,0)$. 
 \end{proposition}
  \begin{proof}
   We can use here the same argument as in the proof of Proposition \ref{2d_CC1}. 
  \end{proof}
  We sketch these vector fields in Figures \ref{fig:nablaI1} and \ref{fig:nablaI2}, with the explanations below. 
 
  In $\S^2_{xyz}$, $(a,b)=(0,0)$ leads to the great circle $\S^1_{xy}$, and $(c,d)=(0,0)$ gives the great circles passing through $(0,0,1,0)$.
  
   In $\S^2_{xzw}$, $(a,b)=(0,0)$ leads to the great circles passing through $(1,0,0,0)$, and $(c,d)=(0,0)$ gives the great circle $\S^1_{zw}$. 
   
In $\H^2_{xyw}$, $(a,b)=(0,0)$ leads to an empty set, and $(c,d)=(0,0)$ gives the great hyperbolic circles passing through $(0,0,0,1)$.
   
  In $\H^2_{xzw}$, $(a,b)=(0,0)$ leads to the great hyperbolic circles
  \[ \{ (x,y,z,w)^T\in \H^2_{xzw}| \cosh \th z- \sinh \th w=0 \}, \ \th \in \R, \]
  and $(c,d)=(0,0)$ gives the great hyperbolic circles $\H^1_{zw}$.

\begin{figure}[!h]
 	\centering
    \begin{tikzpicture}
    \vspace*{0cm}\hspace*{-4cm}
    \draw (0,  0) circle (2);
        \draw [dashed] (0, -2.2)--(0, 2);
        \draw [dashed] ( -2.2,0)--( 2,0);
        \draw [dashed] ( .8,0.5)--( -0.8,-0.5);
       \draw [->] (0, 2)--(0, 2.3) node (zaxis) [left] {$z$};
       \draw [->] (2,0)--( 2.3,0) node (zaxis) [below] {$y$};
       \draw [->] ( -0.8,-0.5)--(-1,  -0.625)node (zaxis) [below] {$x$};
         \draw [dashed] [domain=-2:2] plot (\x, {sqrt(4-\x*\x)/4}); 
                \draw  [domain=-1.95:1.95] plot (\x, {-sqrt(4-\x*\x)/4});

        \draw (0,2) arc (135:225:2.8);
         \draw (0,-2) arc (-45:45:2.8); 
            
        \draw[->] (70:2) -- (69:2.);   
           \draw[->] (20:2) -- (19:2);
           
               \draw[->] (110:2) -- (111:2.);
               \draw[->] (160:2) -- (161:2);
            
            \draw[->] (200:2) -- (199:2.);       
            \draw[->] (250:2) -- (249:2);
            
            \draw[->] (290:2) -- (291:2.);        
            \draw[->] (340:2) -- (341:2);
        
         \draw[->] (.22,1.7) -- (.25,1.65);
           \draw[->] (.8,.2) -- (.8,0);
           
           \draw[->] (-.24,1.7) -- (-.27,1.65);
              \draw[->] (-.8,.2) -- (-.8,0);
              
              \draw[->] (.3,-1.6) -- (.4,-1.5);
                 \draw[->] (.74,-.7) -- (.75,-.6);
                 
                \draw[->] (-.3,-1.6) -- (-.4,-1.5);
                         \draw[->] (-.74,-.7) -- (-.75,-.6);
           
         \draw [fill] (90:2) circle [radius=0.05];
         \draw [fill] (0:2) circle [radius=0.05];
         \draw [fill] (180:2) circle [radius=0.05];
         \draw [fill] (270:2) circle [radius=0.05]; 
         
             \draw [fill] (.78,-.45) circle [radius=0.05];
              \draw [fill] (-0.78,-.45) circle [radius=0.05]; 
\vspace*{0cm}\hspace*{8 cm}           
     \draw (0,  0) circle (2);
        \draw [dashed] (0, -2.2)--(0, 2);
        \draw [dashed] ( -2.2,0)--( 2,0);
        \draw [dashed] ( .8,0.5)--( -0.8,-0.5);                            
       \draw [->] (0, 2)--(0, 2.3) node (zaxis) [left] {$x$};
       \draw [->] (2,0)--( 2.3,0) node (zaxis) [below] {$w$};
        \draw [->] ( -0.8,-0.5)--(-1,  -0.625)node (zaxis) [below] {$z$};
        \draw [dashed] [domain=-2:2] plot (\x, {sqrt(4-\x*\x)/4}); 
        \draw  [domain=-1.95:1.95] plot (\x, {-sqrt(4-\x*\x)/4});

    \draw (0,2) arc (135:225:2.8);
     \draw (0,-2) arc (-45:45:2.8); 
        
    \draw[->] (69:2.)--(70:2);   
    \draw[->]   (19:2)--(20:2);
    
        \draw[->]  (111:2.)--(110:2);
        \draw[->]   (161:2)--(160:2);
        
       \draw[->]  (199:2)-- (200:2);       
               \draw[->]  (249:2)--(250:2);
               
  \draw[->]   (291:2.)--(290:2);        
 \draw[->]  (341:2)--(340:2);
    
     \draw[->]  (.25,1.65)--(.22,1.7);
       \draw[->]  (.8,0)-- (.8,.2);
       
       \draw[->] (-.27,1.65)-- (-.24,1.7) ;
          \draw[->]   (-.8,0)--(-.8,.2);
          
          \draw[->]  (.4,-1.5)--(.3,-1.6);
             \draw[->]  (.75,-.6)--(.74,-.7);
             
            \draw[->]   (-.4,-1.5)--(-.3,-1.6);
                     \draw[->]  (-.75,-.6)-- (-.74,-.7);
       
     \draw [fill] (90:2) circle [radius=0.05];
     \draw [fill] (0:2) circle [radius=0.05];
     \draw [fill] (180:2) circle [radius=0.05];
     \draw [fill] (270:2) circle [radius=0.05]; 
     
         \draw [fill] (.78,-.45) circle [radius=0.05];
          \draw [fill] (-0.78,-.45) circle [radius=0.05];

     \end{tikzpicture}
\caption{$\nabla (x^2+y^2)$  on $\S_{xyz}^2$  and $\S_{xzw}^2$}
  \label{fig:nablaI1}  
\end{figure}
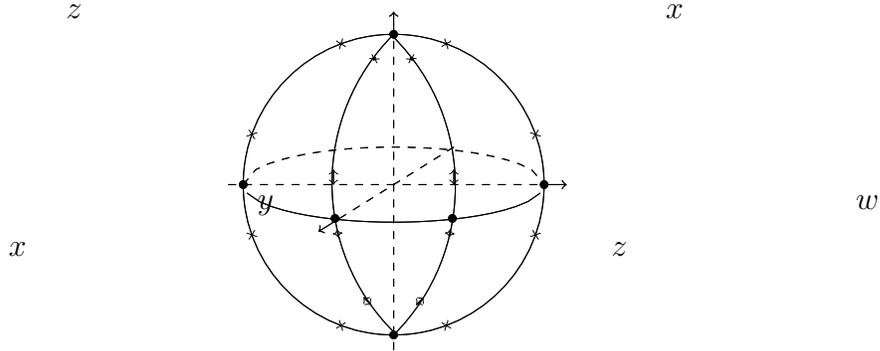

\begin{figure}[!h]
 	\centering
    \begin{tikzpicture}
    \vspace*{0cm}\hspace*{-4cm}
     \draw[black, domain=-2:2] plot (\x, {sqrt(\x*\x+1)});
      \draw [->] (0, 0)--(0, 2.5) node (zaxis) [above] {$w$};
     \draw [->](-1.3, 0)--(1.3, 0) node (yaxis) [right] {$y$};                                                             
      \draw [->]( .8,0.5)--( -0.8,-0.5) node (yaxis) [right] {$x$};
      \draw[black, domain=-0.5:0] plot (\x, {sqrt(30*\x*\x+1)});
       \draw[dashed, domain=0:0.5] plot (\x, {sqrt(30*\x*\x+1)}); 
       
 \draw[black, domain=-1:0] plot (\x, {sqrt(5*\x*\x+1)});
        \draw[dashed, domain=0:1] plot (\x, {sqrt(5*\x*\x+1)});       
  \draw[->] (0.5,1.12)--(0.6,1.17);
  \draw[->] (1,1.42)--(1.1,1.49);
 \draw[->] (-0.5,1.12)--(-0.6,1.17);
   \draw[->] (-1,1.42)--(-1.1,1.49);
 \draw[->] (0.4,1.34)--(0.45,1.42);
 \draw[->] (0.7,1.86)--(0.75,1.95);
   \draw[->] (-0.4,1.34)--(-0.45,1.42);
   \draw[->] (-0.7,1.86)--(-0.75,1.95);
   
   \draw[->] (0.2,1.48)--(0.23,1.6); 
   \draw[->] (0.4,2.4)--(0.43,2.56); 
  \draw[->] (-0.2,1.48)--(-0.23,1.6); 
     \draw[->] (-0.4,2.4)--(-0.43,2.56);  
       
 \fill (0,1) circle (1.6pt);

   \vspace*{0cm}\hspace*{8cm}                   
        
 \draw[black, domain=-2:2] plot (\x, {sqrt(\x*\x+1)});
  \draw [->] (0, 0)--(0, 2.5) node (zaxis) [above] {$w$};
 \draw [->](-1.3, 0)--(1.3, 0) node (yaxis) [right] {$z$};                                                             
  \draw [->]( .8,0.5)--( -0.8,-0.5) node (yaxis) [right] {$x$};
  \draw[black, domain=-0.5:0] plot (\x, {sqrt(30*\x*\x+1)});
         \draw[dashed, domain=0:0.5] plot (\x, {sqrt(30*\x*\x+1)}); 
   
\draw(0.52,1.3)..controls (0.3,0.9) and (0.4,1)..(1.4,2);   
\draw[dashed] (0.52,1.3)--(1.07, 2.3);

\draw[->] (0.2,1.48)--(0.23,1.6); 
   \draw[->] (0.4,2.4)--(0.43,2.56); 
  \draw[->] (-0.2,1.48)--(-0.23,1.6); 
     \draw[->] (-0.4,2.4)--(-0.43,2.56);  
       
\draw[->] (0.6,1.46)--(0.67,1.6); 
 \draw[->] (0.9,2)--(1,2.15);  

\draw[->] (0.8,1.4)--(0.9,1.5); 
\draw[->] (1.2,1.8)--(1.3,1.9); 
    
   \fill (0,1) circle (1.6pt);
  \fill (0.45,1.09) circle (1.6pt);

     \end{tikzpicture}
\caption{$\nabla (x^2+y^2)$  on $\H_{xyw}^2$ and $\H_{xzw}^2$ }
  \label{fig:nablaI2}  
\end{figure}
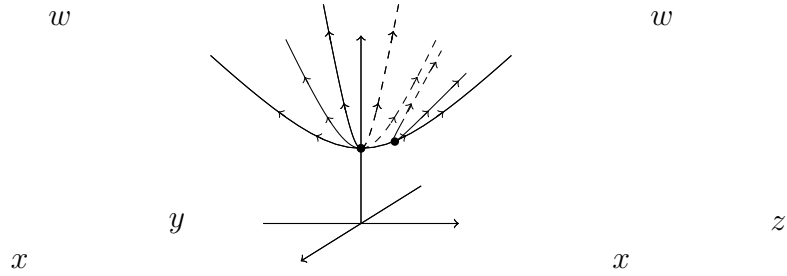

Let us further denote
 \begin{equation*} 
 \begin{split}
 \S^1_{xz}:=&\{(x,y,z,w)^T\in \R^4 \ \! |\ \!   x^2+z^2=1, y=w=0\},\\ 
   \H^1_{xw}:=&\{(x,y,z,w)^T\in \R^4 \ \! |\ \!   x^2-w^2=-1, y=z=0\}.
 \end{split}
 \end{equation*}  
 We can now state and prove the following result.

\begin{theorem}\label{1d_CC}
 Any geodesic central configuration in $\S^3$ is equivalent to some  central configuration on $\S_{xz}^1$.   Any geodesic central configuration in $\H^3$ is equivalent to some  central configuration on $\H_{xw}^1$.
 \end{theorem}

\begin{proof}
By Proposition \ref{3}, we see that a geodesic central configuration is possible only if it lies on
$\S_{xyz}^2$, $\S_{xyw}^2$,   $\H_{xyw}^2$, or $\H_{xzw}^2$. Using the same argument as in the proof of Theorem \ref{2d_CC2},  we see that each of those particular geodesics on $\S^2_{xyz}$ and $\S^2_{xzw}$ can be transformed to $\S^1_{xz}$ by some element in $SO(2)\times SO(2)$, and each of  those particular geodesics on $\H^2_{xyw}$ and $\H^2_{xzw}$ can be transformed to $\H^1_{xw}$ or $\H^1_{zw}$ by some element in $SO(2)\times SO(1,1)$. Now recall that $\nabla_{\q_i}I=0$ if $\q_i\in \H^1_{zw}$, so a geodesic central configuration on $\H^1_{zw}$ must be a special central configuration, which does not exist. This remark completes the proof.
\end{proof} 

Consequently, when looking for $\S^2$ central configurations and geodesic central configurations in $\S^3$, it suffices to seek them on $\S^2_{xyz}$ and $\S^1_{xz}$, respectively. When looking for $\H^2$ central configurations and geodesic central configurations in $\H^3$, it suffices to seek them on $\H^2_{xyw}$, $\H^2_{xzw}$, and $\H^1_{xw}$, respectively. This is a bit surprising, since the symmetry groups for central configurations are just subgroups of the isometry groups.

  \section{Criteria for central configurations and the value of $\lambda$}\label{lambda}
  
  In this section we rewrite the central configuration equations, both in $\S^3$ and $\H^3$, and state them as existence criteria. For practical purposes, the new equations are more useful than the original ones. Although we could merge these results into a single criterion, we prefer to state them separately since this is the way we apply them. But first let us further simplify the notation by taking
  $$
  r_i:=(x_i^2+y_i^2)^{1/2},\ \ \rho_i:=(\sigma z_i^2+w_i^2)^{1/2},
  \ \  \ i=\overline{1,N}.
  $$
  Then we have  
  $$r_i^2+\sigma \rho_i^2=\sigma, \ \  {\rm and}\ \  \rho_i^2>0\ \  {\rm in \ the \ case\  of } \ \H^3. $$
   Recall that 
   \begin{equation*}
  \nabla_{\q_i} U =\sum_{j=1,j\ne i}^N\frac{m_im_j[\q_j - \csn d_{ij}\q_i ]}{\sn^3 d_{ij} }, \ \ \nabla_{\q_i} I(\q) =  2m_i 
          \begin{bmatrix}
           x_i\rho_i^2\\
           y_i\rho_i^2\\
          - \sigma z_ir_i^2\\
           - \sigma w_ir_i^2 
           \end{bmatrix}, \ \  i=\overline{1,N}.
   \end{equation*}

  \subsection{Criterion for central configurations in $\S^3$}
  
  We can now state and prove the following criterion for the existence of central configurations in $\S^3$.
  
  \begin{criterion} \label{criterion++}
  Consider the masses $m_1,\dots, m_N>0$ in $\S^3$ at the configuration 
 $\q=(\q_1,\dots,\q_N)$, $\q_i=(x_i,y_i,z_i,w_i)^T \in \S^1_{xy} \cup \S^1_{zw}$
   for $1\le i\le k$, $k\le N$, and 
 $\q_i=(x_i,y_i,z_i,w_i)^T\notin\S^1_{xy} \cup \S^1_{zw},\ k<i\le N$. 
 Then the first central configuration equation, $\nabla_{\q_i}U=\lambda\nabla_{\q_i}I, i=\overline{1,N},$ is equivalent to the following $3N$ equations: 
for each $i$ with $1\le i\le k$, choose three of the four equations 
  \[  \sum_{j=1,j\ne i}^N\frac{m_im_j[\q_j - \cos d_{ij}\q_i ]}{\sin^3 d_{ij} }=0,\]
and for $k<i\le N$ take 
  \begin{equation*}
  \begin{cases}
  &\sum_{j=1,j\neq i}^N\dfrac{m_j(x_ix_j + y_iy_j -r_i^2\cos d_{ij})}{\sin^3 d_{ij}}=2\lambda r_i^2\rho_i^2\\
  &\sum_{j=1,j\neq i}^N\dfrac{m_j(x_iy_j-x_jy_i)}{\sin^3 d_{ij}}=0\\
  &\sum_{j=1,j\neq i}^N\dfrac{m_j(z_iw_j-z_jw_i)}{\sin^3 d_{ij}}=0. \ \
  \end{cases}
  \end{equation*}
 \end{criterion}
  \begin{proof}
  The idea is to decompose the central configuration equation $\nabla_{\q_i}U=\lambda\nabla_{\q_i}I $ along some basis. For $i>k$,  $\q_i \notin \S^1_{xy} \cup \S^1_{zw}$, i.e.\ $r_i\rho_i\ne0$,  the following four vectors form an orthogonal basis of $\R^4$:
  \begin{equation*} 
  e_1= \frac{(x_i,y_i,0,0)^T}{r_i},\
  e_2= \frac{(-y_i, x_i, 0,0)^T}{r_i},\
  e_3= \frac{(0,0,z_i,w_i)^T}{\rho_i},\
  e_4= \frac{(0,0,-w_i,z_i)^T}{\rho_i}.
  \end{equation*}
  Decomposing $\nabla_{\q_i}U$ and  $\nabla_{\q_i}I$ along these vectors, we obtain 
  \begin{equation} 
  \begin{split}
 \nabla_{\q_i}U=& \sum_{j=1,j\ne i}^N\frac{m_im_j}{r_i\sin^3 d_{ij}} \left[ (x_ix_j +y_iy_j -r_i^2 \cos d_{ij} )e_1 + \begin{vmatrix}
   x_i & y_i\\
   x_j& y_j
   \end{vmatrix}e_2\right]\\
   &+\frac{m_im_j}{\rho_i\sin^3 d_{ij}} \left[ (z_iz_j +w_iw_j - \cos d_{ij} \rho_i^2)e_3 + 
   \begin{vmatrix}
    z_i & w_i\\
    z_j& w_j
    \end{vmatrix}e_4\right],\\
 \nabla_{\q_i}I&=2m_i (r_i\rho_i^2 e_1-\rho_ir_i^2e_3),    \ \ i=\overline{1,N}.
  \end{split}
  \notag
  \end{equation}
  Thus the central configuration equation $\nabla_{\q_i}U=\lambda\nabla_{\q_i}I $ is 
  \begin{equation*} 
  \begin{split}
  &\sum_{j=1,j\neq i}^N\frac{m_j(x_ix_j + y_iy_j -r_i^2\cos d_{ij} )}{\sin^3 d_{ij}}= 2\lambda r_i^2\rho_i^2,\\
  &\sum_{j=1,j\neq i}^N\frac{m_j}{\sin^3 d_{ij}} \begin{vmatrix}
   x_i & y_i\\
   x_j& y_j
   \end{vmatrix} =0,\\
   &\sum_{j=1,j\neq i}^N\frac{m_j(z_iz_j + w_iw_j - \rho_i^2\cos d_{ij})}{\sin^3 d_{ij}}=-2\lambda r_i^2\rho_i^2,\\
   &\sum_{j=1,j\neq i}^N\frac{m_j}{\sin^3 d_{ij}} 
  \begin{vmatrix}
    z_i & w_i\\
    z_j& w_j
  \end{vmatrix} =0, \ \ i=\overline{1,N}.
  \end{split}
  \end{equation*}
  By adding the first and third equation we obtain an
  identity, which means that these equations are
  dependent, so we can eliminate one of them (say, the third) to obtain $3(N-k)$ equations.

 For $i\le k$, $\q_i \in \S^1_{xy} \cup \S^1_{zw}$,  by Proposition \ref{vanish_of_I}, we have  \[\nabla _{\q_i}U=\sum_{j=1,j\ne i}^N\frac{m_im_j[\q_j - \cos d_{ij}\q_i ]}{\sin^3 d_{ij} }=0. \]
 Notice that $\nabla_{\q_i}U\cdot \q_i=0$, so the four equations  are dependent, thus we can select three independent ones as follows. Suppose that $w_i\ne 0$, then we take the first three equations since the first three components being zero imply that the fourth is also zero. This remark completes the proof.
  \end{proof}
    \subsection{Criterion for central configurations in $\H^3$}
    We can now state and prove the following criterion for the existence of central configurations in $\H^3$.
    
  \begin{criterion}\label{criterion--}
    Consider the masses $m_1,\dots, m_N>0$ in $\H^3$ at the configuration 
     $\q=(\q_1,\dots,\q_N)$, where $\q_i=(x_i,y_i,z_i,w_i)^T \in  \H^1_{zw}$
       for $1\le i\le k$, $k\le N$, and 
     $\q_i=(x_i,y_i,z_i,w_i)^T\notin \H^1_{zw}$
     for $k<i\le N$. 
     Then the first central configuration equation, $\nabla_{\q_i}U=\lambda\nabla_{\q_i}I, $ is equivalent to the following $3N$ equations: 
 for each $i$ with $1\le i\le k$, take
  \begin{equation*}
  \begin{cases}
  & \sum_{j=1,j\ne i}^N\dfrac{m_im_j[x_j - \cosh d_{ij}x_i ]}{\sinh^3 d_{ij} }=0\\
  & \sum_{j=1,j\ne i}^N\dfrac{m_im_j[y_j - \cosh d_{ij}y_i ]}{\sinh^3 d_{ij} }=0\\
  & \sum_{j=1,j\ne i}^N\dfrac{m_im_j[z_j - \cosh d_{ij}z_i ]}{\sinh^3 d_{ij} }=0,
\end{cases}
  \end{equation*}
   and for $k<i\le N$, choose
    \begin{equation*}
       \begin{cases}
       &\sum_{j=1,j\neq i}^N\dfrac{m_j(x_ix_j + y_iy_j -r_i^2\cosh d_{ij})}{\sinh^3 d_{ij}}=2\lambda r_i^2\rho_i^2\\
       &\sum_{j=1,j\neq i}^N\dfrac{m_j(x_iy_j-x_jy_i)}{\sinh^3 d_{ij}}=0\\
       &\sum_{j=1,j\neq i}^N\dfrac{m_j(z_iw_j-z_jw_i)}{\sinh^3 d_{ij}}=0.
       \end{cases}
       \end{equation*}
    \end{criterion}
    \begin{proof}
    We use the same idea as in the proof of the previous criterion.  If $\q_i \notin \H^1_{zw}$, i.e.\ $r_i\ne0$,  then  the following four vectors form an orthogonal basis of $\R^{3,1}$:
    \begin{equation*} 
    e_1= \frac{(x_i,y_i,0,0)^T}{r_i},\
    e_2= \frac{(-y_i, x_i, 0,0)^T}{r_i},\
    e_3= \frac{(0,0,z_i,w_i)^T}{\rho_i},\
    e_4= \frac{(0,0,w_i,z_i)^T}{\rho_i}.
    \end{equation*}
    Decomposing $\nabla_{\q_i}U$ and  $\nabla_{\q_i}I$ along these vectors, we obtain 
    \begin{equation} 
    \begin{split}
   \nabla_{\q_i}U=& \sum_{j=1,j\ne i}^N\frac{m_im_j}{r_i\sinh^3 d_{ij}} \left[ (x_ix_j +y_iy_j -r_i^2 \cosh d_{ij} )e_1 + \begin{vmatrix}
     x_i & y_i\\
     x_j& y_j
     \end{vmatrix}e_2\right]\\
     &-\frac{m_im_j}{\rho_i\sinh^3 d_{ij}} \left[ (z_iz_j -w_iw_j + \rho_i^2\cosh d_{ij} )e_3 - 
     \begin{vmatrix}
      z_i & w_i\\
      z_j& w_j
      \end{vmatrix}e_4\right],\\
   \nabla_{\q_i}I&=2m_i (r_i\rho_i^2 e_1+\rho_ir_i^2e_3),    \ \ i=\overline{1,N}.
    \end{split}
    \notag
    \end{equation}
    Thus the central configuration equation $\nabla_{\q_i}U=\lambda\nabla_{\q_i}I $ is 
    \begin{equation*} 
    \begin{split}
    &\sum_{j=1,j\neq i}^N\frac{m_j(x_ix_j + y_iy_j -r_i^2\cosh d_{ij} )}{\sinh^3 d_{ij}}= 2\lambda r_i^2\rho_i^2,\\
    &\sum_{j=1,j\neq i}^N\frac{m_j}{\sinh^3 d_{ij}} \begin{vmatrix}
     x_i & y_i\\
     x_j& y_j
     \end{vmatrix} =0,\\
     &\sum_{j=1,j\neq i}^N\frac{m_j(z_iz_j - w_iw_j + \rho_i^2\cosh d_{ij})}{\sinh^3 d_{ij}}=-2\lambda r_i^2\rho_i^2,\\
     &\sum_{j=1,j\neq i}^N\frac{m_j}{\sinh^3 d_{ij}} 
    \begin{vmatrix}
      z_i & w_i\\
      z_j& w_j
    \end{vmatrix} =0, \ \ i=\overline{1,N}.
    \end{split}
    \end{equation*}
    But adding the first and third equation we obtain an
    identity, which means that these equations are
    dependent, so we can eliminate one of them (say, the third) 
    to obtain the $3(N-k)$ equations. 
    
     If  $\q_i \in \H^1_{zw}$,  then by Proposition \ref{vanish_of_I}, \[\nabla _{\q_i}U=\sum_{j=1,j\ne i}^N\frac{m_im_j[\q_j - \cosh d_{ij}\q_i ]}{\sinh^3 d_{ij} }=0. \]
     Notice that $\nabla_{\q_i}U \cdot \q_i=0$, so the four equations  are dependent and we can select the first  three independent equations. Since $w_i\ge1 $, then the first three components being zero imply that the fourth is zero too.  This remark completes the proof.
    \end{proof}
   \subsection{The value of $\lambda$} 
    A configuration $\q=(\q_1,\dots,\q_N)$, $\q_i=(x_i,y_i,z_i,w_i)^T,\  i=\overline{1,N}$, is a central configuration if it satisfies the central configuration equation \eqref{CCE} or, equivalently, equation \eqref{CCE2} for some constant $\lambda$. It turns out that the value of $\lambda$ is uniquely determined by the central configuration equation. 
     Recall that $\lambda=-U/2I$  for central configurations of the Newtonian $N$-body problem. Here the form of $\lambda$ is not as nice as that.
     
    \begin{corollary}
    For any given ordinary  central configuration $\q=(\q_1,\dots,\q_N)$, $\q_i=(x_i,y_i,z_i,w_i)^T,\  i=\overline{1,N}$, in $\S^3$,  
  \begin{equation}\label{lambda+}
     \begin{split}
     \lambda=\sum_{1\le i\le N} \frac{ \sum_{j=1,j\ne i}^N\dfrac{m_im_j[2x_ix_j+2y_iy_j-(r_i^2+r_j^2)\cos d_{ij}]}{\sin^3d_{ij}}}{2 m_i r_i^2 \rho_i^2}.
     \end{split}
     \end{equation} 
   \end{corollary} 
 \begin{proof}
 From the system of the $3N$ equations derived in Criterion \ref{criterion++}, for $\q_i \notin\S^1_{xy}\cup \S^1_{zw}$ we have  
 \[ \sum_{j=1,j\neq i}^N\frac{m_im_j(x_ix_j + y_iy_j -r_i^2\cos d_{ij} )}{\sin^3 d_{ij}}= 2\lambda m_ir_i^2\rho_i^2, \ {\rm for } \ \q_i \notin\S^1_{xy}\cup \S^1_{zw}, i=\overline{1,N},\]
 and it is easy to see that these equations also hold for $\q_i \in \S^1_{xy}\cup \S^1_{zw}$. Adding these $N$ equations, we obtain that
 \[ \sum_{1\le i\le N} \sum_{j=1,j\ne i}^N \frac{m_im_j\left(2x_ix_j+2y_iy_j-(r_i^2+r_j^2)\cos d_{ij}\right)}{\sin^3d_{ij}}= \sum_{1\le i\le N}  2\lambda m_ir_i^2\rho_i^2.\]
 
 Notice that we have assumed that the central configuration is not a special central configuration, which implies that there is some  $\q_i \notin\S^1_{xy}\cup \S^1_{zw}$ by Corollary \ref{SCC_on_S1&S1}.  Hence we see that $\sum_{1\le i\le N}  2 m_ir_i^2\rho_i^2\ne 0$. Then the above equation leads to equation \eqref{lambda+}, a remark that completes the proof.
 \end{proof}   
    
\begin{corollary}
For central configurations $\q=(\q_1,\dots,\q_N)$, $\q_i=(x_i,y_i,z_i,w_i)^T,\  i=\overline{1,N}$, in $\H^3$, we have
   \begin{equation}\label{lambda-}
   \begin{split}
   \lambda= \sum_{1\le i\le N}\frac{ \sum_{j=1,j\ne i}^N\dfrac{m_im_j[2x_ix_j+2y_iy_j-(r_i^2+r_j^2)\cosh d_{ij}]}{\sinh^3d_{ij}}}{2  m_i r_i^2 \rho_i^2}<0.
   \end{split}
   \end{equation} 
    \end{corollary} 
\begin{proof}
From the system of the $3N$ equations derived in Criterion \ref{criterion--}, for $\q_i \notin \H^1_{zw}$ we have  
\[ \sum_{j=1,j\neq i}^N\frac{m_im_j(x_ix_j + y_iy_j -r_i^2\cosh d_{ij} )}{\sinh^3 d_{ij}}= 2\lambda m_ir_i^2\rho_i^2, \ {\rm for } \ \q_i \notin\H^1_{zw}, i=\overline{1,N},\]
and it is easy to see that these relationships also hold for $\q_i \in \H^1_{zw}$. Adding these $N$ equations, we obtain that
\[ \sum_{1\le i\le N} \sum_{j=1,j\ne i}^N\frac{m_im_j\left(2x_ix_j+2y_iy_j-(r_i^2+r_j^2)\cosh d_{ij}\right)}{\sinh^3d_{ij}}= \sum_{1\le i\le N} 2\lambda m_ir_i^2\rho_i^2.\]
Notice that $\cosh d_{ij}>1$ since $d_{ij}>0$, so 
\begin{equation*} 
 \begin{split}
2x_ix_j+2y_iy_j-(r_i^2+r_j^2)\cosh d_{ij}
<&2x_ix_j+2y_iy_j-(x_i^2+x_j^2+y_i^2+y_j^2)\\
<&-(x_i-x_j)^2-(y_i-y_j)^2<0.
 \end{split}
 \end{equation*}
This implies that  the left hand side of the above identity is not zero, consequently we have 
$\sum_{1\le i\le N} m_i r_i^2 \rho_i^2\ne 0$, a relationship that leads to equation \eqref {lambda-}.
\end{proof}
   
The above result also implies that there is no special central configuration on $\H^3$ that corresponds to the case $\lambda=0$. We can thus conclude that there is no such central configurations with $\q_i\in \H^1_{zw}$ for all $i=\overline{1,N}$.

 \section{Existence of ordinary central configurations }\label{criticalpoint}
In this section we will interpret  central configurations as  critical points of functions related to $U$,  which turns out to be a very useful approach. Then we will prove the existence of ordinary central configurations for any given masses. Finally we will discuss the Wintner-Smale conjecture for the curved $N$-body problem.

\subsection{Central configurations as critical points}
From the first central configuration equation, 
\[ \nabla_{\q_i}U(\q)-\lambda\nabla_{\q_i}I(\q)=0, \]
we can derive the following property.
\begin{proposition}
Central configurations in $\M^3$ are  critical points of the function 
$$U-\lambda I: (\M^3)^N\setminus\Delta\to \R,$$ 
where $\lambda$ is a constant. In the case of the sphere $\S^3$, $\lambda=0$  corresponds to special central configurations.   
\end{proposition}

We can also see that an ordinary  central configuration is a critical point of the restriction of $U$ subject to the constraint $I=constant$. From this point of view, $-\lambda$ is a Lagrange multiplier. More precisely, let us denote 
\[ S_c:=\{  \q\in(\M^3)^N\setminus\Delta\ \! | \ \! I(\q)=c\}. \]
We then have the following result.
\begin{proposition}
 Ordinary central configurations in $\M^3$ are critical points of
  $U|_{S_c}$, i.e.\ critical points of  
  $$  U : \ S_c \to\R.$$  
\end{proposition}

Let $\q$ be an ordinary central configuration and $\phi$ an element of $SO(2)\times SO(2)$ or  $SO(2)\times SO(1,1)$.  Then  $\phi \q$ is also a central configuration. Thus it follows that the critical points of  $U|_{S_c}$ are not isolated, but rather occur as manifolds of critical points. Similarly, these special central configurations are not isolated either.  This fact suggests that we can further look at the central configurations as critical points of $U$ subject to a quotient manifold. Note that both $U$ and $(\M^3)^N$ are invariant under the isometry group,  and the set $S_c$ is invariant under the subgroup $SO(2)\times SO(2)$ or $SO(2)\times SO(1,1)$. We thus have the following property.

 \begin{proposition}\label{six}
 There is a one-to-one correspondence between the classes of central configurations and the critical points of the force function $\hat{U}$ induced by $U$ on the quotient set 
 \begin{enumerate}
 \item $((\S^3)^N\setminus \D) /SO(4)$
for special central configurations in $\S^3$,
 \item $S_c/(SO(2)\times SO(2))$
 for ordinary central configurations in $\S^3$, and
 \item $S_c/(SO(2)\times SO(1,1))$ for central configurations in $\H^3$.
 \end{enumerate}
 \end{proposition}          
 
  Let $\q$ in the quotient set be a critical point of $\hat{U}$. In the case of special central configuration on $\S^3$, the Hessian of $\hat{U}$ at $\q$, $D^2\hat{U}(\q)$, is an invariantly defined symmetric bilinear form on $T_{\q} ((\S^3)^N\setminus \D) /SO(4)  )$. For ordinary central configurations in $\S^3$ and $\H^3$, $D^2\hat{U}(\q)$ is an invariantly defined symmetric bilinear form on $T_{\q} \hat{S}_c$, where $\hat{S_c}$ is the quotient set in either (2) or (3) of Proposition \ref{six}.   The index of $D^2\hat{U}(\q)$ is the maximal dimension of a subspace of the tangent space 
   on which this form is negative definite. A critical point $\q$  of $\hat{U}$ is degenerate whenever the Hessian has a non-trivial nullspace.

 We can now formally introduce the following two concepts.

\begin{definition}
A central configuration is degenerate (nondegenerate) provided that the corresponding critical point $\q$  of $\hat{U}$ is degenerate (nondegenerate). 
\end{definition}

       \subsection{The structure of $I^{-1}(c)$}
Unlike in the Newtonian $N$-body problem, where $I=c>0$ is always a $(3N-1)$-dimensional ellipsoid, the set $I^{-1}(c)$ may not be a smooth manifold. 
To understand the structure of this set, we need the classical Regular Value Theorem, which we further recall for completeness. Let $\mathcal M, \mathcal N$ be differentiable manifolds and $f\colon \mathcal M\to \mathcal N$ a differentiable function. Then $f$ is called a submersion at $x\in \mathcal M$ if its differential, $Df_x\colon T_x\mathcal M\to T_{f(x)}\mathcal N$, is surjective. In this case, $x$ is called a regular point and $f(x)$ a regular value. Otherwise, $x$ is called a critical point and $f(x)$ a critical value. We can now state the following well known result, \cite{Hirsch}.
 
 \medskip
 \noindent $\bf{Regular \ Value\  Theorem.}$
{\it Let $f\colon \mathcal M\to \mathcal N$ be a $C^r$\!-map, $r\ge 1$. If $y\in f(\mathcal M)$ is a regular value, then $f^{-1}(y)$ is a $C^r$\!-submanifold of $\mathcal M$.}
 \medskip
 
If we further regard the moment of inertia as the smooth map
 $$
 I\colon(\M^3)^N\to[0,\infty),
 $$         
we have the following properties.  
 \begin{lemma}
 Assume that the masses $m_1,\dots, m_N$ are in $\S^3$, and consider $c\ge 0$, not of the form $c=\sum_{i=1}^Nm_i\mu_i$, where $\mu_1,\dots, \mu_N\in\{0,1\}$. Then the set
 $I^{-1}(c)$ is a smooth manifold.
 \end{lemma}
 \begin{proof}
 Suppose that $c\ge 0$ is a critical value for $I$. This is equivalent to saying that there exists a $\q=(\q_1,\dots,\q_N)$
 such that $\q\in I^{-1}(c)$ and
 \begin{equation*}
 \nabla_{\q_1}I(\q)=\dots=\nabla_{\q_N}I(\q)=\bf 0,
 \end{equation*} 
 which implies that $\q_i\in \S^1_{xy}\cup \S^1_{zw}$ by Proposition \ref{vanish_of_I} in Section \ref{central-config}. Then $x_i^2+y_i^2=0$ or $1$ and $c=\sum_{i=1}^Nm_i\mu_i$, where $\mu_1,\dots, \mu_N\in\{0,1\}$, a remark that  completes the proof.
 \end{proof}
 \begin{lemma}
 Assume that the masses $m_1,\dots,m_N$ are in $\H^3$, and
 consider $c\ge 0$. Then $I^{-1}(c)$ is always a smooth manifold.
 \end{lemma}
 \begin{proof}
  Suppose that $c\ge 0$ is a critical value for $I$. This is equivalent to saying that there exists a $\q=(\q_1,\dots,\q_N)$
  such that $\q\in I^{-1}(c)$ and
  \begin{equation*}
  \nabla_{\q_1}I(\q)=\dots=\nabla_{\q_N}I(\q)=\bf 0,
  \end{equation*} 
  which implies that $\q_i\in \H^1_{zw}$ by Proposition \ref{vanish_of_I} in Section \ref{central-config}. Then $x_i^2+y_i^2=0$  and $c=0$. Moreover,  $I^{-1}(0)= (\H^1_{zw})^N$, which is homeomorphic with $\R^N$, a remark that  completes the proof.
  \end{proof}

       \subsection{The existence result}
  The characterization of central configurations as critical points provides an easy way to see that ordinary central configurations exist, i.e.\ that the complicated criteria developed earlier always have solutions for $\lambda \ne 0$.
\begin{theorem}\label{existence}
Assume that the  masses $m_1,\dots,m_N$ are in $\S^3$ or $\H^3$. Then for any positive values these masses take, there is at least one ordinary central configuration in $\S^3$ and at least one ordinary central configuration in $\H^3$.
\end{theorem}
\begin{proof}
Let us first prove the result in $\H^3$.
In general, the manifold $I^{-1}(c)$ is not compact in this case. However, this changes if we confine all masses to the hyperbolic circle $\H^1_{xw}$, since the set $I=c>0$ is homeomorphic to an ellipsoid. Then $U$ defines a smooth function on the open subset $S_c$, and the boundary of $S_c$ is composed of points in the singularity set. Since the ellipsoid is compact and $U\to +\infty$ as $\q$ approaches the boundary of $S_c$,  it follows that $U$ attains a minimum at some non-singular configuration $\q$. This will be a  critical point of $U$ on $S_c$ and hence an ordinary central configuration. 

 Let is now prove the result in $\S^3$.
In this case, for the proper value $c$, $S_c$ is a compact manifold since  $(\mathbb{S}^3)^N$ is compact. The problem is that there are two types of singularities, since we can write
$\Delta= \Delta^+ \bigcup \Delta^-,$ where 
$$
\Delta^+=\cup_{1\le i<j\le N}\{\q\in(\S^3)^N \ \! | \ \! \q_i=\q_j \},\ \
\Delta^-=\cup_{1\le i<j\le N}\{\q\in(\S^3)^N \ \! | \ \! \q_i=-\q_j \}.
$$ 
On $\Delta^+$, $U$ is $\infty$; on $\Delta^-$ it is $-\infty$. However,  we can prove the existence of central configurations  by finding a connected component of $S_c$, whose boundary is composed of points that lie only in $\Delta^+$.   
To find such a connected component, we further confine the particles to $\S_{xyz}^2$ and order the masses such that 
$$
0<c<m_1\le\dots\le m_N,
$$
where $c$ is a constant of our choice.
Then $S_c$ is a smooth manifold.
 Let us further choose a configuration $\q\in S_c$ with all bodies lying near the North Pole $(0,0,1)$,  which means that $z_i>0,\ \! i=\overline{1,N}$. Denote by $\mathcal J$ the connected component of the manifold $S_c$ that contains the configuration $\q$. We claim that the boundary $\partial\mathcal J$ of $\mathcal J$ contains only points from $\Delta^+$.

To prove this claim, we first define the sets
$$
\mathcal U=\{(x,y,z)\in\S_{xyz}^2\ \! |\ \! x^2+y^2<c/m_1,\ \! z>0\},
$$
$$  
\mathcal V=\{(x,y,z)\in\S_{xyz}^2\ \! |\ \! x^2+y^2<c/m_1,\ \! z\le 0\}.
$$
Since $I(\q)=\sum_{i=1}^Nm_i(x_i^2+y_i^2)\ge m_1(x_i^2+y_i^2), \ \! i=\overline{1,N},$ it follows that
$$
x_i^2+y_i^2\le c/m_1, \ \! i=\overline{1,N},
$$
which means that for any configuration $\q\in\mathcal J$ each body
lies either in $\mathcal U$ or in $\mathcal V$. 

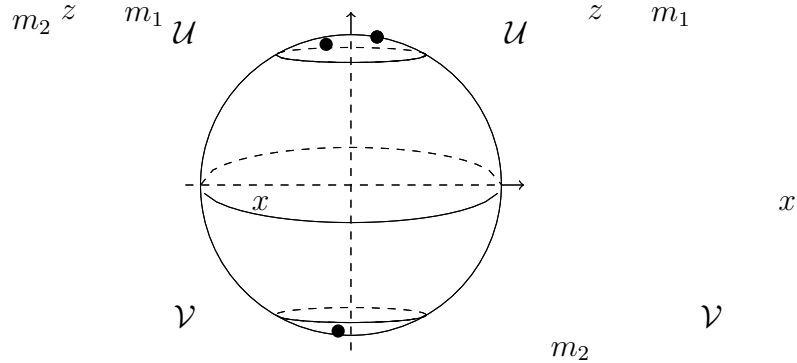
\begin{figure}[!h]
      	\centering
         \begin{tikzpicture}
         \vspace*{0cm}\hspace*{-3.5cm}
         \draw (0,  0) circle (2);
         \draw [dashed] (0, -2.2)--(0, 2);
        \draw [->] (0, 2)--(0, 2.3) node (zaxis) [left] {$z$};
         \draw [dashed] ( -2.2,0)--(2,0);
        \draw [->] (2, 0)--( 2.3,0) node (yaxis) [below] {$x$};
 \draw [dashed] [domain=-2:2] plot (\x, {sqrt(4-\x*\x)/4}); 
                \draw  [domain=-1.95:1.95] plot (\x, {-sqrt(4-\x*\x)/4});

         \draw [dashed] (0,1.73) ellipse (1  and .1); 
                        \draw [dashed] (0,-1.73) ellipse (1  and .1); 
           \draw [domain=-1:1] plot (\x, {1.73-sqrt(1-\x*\x)/10}); 
                       \draw [domain=-1:1] plot (\x, {-1.73-sqrt(1-\x*\x)/10});
       \node[above right] at (1,1.75) {$\mathcal{U}$}; 
                     \node[right] at (1,-1.75) {$\mathcal{V}$};

         \fill (80:2) circle (2.5pt) node[above right] {$m_1$};
          \fill ( 100:1.9) circle (2.5pt) node[ above left] {$m_2$}; 
           
         \vspace*{0cm}\hspace*{7 cm} 
         \draw (0,  0) circle (2);
                 \draw [dashed] (0, -2.2)--(0, 2);
                \draw [->] (0, 2)--(0, 2.3) node (zaxis) [left] {$z$};
                 \draw [dashed] ( -2.2,0)--(2,0);
                \draw [->] (2, 0)--( 2.3,0) node (yaxis) [below] {$x$};
  \draw [dashed] [domain=-2:2] plot (\x, {sqrt(4-\x*\x)/4}); 
                 \draw  [domain=-1.95:1.95] plot (\x, {-sqrt(4-\x*\x)/4}); 
               
                \draw [dashed] (0,1.73) ellipse (1  and .1); 
                \draw [dashed] (0,-1.73) ellipse (1  and .1); 
                \draw [domain=-1:1] plot (\x, {1.73-sqrt(1-\x*\x)/10}); 
              \draw [domain=-1:1] plot (\x, {-1.73-sqrt(1-\x*\x)/10});
                
               \node[above left] at (-1,1.75) {$\mathcal{U}$}; 
               \node[right] at (1,-1.75) {$\mathcal{V}$};

                \fill (80:2) circle (2.5pt)  node[above right] {$m_1$};
                          \fill ( 265:1.95) circle (2.5pt)  node[ below left] {$m_2$};  
        
          \end{tikzpicture}
     \caption{$\q$ and $\tilde\q$ on $\S^2_{xyz}$    }
       \label{fig:exist+}  
       \end{figure}

Let us now suppose that $\partial\mathcal J\cap\Delta^-\ne\emptyset$.  Then there must exist a configuration
$\tilde\q=(\tilde\q_1,\dots,\tilde\q_N)\in\mathcal J$ such that one body is in $\mathcal U$ and the another in $\mathcal V$, say, $\tilde\q_1\in\mathcal U$ and $\tilde\q_2\in\mathcal V$. Since $\mathcal J$ is connected, it is also path connected. Then there is a path in $\mathcal J$ connecting $\q$ and $\tilde{\q}$, so there is a path that connects $\q_2\in\mathcal U$ and $\tilde{\q}_2\in\mathcal V$, and this path must lie in $\mathcal U$ and $\mathcal V$. But this is impossible since $\mathcal U\cap\mathcal V=\emptyset$. Thus we have verified our claim that $\mathcal J$ is a connected component of the compact manifold $S_c$ whose boundary consists only of points from $\Delta^+$. 

Therefore $U\to +\infty$ as $\q$ approaches $\partial\mathcal J$. 
It follows that $U$ attains a minimum at some configuration $\q$, which is then a  critical point of $U$ on $S_c$, hence an ordinary central configuration. This remark completes the proof. 
\end{proof}
 \subsection{The Wintner-Smale conjecture in spaces of constant curvature}

Notice that the proof in the previous subsection also works for other constant values of $I$, namely any $\bar{c}>0$ in $\H^3$ and any $\bar{c}<c$  in $\S^3$, and that these central configurations with different values of $I$ are not equivalent. Hence there always exist central configurations on $S_c$ for $c$ belonging to some open intervals. So we have the following obvious consequence of the above existence result.
\begin{corollary}
Assume that the  masses $m_1,\dots,m_N$ are in $\S^3$ or $\H^3$. Then for any positive values these masses take, the  set of ordinary central configurations  has the power of the continuum. 
\end{corollary}
That is to say, if we extend the Wintner-Smale Conjecture (Smale's 6th problem), which asks whether for some given masses, $m_1,\dots,m_N>0$, the number of classes of planar central configurations for the Newtonian $N$-body problem is finite or not, \cite{Smale}, to the curved $N$-body problem, i.e.\ whether for some given
masses, $m_1,\dots,m_N>0$, the number of classes of central configurations for the curved $N$-body problem is finite or not,
 then this extension has an obvious and uninteresting answer. Also, we have already seen that  any two masses cannot form  a special central configuration, and that three masses can form special central configurations if and only if 
  the  mass triple $(m_1, m_2, m_3)$ belongs to some subset of $(\R^+)^3$, \cite{Diacu15-1}. In light of these facts, we can modify the conjecture as follows. 
 
 \medskip 
  \begin{enumerate}
  \item In the curved $N$-body problem, for given positive masses $m_1,\dots, m_N$ and all possible values of $c$, is the number of ordinary central configurations with $I(\q)=c$  finite?
  \item In the curved $N$-body problem in $\S^3$, for given positive masses $m_1,\dots, m_N$, is the number of special central configurations (if they exist) finite?
  \end{enumerate}
 \medskip
We can also formulate this problem as follows.

  \medskip
     \begin{enumerate}
     \item  For given masses $m_1,$ ..., $m_N$ and all possible $c$, is the number of critical points of $\hat{U}$ on $\hat{S_c}$  finite?
    \item For given positive masses $m_1,$ ..., $m_N$ in $\S^3$, is the number of critical points of $\hat{U}$ on $((\S^3)^N\setminus \D) /SO(4)$ (if they exist) finite?
    \end{enumerate}
   \medskip 

We will see in Section \ref{Moulton} that even for two equal masses, $m_1=m_2=:m$, there are infinitely many critical points of the  function  $\hat{U}$ on $\hat{S}_c$ when $c=m$. So the above nontrivial formulation of the Wintner-Smale conjecture in spaces of constant curvature is not difficult to answer in some particular cases, although the problem remains very difficult in general.

 \section{A property of central configurations}

 Recall that for central configurations of the Newtonian $N$-body problem in $\R^3$, the centre of mass of the configuration is set at the origin, i.e.
 \[ \sum_{i=1}^Nm_ix_i=\sum_{i=1}^Nm_iy_i=\sum_{i=1}^Nm_iz_i=0. \]
In this section we will provide an analogue of this property in $\M^3$. However, we should  keep in mind that, in the curved $N$-body problem, the centre of mass of a geometric configuration,
whatever definition we take for it, has no dynamical significance.
So let us state and prove the following result.
 
 \begin{theorem}
 Let $\q=(\q_1,\dots,\q_N),\ \q_i=(x_i,y_i,z_i,w_i)^T\in\M^3,\ i=\overline{1,N},$ be an ordinary central configuration. Then we have the relationships
 \begin{equation}\label{masscentre}
 \sum_{i=1}^Nm_ix_iz_i=\sum_{i=1}^Nm_ix_iw_i=\sum_{i=1}^Nm_iy_iz_i=\sum_{i=1}^Nm_iy_iw_i=0.
 \end{equation}
 \end{theorem}
\begin{proof}Recall from the first central configuration equation given by system \eqref{CCE}, $\nabla_{\q_i}U=\lambda\nabla_{\q_i}I,\ i=\overline{1,N}$, that the $i$-th equation can be explicitly written as
 \[ \sum_{j=1,j\ne i}^N\frac{m_im_j[\q_j - \csn d_{ij}\q_i ]}{\sn^3 d_{ij} }= 2\lambda m_i 
          \begin{bmatrix}
           x_i\rho_i^2\\
           y_i\rho_i^2\\
          - \sigma z_ir_i^2\\
           - \sigma w_ir_i^2 
           \end{bmatrix}.\]
Taking dot  products with the vectors
\[ \v_1=\begin{bmatrix} z_i\\0\\-x_i\\0\end{bmatrix}, \ \ 
\v_2=\begin{bmatrix} w_i\\0\\0\\-\sigma x_i\end{bmatrix}, \ \ 
\v_3=\begin{bmatrix} 0\\ z_i\\-y_i\\0\end{bmatrix}, \ \ 
\v_3=\begin{bmatrix} 0\\w_i\\0\\-\sigma y_i\end{bmatrix}, \ \  \] 
leads to the relationships in the above statement.  For example, multiplying by $\v_1$, we obtain  
\[\q_j\cdot \v_1=z_ix_j-x_iz_j, \ \ \q_i\cdot \v_1=0, \ \ \nabla_{\q_i}I\cdot \v_1=2m_ix_iz_i(\rho_i^2-\sigma r_i^2)=2m_ix_iz_i,   \]
and consequently
\[ \sum_{j=1,j\ne i}^N\frac{m_im_j}{\sn^3 d_{ij} } (z_ix_j-x_iz_j) = 2\lambda m_ix_iz_i, \ i=\overline{1,N}. \]
Adding the above $N$ equations, we can conclude that
\begin{equation*} 
 \begin{split}
2\lambda\sum_{i=1}^N m_ix_iz_i&= \sum_{i=1}^N\sum_{j=1,j\ne i}^N\frac{m_im_j}{\sn^3 d_{ij} } (z_ix_j-x_iz_j)\\
&= \sum_{1\le i<j\le N}\frac{m_im_j}{\sn^3 d_{ij} } [(z_ix_j-x_iz_j)+(z_jx_i-x_jz_i)]=0.
 \end{split}
 \end{equation*}
Since the central configuration is not a special central configuration, we necessarily have $\lambda \ne 0$, which implies that  $\sum_{i=1}^Nm_ix_iz_i=0$.  The other relationships can be obtained in the same way, a remark that completes the proof.
\end{proof}

An obvious application of equations \eqref{masscentre} is that of showing with little computational effort why certain configurations are not ordinary central configurations. But these relationships can be also used to find ordinary geodesic central configurations. In $\S^3$, for instance, suppose that the geodesic central configuration is on $\S^1_{xz}$,  and let $\q_i=(\sin \th_i, 0, \cos \th_i,0)^T$. Then 
 \[ 2\sum_{i=1}^N m_i x_iz_i =\sum_{i=1}^N m_i \sin 2 \theta_i=0.  \]  \\
 In  $\mathbb{H}^3$, suppose that the geodesic central configuration is on $\H^1_{xw}$,  and let's take $\q_i=(\sinh \th_i, 0, 0,\cosh \th_i)^T$. Then 
  \[ 2\sum_{i=1}^N m_i x_iw_i =\sum_{i=1}^N m_i \sinh 2 \theta_i=0.  \]  
For $N=2$, the above equations help us find the central configurations, as we will show in Section \ref{Moulton}.

   \section{Examples}\label{example}
In this section we will produce some examples of central configurations of the curved $N$-body problem in $\S^3$ and $\H^3$ and discuss the associated relative equilibria. Several examples will concern special and ordinary central configurations for $N=3$ that lie on the great sphere $\S^2_{xyz}$ and the great hyperbolic sphere $\H^2_{xyw}$. It is known that in the Newtonian $N$-body problem there are only two classes of central configurations for $N=3$, the Lagrangian (equilateral triangles) and the Eulerian (collinear configurations). For nonzero constant curvature, however, the set of central configurations (and therefore that of relative equilibria) is richer, as we will further show. We also include in this section examples of central configurations for $N>3$. Unless otherwise stated, the relative equilibria associated to all these central configurations were already found in \cite{Diacu03} and \cite{Diacu05}.

       \subsection{Acute triangle special central configurations on $\S^1_{xy}$}
      Recall that $\S^1_{xy}=\{(x,y,z,w)^T\in \R^4| x^2+y^2=1, z=w=0\}.$
      Let us assume that three masses, $m_1=\frac{\sin^2 \alpha}{\sin^2 \beta}$, $m_2=\frac{\sin^2 \alpha}{\sin^2 (\alpha+\beta)}$, and $m_3=1$ form an acute scalene triangle on $\S^1_{xy}$ and are given by the coordinates
      $$
      \q=(\q_1,\q_2,\q_3), \ \q_i=(x_i,y_i,z_i,w_i)^T,\ i=1,2,3,
      $$
      \begin{align*}
      x_1&=1, & y_1&=0, & z_1&=0, & w_1=0,\displaybreak[0]\\
      x_2&=\cos\alpha, & y_2&=\sin\alpha, & z_2&=0, & w_2=0,\displaybreak[0]\\
      x_3&=\cos(\alpha+\beta), & y_3&=\sin(\alpha+\beta), & z_3&=0, & w_3=0,
      \end{align*}
      for any fixed $0<\alpha<\pi, 0<\beta<\pi, \pi<\alpha+\beta<2\pi$, see Figure \ref{fig:acute}.  Then it is easy to verify that $\nabla_{\q_i} U=0$ for each $i=1,2,3$, and these configurations are special central configurations. 
     \begin{figure}[!h]
           	\centering
             \begin{tikzpicture}
             \draw (0,  0) circle (1.6);
             \draw[->] (0,0) -- (165:1.6);
             \draw[->] (0,0) -- (285:1.6);
             \draw[->] (0:1.6) -- (165:1.6);
             \draw[->] (0:1.6) -- (285:1.6);
             \draw[->] (165:1.6) -- (285:1.6);
             \draw[->] (0:.6)arc(0:165:.6);
             \draw[->] (165:.5)arc(165:285:.5);
             
             \node[above] at(120:.8) {$\alpha$};
             \node[below] at(220:.7) {$\beta$};
             \fill (0:1.6) circle (2.5pt) node[below] {$m_1$};
             \fill (165:1.6) circle (2.5pt) node[left] {$m_2$};
             \fill (285:1.6) circle (2.5pt) node[below] {$m_3$};

              \draw [->] (0, -1.7)--(0, 1.7) node (zaxis) [left] {$y$};
                    \draw [->](-1.7, 0)--(1.7, 0) node (yaxis) [right] {$x$};
             \end{tikzpicture}
             \caption{An acute triangle special central configuration}
             \label{fig:acute}
             \end{figure}
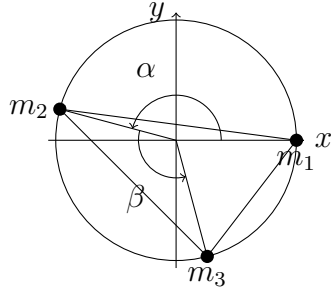
        
Since these special central configurations are confined to $\S^1_{xy}\cup \S^1_{zw}$,  each of them  gives rise to a two-parameter family of associated relative equilibria:  $A_{\a,\b}(t)\q$, $\a,\b \in \R$. The rotation in $zw$-plane does not affect the configuration, so they will be kept on $\S^1_{xy}$, thus forming a one-parameter family of associated relative equilibria,  $A_{\a,0}(t)\q$, $\a\in \R$.  
       \subsection{Regular tetrahedron  special central configurations on $\S_{xyz}^2$}
      Recall that $\S^2_{xyz}=\{(x,y,z,w)^T\in \R^4| x^2+y^2+z^2=1, w=0\}.$
       Let us assume that four masses,  $m_1=m_2=m_3=m_4=:m>0$,  
      form a regular tetrahedron, see Figure \ref{fig:tetra}, 
      $$
      \q=(\q_1,\q_2,\q_3,\q_4), \ \ \q_i=(x_i,y_i,z_i,w_i)^T,\ \ i=1,2,3,4,
      $$
      \begin{align*}
      x_1&=0, & y_1&=0, & z_1&=1, & w_1&=0, \displaybreak[0]\\
      x_2&=0, & y_2&=2\sqrt{2}/3, & z_2&=-1/3, & w_2&=0, \displaybreak[0]\\
      x_3&=-\sqrt{6}/3,& y_3&=-\sqrt{2}/3,& z_3&=-1/3, & w_3&=0,\displaybreak[0]\\
      x_4&=\sqrt{6}/3,& y_4&=-\sqrt{2}/3,& z_4&=-1/3, & w_4&=0.
      \end{align*}
     By symmetry,  it is easy to see that $\sum _{j=1,j\ne i}^4 \F_{ij}=\nabla_{\q_i} U=0$ for $i=1,2,3,4$, and the computations show that this is a special central configuration. 
      \begin{figure}[!h]
      	\centering
         \begin{tikzpicture}
         \draw (0,  0) circle (2);
         \draw [dashed] (0, -2.2)--(0, 2);
        \draw [->] (0, 2)--(0, 2.3) node (zaxis) [left] {$z$};
         \draw [dashed] ( -2.2,0)--(2,0);
        \draw [->] (2, 0)--( 2.3,0) node (yaxis) [below] {$y$};
         \draw [dashed] [domain=-2:2] plot (\x, {sqrt(4-\x*\x)/4}); 
                \draw  [domain=-1.95:1.95] plot (\x, {-sqrt(4-\x*\x)/4});

          \draw[dashed]  (0,-2) to [out=120,in=240] (0,2);
          \draw  (0,-2) to [out=170,in=190] (0,2);
      
       \fill (90:2) circle (2.5pt) node[above right] {$m_1$};   
         \fill (330:2) circle (2.5pt) node[right] {$m_2$};
          \fill ( -0.45,-0.9) circle (2.5pt) node[ right] {$m_3$}; 
           \fill ( -1,-1) circle (2.5pt) node[below left] {$m_4$};  
          
          \end{tikzpicture}
     \caption{Regular tetrahedron  special central configuration    }
       \label{fig:tetra}  
       \end{figure}
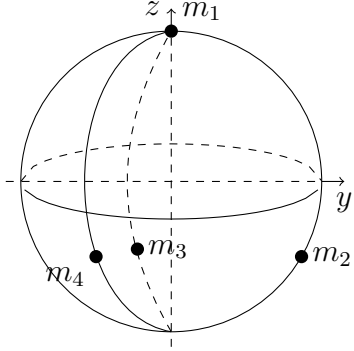

  Since this special central configuration is not confined to $\S^1_{xy}\cup \S^1_{zw}$, it gives rise to a one-parameter family of associated relative equilibria, $A_{\a,\pm \a}(t)\q$, $\a \in \R$. They are periodic orbits, but the motion is not confined to $\S^2_{xyz}$.  
          \subsection{Regular pentatope special central configurations in $\S^3$}
         Let us assume that five masses,  $m_1=m_2=m_3=m_4=m_5=:m>0$  
         form a regular pentatope in  $\S^3$, with
         $$
         \q=(\q_1,\q_2,\q_3,\q_4,\q_5), \ \ \q_i=(x_i,y_i,z_i,w_i)^T,\ \ i=1,2,3,4,5,
         $$
         \begin{align*}
         x_1&=1,& y_1&=0,& z_1&=0,& w_1&=0, \displaybreak[0]\\
          x_2&=-1/4,& y_2&=\sqrt{15}/4,& z_2&=0,& w_2&=0, \displaybreak[0]\\
          x_3&=-1/4,& y_3&=-\sqrt{5}/(4\sqrt{3}),& z_3&=\sqrt{5}/\sqrt{6},& w_3&=0,\displaybreak[0]\\
         x_4&=-1/4,& y_4&=-\sqrt{5}/(4\sqrt{3}),& z_4&=-\sqrt{5}/(2\sqrt{6}),& w_4&=\sqrt{5}/(2\sqrt{2}),  \displaybreak[0]\\
          x_5&=-1/4,& y_5&=-\sqrt{5}/(4\sqrt{3}),& z_5&=-\sqrt{5}/(2\sqrt{6}),& w_5&=-\sqrt{5}/(2\sqrt{2}).
         \end{align*}
         By symmetry,  it is easy to see that $\sum _{j=1,j\ne i}^4 \F_{ij}=\nabla_{\q_i} U=0$ for $i=1,2,3,4,5$, and computations show that this is a special central configuration. Since this special central configuration is not confined to $\S^1_{xy}\cup \S^1_{zw}$, it gives rise to a one-parameter family of associated relative equilibria, $A_{\a,\pm \a}(t)\q$, $\a \in \R$, which are periodic orbits. 
         
          \subsection{Pair of equilateral triangle special central configuration in $\S^3$}
         Let us assume that six masses,  $m_1=m_2=m_3=m_4=m_5=m_6=:m>0$ in 
         $\S^3$ 
         form two equilateral triangles on complementary great circles:
         $\S^1_{xy}$ and $\S^1_{zw}$, with
         $$
         \q=(\q_1,\q_2,\q_3,\q_4,\q_5,\q_6), \ \ \q_i=(x_i,y_i,z_i,w_i)^T,\ \ i=1,2,3,4,5,6,
         $$
         \begin{align*}
         x_1&=0,& y_1&=1,& z_1&=0, & w_1&=0,\displaybreak[0]\\
         x_2&=\sqrt{3}/2,& y_2&=-1/2,& z_2&=0, & w_2&=0,\displaybreak[0]\\
         x_3&=-\sqrt{3}/2,& y_3&=-1/2,& z_3&=0, & w_3&=0, \displaybreak[0]\\
         x_4&=0,& y_4&=0,& z_4&=0, & w_4&=1, \displaybreak[0]\\
         x_5&=0,& y_5&=0,& z_5&=\sqrt{3}/2, & w_5&=-1/2, \displaybreak[0]\\
         x_6&=0,& y_6&=0,& z_6&=-\sqrt{3}/2, & w_6&=-1/2.
         \end{align*}
        To see that $\sum _{j=1,j\ne i}^6 \F_{ij}=\nabla_{\q_i} U=0$ for $i=1,2,3,4,5,6$, it suffices to check that for $m_1$. That is,
        \[ \F_1=\F_{12}+\F_{13}+\F_{14}+\F_{15}+\F_{16}=0.\]
By symmetry we obtain that $ \F_{12} +\F_{13}=0$. For $i=4,5,6$,  since $\q_1\cdot \q_i =0$, we have $d_{1i}= \pi/2$ and    $\F_{1i}=\frac{m^2 (\q_i-\cos d_{1i}\q_1)}{\sin ^3 d_{1i}} =m^2 \q_i$. 
        Then 
        \[ \F_{14}+\F_{15}+\F_{16}= m^2(\q_4+\q_5+\q_6)=0, \]
        hence this  is a  special central configuration. 
        Since this special central configurations is on $\S^1_{xy}\cup \S^1_{zw}$,  it gives rise to a two-parameter family of associated relative equilibria,  $A_{\a,\b}(t)\q$, $\a,\b \in \R$. They are periodic orbits if $\a/\b$ is rational, but quasi-periodic orbits if $\a/\b$ is irrational.
         
   \subsection{Lagrangian central configurations in $\S_{xyz}^2$}
  This is the example we presented in Section \ref{count}. Let us assume that three equal masses, $m_1=m_2=m_3=:m>0$,  form an equilateral configuration on $\S_{xyz}^2$, parallel with the $xy$-plane, so the coordinates are given by
   $$
   \q=(\q_1,\q_2,\q_3),\ \  \q_i=(x_i,y_i,z_i,w_i)^T,\ \ i=1,2,3,
   $$
   \begin{align*}
   x_1&= r, & y_1&= 0, & z_1&= z, & w_1&=0,\displaybreak[0]\\
   x_2&= -r/2, & y_2&= r\sqrt{3}/2, & z_2&= z, & w_2&=0,\displaybreak[0]\\
   x_3&= -r/2, & y_3&= -r\sqrt{3}/2, & z_3&= z, & w_3&=0,\displaybreak[0]
   \end{align*}
  where $r^2+z^2=1$, $r\in (-1,1)$. By symmetry, we notice that $\F_i$ is pointing towards the North or South poles and that $|\F_i|=|\F_j|$. Comparing this with the vector field $\nabla (x^2+y^2)$ on $\S_{xyz}^2$  (see Figure \ref{fig:nablaI1}), we see that the central configuration equation $\nabla_{\q_i} U=\lambda\nabla_{\q_i} I$ is satisfied for $i=1,2,3$. 
  
   To find the value of $\lambda$, we use equation \eqref{lambda+}. 
   Since the configuration is equilateral, we obtain that 
   $$
   d_{ij}=d_{jk}=:d, \ \ \cos d=1-\frac{3r^2}{2},\ \ \sin^3 d=3\sqrt{3}r^3\left(1-\frac{3r^2}{4}\right)^{3/2},
   $$
   $$
   x_1x_2+y_1y_2=x_1x_3+y_1y_3=x_2x_3+y_2y_3=-r^2/2.
   $$
   Then equation \eqref{lambda+} yields
   \begin{equation*} 
    \begin{split}
     \lambda &=  \sum_{1\le i\le N} \sum_{j=1,j\ne i}^N\frac{m_im_j(2x_ix_j+2y_iy_j-(r_i^2+r_j^2)\cos d_{ij})}{\sin^3d_{ij}}/\left( 2 \sum_{1\le i\le N} m_i r_i^2 \rho_i^2 \right)  \\
     &= \frac{3m^2(-r^2-2r^2 \cos d)}{\sin^3d \cdot 6mr^2z^2}=\frac{-3m^2r^2(3-3r^2 )}{\sin^3d \cdot 6mr^2z^2}\\
     &=\frac{3m}{2\sin^3d}= -\frac{m}{2\sqrt{3}r^3\left(1-\frac{3r^2}{4}\right)^{3/2}}<0.
    \end{split}
    \end{equation*} 
   For $r=1$, we necessarily have $z=0$, i.e.\ the central configuration is on $\S^1_{xy}$, the special central configuration discussed in the first example. 
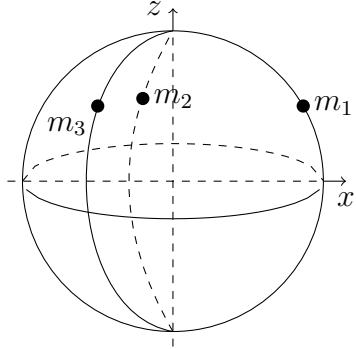
\begin{figure}[!h]
      	\centering
         \begin{tikzpicture}
         \draw (0,  0) circle (2);
         \draw [dashed] (0, -2.2)--(0, 2);
        \draw [->] (0, 2)--(0, 2.3) node (zaxis) [left] {$z$};
         \draw [dashed] ( -2.2,0)--(2,0);
        \draw [->] (2, 0)--( 2.3,0) node (yaxis) [below] {$x$};
         \draw [dashed] [domain=-2:2] plot (\x, {sqrt(4-\x*\x)/4}); 
                \draw  [domain=-1.95:1.95] plot (\x, {-sqrt(4-\x*\x)/4}); 
        
          \draw[dashed]  (0,-2) to [out=120,in=240] (0,2);
          \draw  (0,-2) to [out=170,in=190] (0,2);

         \fill (30:2) circle (2.5pt) node[right] {$m_1$};
          \fill ( -0.4,1.1) circle (2.5pt) node[ right] {$m_2$}; 
           \fill ( -1,1) circle (2.5pt) node[below left] {$m_3$};  
          
          \end{tikzpicture}
     \caption{Lagrangian  central configurations on $\S^2_{xyz}$    }
       \label{fig:lag+2}  
       \end{figure}

Each of these central configurations gives rise to a one-parameter family of associated relative equilibria, $A_{\a, \b}(t)\q$ with $\lambda=\frac{\b^2-\a^2}{2}$. 

   \subsection{Geodesic central configurations on $\S^1_{xz}$}
   Recall that we earlier defined  $\S^1_{xz}=\{(x,y,z,w)^T\in \R^4| x^2+z^2=1, y=w=0\}.$
  Let the coordinates of the three  bodies of masses $m_1=m_2=m_3=:m>0$ be given by
   $$
   \q=(\q_1,\q_2,\q_3),\ \  \q_i=(x_i,y_i,z_i,w_i)^T,\ \ i=1,2,3,
   $$
   \begin{align*}
   x_1&= 0, & y_1&= 0, & z_1&= 1, & w_1&=0,\displaybreak[0]\\
   x_2&= r, & y_2&= 0, & z_2&= z, & w_2&=0,\displaybreak[0]\\
   x_3&= -r, & y_3&= 0, & z_3&= z, & w_3&=0,\displaybreak[0]
   \end{align*}
with $r>0$, $z\in(-1,0)\cup(0,1)$ and $r^2+z^2=1$. Given the many zeroes that occur in the above coordinates, it is not difficult to check that the nine equations of Criterion \eqref{criterion++} are satisfied. We can also check the existence of this central configuration by using an argument similar to the one employed in the earlier examples, that is,  $\F_i$ and $\nabla_{\q_i}I$ are collinear and the ratios are the same for each $i=1,2,3$. By symmetry, $\F_1=0$ and $|\F_2|=|\F_3|$.  Comparing with the vector field $\nabla (x^2+y^2)$ on $\S^2_{xyz}$  (see Figure \ref{fig:nablaI1}), we see that the central configuration equation $\nabla_{\q_i} U=\lambda\nabla_{\q_i} I$ is satisfied for $i=1,2,3$.
    
To find the value of $\lambda$, we use equation \eqref{lambda+}. We have 
  $$
     d_{12}=d_{23},  \ \ r_1^2=0,\ \  r_2^2=r_3^2=r^2,\ \ 
     $$
     $$
     x_1x_2+y_1y_2=x_1x_3+y_1y_3=0, \ \ x_2x_3+y_2y_3=-r^2.
     $$
  $$\cos d_{12}=z,\ \ \sin^3 d_{12}=r^3,\ \ \cos d_{23}=z^2-r^2,\ \ \sin^3 d_{23}=8r^3|z|^3.\ \
       $$
 Then equation \eqref{lambda+} yields
   \begin{equation*} 
    \begin{split}
     \lambda
     &= \frac{1}{4mr^2z^2}\left(\frac{m^2(-r^2 \cos d_{12})}{\sin^3d_{12}} +\frac{m^2(-r^2 \cos d_{13})}{\sin^3d_{13}} +\frac{m^2(-2r^2-2r^2 \cos d_{23})}{\sin^3d_{23}}\right) \\
     &=\frac{-m}{2z^2}\left(\frac{\cos d_{12}}{\sin^3d_{12}} +\frac{1+\cos d_{23}}{\sin^3d_{23}}\right)=\frac{-m}{2r^3}\left(\frac{1}{z} +\frac{1}{4|z|^3}\right).
    \end{split}
    \end{equation*} 
It is easy to see that $\lambda <0$ for $z\in ( -1/2,0)\cup (0,1)$,  $\lambda >0$ for  $z\in (-1,-1/2)$,  and $\lambda =0$ for  $z= -1/2$, which shows the connection with the special central configuration discussed in the first example. 
   
   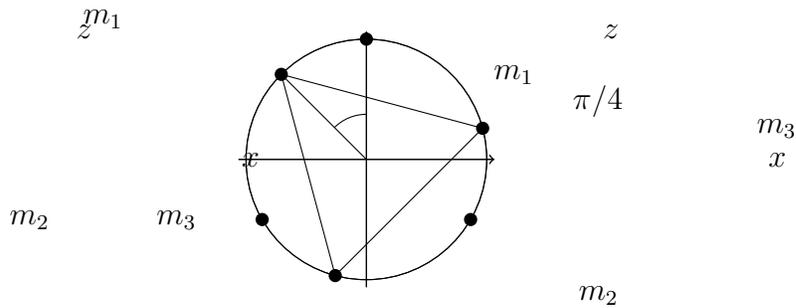
\begin{figure}[!h]
         	\centering
           \begin{tikzpicture}
            \vspace*{0cm}\hspace*{-3.5cm}
           \draw (0,  0) circle (1.6);
           
\fill (90:1.6) circle (2.5pt) node[above] {$m_1$};          
\fill (330:1.6) circle (2.5pt) node [left]   {$m_3$};
\fill (210:1.6) circle (2.5pt) node[right] {$m_2$};

            \draw [->] (0, -1.7)--(0, 1.7) node (zaxis) [left] {$z$};
                  \draw [->](-1.7, 0)--(1.7, 0) node (yaxis) [right] {$x$};
                  \vspace*{0cm}\hspace*{7cm}
       \draw (0,  0) circle (1.6);
                  
                  \draw[-] (135:1.6) -- (255:1.6);
                  \draw[-] (255:1.6) -- (15:1.6);
                  \draw[-] (15:1.6) -- (135:1.6);
                  \draw[-] (0,0) -- (135:1.6);
                  \draw[-] (90:.6)arc(90:135:.6) node [above]   {$\pi/4$};

 \fill (135:1.6) circle (2.5pt) node[left] {$m_1$};          
 \fill (255:1.6) circle (2.5pt) node [below]   {$m_2$};
 \fill (15:1.6) circle (2.5pt) node[right] {$m_3$};
                  
                   \draw [->] (0, -1.7)--(0, 1.7) node (zaxis) [left] {$z$};
                         \draw [->](-1.7, 0)--(1.7, 0) node (yaxis) [right] {$x$};

           \end{tikzpicture}
           \caption{Geodesic central configurations on $\S^1_{xz}$}
           \label{fig:euler+}
           \end{figure}
 
 All ordinary geodesic central configurations of three masses on $\S^1_{xz}$ were found in \cite{Zhu}. Some  interesting  examples were given there, such as the one in which three distinct masses form an equilateral triangle. For instance, take the masses
   $m_1=2, m_2=1, m_3=3$ on $\S^1_{xz}$ with configuration
   $$
   \q=(\q_1,\q_2,\q_3),\ \  \q_i=(x_i,y_i,z_i,w_i)^T,\ \ i=1,2,3,
   $$
   \begin{align*}
   x_1&= -\sin\frac{\pi}{4}, & y_1&= 0, & z_1&=\cos\frac{\pi}{4}, & w_1&=0,\displaybreak[0]\\
   x_2&=-\sin(11\pi/12), & y_2&= 0, & z_2&=\cos(11\pi/12), & w_2&=0,\displaybreak[0]\\
   x_3&=-\sin(19\pi/12), & y_3&= 0, & z_3&=\cos(19\pi/12), & w_3&=0.
   \end{align*}
 We could also verify that the equations of Criterion \ref{criterion++} are satisfied and 
   $$
   \lambda=-\frac{8}{3}.$$
We can actually find many such examples. For any three masses, if there are $\lambda\ne 0$ and $\theta$ such that
   $$
   \sin 2\theta= -\frac{4}{3\lambda}(m_3-m_2), \ \  \cos 2\theta= \frac{4\sqrt{3}}{9\lambda} (2m_1-m_3-m_2),
   $$
   then the configuration
   $$
   \q=(\q_1,\q_2,\q_3),\ \  \q_i=(x_i,y_i,z_i,w_i)^T,\ \ i=1,2,3,
   $$
   \begin{align*}
   x_1&= -\sin\theta, & y_1&= 0, & z_1&=\cos\theta, & w_1&=0,\displaybreak[0]\\
   x_2&=-\sin(\theta+2\pi/3), & y_2&= 0, & z_2&=\cos(\theta+2\pi/3), & w_2&=0,\displaybreak[0]\\
   x_3&=-\sin(\theta+4\pi/3), & y_3&= 0, & z_3&=\cos(\theta+4\pi/3), & w_3&=0,
   \end{align*}
   is always a central configuration. 
   
      \subsection{Isosceles central configuration in $\S^2_{xyz}$}
      Let us assume that three masses, $m_1=-2\cos\varphi$, with $\varphi\in(\pi/2,\pi)$, $m_2=m_3=1$,  form an isosceles triangle on the sphere $\S^2_{xyz}$, parallel with the $xy$-plane,  and are given by the coordinates 
      $$
      \q=(\q_1,\q_2,\q_3), \ \q_i=(x_i,y_i,z_i,w_i)^T,\ i=1,2,3,
      $$
      \begin{align*}
      x_1&=\sin\theta, & y_1&=0, & z_1&=\cos\theta, & w_1=0,\displaybreak[0]\\
      x_2&=\sin\theta\cos\varphi, & y_2&=\sin\theta\sin\varphi, & z_2&=\cos\theta, & w_2=0,\displaybreak[0]\\
      x_3&=\sin\theta\cos\varphi, & y_3&=-\sin\theta\sin\varphi, & z_3&=\cos\theta, & w_3=0,
      \end{align*}
      with $\theta$ chosen such that
      $$
      \cos^2\theta=1+\frac{2}{(\cos\varphi-1)(2\cos\varphi+3)}.
      $$
      
       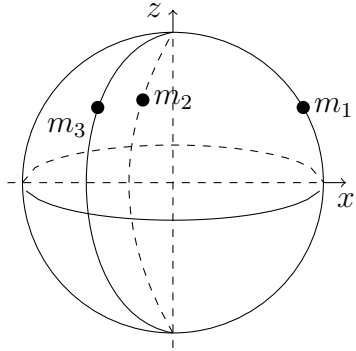
\begin{figure}[!h]
                  	\centering
                     \begin{tikzpicture}
                     \draw (0,  0) circle (2);
                     \draw [dashed] (0, -2.2)--(0, 2);
                    \draw [->] (0, 2)--(0, 2.3) node (zaxis) [left] {$z$};
                     \draw [dashed] ( -2.2,0)--(2,0);
                    \draw [->] (2, 0)--( 2.3,0) node (yaxis) [below] {$x$};
      \draw [dashed] [domain=-2:2] plot (\x, {sqrt(4-\x*\x)/4}); 
                            \draw  [domain=-1.95:1.95] plot (\x, {-sqrt(4-\x*\x)/4});

                      \draw[dashed]  (0,-2) to [out=120,in=240] (0,2);
                      \draw  (0,-2) to [out=170,in=190] (0,2);

                     \fill (30:2) circle (2.5pt) node[right] {$m_1$};
                      \fill ( -0.4,1.1) circle (2.5pt) node[ right] {$m_2$}; 
                       \fill ( -1,1) circle (2.5pt) node[below left] {$m_3$};  
                      
                      \end{tikzpicture}
                 \caption{Isosceles central configuration on $\S^2_{xyz}$    }
                   \label{fig:iso}  
                   \end{figure}
                   
       By straightforward computations, we can  see that the equations of Criterion \ref{criterion++} are  satisfied, and with equation \eqref{lambda+} we obtain 
      $$
      \lambda=-\frac{2-2\cos \varphi}{2\sin^3d_{12}}=-\frac{2-2\cos \varphi}{2\sin^3\th (1-\cos \varphi)^{3/2}(1+\sin^2\th \cos \varphi +\cos ^2\th)^{3/2}}.
      $$
      The existence of the associated relative equilibria was proved in \cite{Diacu78}. Interesting details concerning this type of central configuration will be given in a future paper.

   \subsection{Lagrangian central configurations in $\H_{xyw}^2$}
   Recall that we earlier defined
   $ \H^2_{xyw}=\{(x,y,z,w)^T\in \R^4| x^2+y^2-w^2=-1, z=0\}.$
    Let us assume that three equal masses, $m_1=m_2=m_3=:m>0$,  form an equilateral configuration in $\H^2_{xyw}$, parallel with the $xy$-plane, and the coordinates are given by
      $$
      \q=(\q_1,\q_2,\q_3),\ \  \q_i=(x_i,y_i,z_i,w_i)^T,\ \ i=1,2,3,
      $$
      \begin{align*}
        x_1&= r, & y_1&= 0, & z_1&=0, & w_1&= w,\displaybreak[0]\\
        x_2&= -r/2, & y_2&= r\sqrt{3}/2, & z_2&=0, & w_2&= w,\displaybreak[0]\\
        x_3&= -r/2, & y_3&= -r\sqrt{3}/2, & z_3&=0, & w_3&= w,\displaybreak[0]
        \end{align*}
     where $r^2-w^2=-1$, $w\in (1,+\infty)$. By symmetry, we notice that $\F_i$ is pointing towards $(0,0,0,1)$ and $|\F_i|=|\F_j|$. Comparing with the vector field $\nabla (x^2+y^2)$ on $\H^2_{xyw}$   (see Figure \ref{fig:nablaI2}), we see that the central configuration equation $\nabla_{\q_i} U=\lambda\nabla_{\q_i} I$ is satisfied for $i=1,2,3$.
     
To compute the value of $\lambda$, we use equation \eqref{lambda-}. 
      Since the configuration is equilateral, we obtain that 
      $$
      d_{ij}=d_{jk}=:d, \ \ \cosh d=1+\frac{3r^2}{2},\ \ \sinh^3 d=3\sqrt{3}r^3\left(1+\frac{3r^2}{4}\right)^{3/2},
      $$
      $$
      x_1x_2+y_1y_2=x_1x_3+y_1y_3=x_2x_3+y_2y_3=-r^2/2.
      $$
      Then equation \eqref{lambda-} yields
      \begin{equation*} 
       \begin{split}
        \lambda &=  \sum_{1\le i\le N} \sum_{j=1,j\ne i}^N\frac{m_im_j(2x_ix_j+2y_iy_j-(r_i^2+r_j^2)\cosh d_{ij})}{\sinh^3d_{ij}}/\left( 2 \sum_{1\le i\le N} m_i r_i^2 \rho_i^2 \right)  \\
        &= \frac{3m^2(-r^2-2r^2 \cosh d)}{\sinh^3d \cdot 6mr^2w^2}=\frac{-3m^2r^2(3+3r^2 )}{\sinh^3d \cdot 6mr^2w^2}\\
        &=\frac{3m}{2\sinh^3d}= -\frac{m}{2\sqrt{3}r^3\left(1+\frac{3r^2}{4}\right)^{3/2}}.
       \end{split}
       \end{equation*} 
Each of these central configurations gives rise to one-parameter family of associated relative equilibria: $B_{\a, \b}(t)\q$ with $\lambda=-\frac{\b^2+\a^2}{2}$. These orbits are a new discovery that has been missed in previous studies, a fact that shows the power of the central-configuration method for finding relative equilibria.  

   Although we build the whole theory of negative-curvature spaces on the hyperbolic-sphere model $\H^3$, it is convenient to visualize the associated  relative equilibria in the Poincar\'{e} ball model. 
    Recall that the Poincar\'e ball model is given by
 \[ \left(  \bar{x}^2+\bar{y}^2+\bar{z}^2<1, \ ds^2=\frac{4(d\bar{x}^2+d\bar{y}^2+d\bar{z}^2)}{1-(\bar{x}^2+\bar{y}^2+\bar{z}^2)}  \right),  \] 
 which can be seen as the perspective projection of the upper 3-dimensional hyperboloid viewed from $(0,0,0,-1)$. The projection mapping is 
 \[ \bar{x}=\frac{x}{1+w}, \ \ \bar{y}=\frac{y}{1+w}, \ \ \bar{z}=\frac{z}{1+w}.  \ \ \] 
This projection mapping shows that the isometries  of the  $SO(2)$  rotations in the $xy$-plane  become the  rotations in the $\bar{x}\bar{y}$-plane, and that the isometries of the  $SO(1,1)$  rotations in the $zw$-plane  become action moving points from $(0,0,-1)$ to $(0,0,1)$ or in the opposite direction.  
Thus the relative equilibria $B_{\a,\b}(t)\q$ in the Poincar\'e ball model can be viewed as bodies that rotate around the $\bar{z}$-axis and move up or down along the projection of the hyperbolic cylinder
   $$
   {\bf C}_{r\rho}:=\{(x,y,z,w)^T\in \H^3\ \! | \ \! x^2+y^2=r^2 \},
   $$
a spindle-shaped surface (within the framework of this model) for which the hyperbolic distance from the $\bar z$-axis is constant (see Figure \ref{fig:Lag-} on the right),
hence the name ``hyperbolic cylinder'' we gave to it in previous studies on relative equilibria, \cite{Diacu03}, \cite{Diacu05}. 
    
   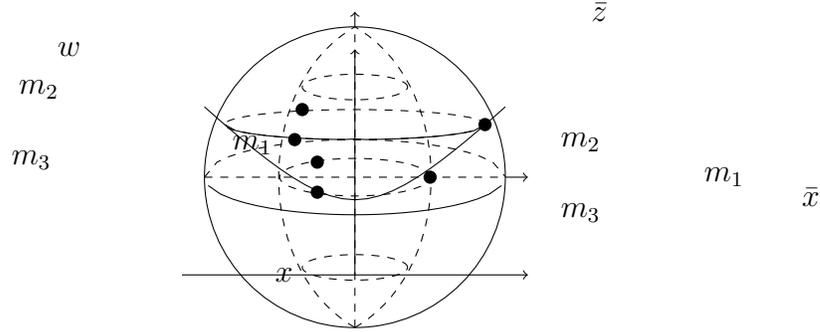
\begin{figure}
                 \centering
                \begin{tikzpicture}
                \vspace*{0cm}\hspace*{-3.5cm}
                \draw[black, domain=-2:2] plot (\x, {sqrt(\x*\x+1) });
                 \draw [dashed] (0,2) ellipse (1.73  and .2);  
                  \draw [domain=-1.73:1.73] plot (\x, {2-0.2*sqrt(1.73*1.73-\x*\x)/1.73});  
                
                \fill (-0.7,2.2) circle (2.5pt) node[above] {$m_2$};
                 \fill (-0.8,1.8) circle (2.5pt) node[below] {$m_3$};
               \fill (1.73,2) circle (2.5pt) node[below right] {$m_1$};

                 \draw [->] (0, 0)--(0, 3) node (zaxis) [left] {$w$};
                       \draw [->](-2.3, 0)--(2.3, 0) node (yaxis) [right] {$x$};
    \vspace*{0cm}\hspace*{7cm}
                    
     \draw (0,  1.3) circle (2);
        \draw [dashed] (0, -0.7)--(0, 3.3);
    \draw [dashed] (-2, 1.3)--(2, 1.3); 
       \draw [->] (0, 3.3)--(0, 3.5) node (zaxis) [left] {$\bar{z}$};
         \draw [->] (2, 1.3)--(2.3, 1.3) node (zaxis) [below right] {$\bar{x}$};
         
     \draw [dashed][domain=-2:2] plot (\x, {1.3+sqrt(4-\x*\x)/4});   
     \draw [domain=-1.95:1.95] plot (\x, {1.3-sqrt(4-\x*\x)/4}); 
    \draw [dashed] (0,1.3) ellipse (1  and .25);

      \draw[dashed]  (0,-0.7) to [out=150,in=210] (0,3.3);                         
     \draw[dashed]  (0,-0.7) to [out=30,in=330] (0,3.3);                            
      \draw [dashed] (0,2.5) ellipse (0.7  and .17);
       \draw [dashed] (0,0.1) ellipse (0.7  and .17);                           
    
     \fill (1,1.3) circle (2.5pt) node[right] {$m_1$};
      \fill (-0.5,1.5) circle (2.5pt) node[ above] {$m_2$}; 
       \fill (-0.5,1.1) circle (2.5pt) node[below] {$m_3$};                    
                \end{tikzpicture}
                \caption{ Lagrangian central configurations on $\H^2_{xyw}$ and the associated relative equilibria in the Poincar\'e ball.} \label{fig:Lag-}
                  \end{figure}

   \subsection{Geodesic central configurations in $\H^1_{xw}$}
    Recall that we earlier defined
    $ \H^1_{xw}=\{(x,y,z,w)^T\in \R^4| x^2-w^2=-1, y=z=0\}.$
     Let three  bodies of masses $m_1=m_2=m_3=:m>0$ have the coordinates
      $$
      \q=(\q_1,\q_2,\q_3),\ \  \q_i=(x_i,y_i,z_i,w_i)^T,\ \ i=1,2,3,
      $$
      \begin{align*}
      x_1&= 0, & y_1&= 0, & z_1&= 0, & w_1&=1,\displaybreak[0]\\
      x_2&= r, & y_2&= 0, & z_2&= 0, & w_2&=w,\displaybreak[0]\\
      x_3&= -r, & y_3&= 0, & z_3&= 0, & w_3&=w,\displaybreak[0]
      \end{align*}
      with $r>0$ and $r^2-w^2=-1$. Given the many zeroes that occur above, it is not difficult to check that system \eqref{criterion--} is satisfied.  We can also check the existence of this central configuration by using the argument employed in the earlier examples, that is,  $\F_i$ and $\nabla_{\q_i}I$ are collinear and the ratios are the same for each $i=1,2,3$. By symmetry, $\F_1=0$ and $|\F_2|=|\F_3|$.  Comparing with the vector field $\nabla (x^2+y^2)$ on $\H^2_{xyw}$  (see Figure \ref{fig:nablaI2}), we see that the central configuration equation $\nabla_{\q_i} U=\lambda\nabla_{\q_i} I$ is satisfied for $i=1,2,3$.
       
To compute the value of $\lambda$, we use equation \eqref{lambda-} and the relationships 
     $$
        d_{12}=d_{23},  \ \ r_1^2=0,\ \  r_2^2=r_3^2=r^2,\ \ 
        $$
        $$
        x_1x_2+y_1y_2=x_1x_3+y_1y_3=0, \ \ x_2x_3+y_2y_3=-r^2.
        $$
     $$\cosh d_{12}=w,\ \ \sin^3 d_{12}=r^3,\ \ \cosh d_{23}=w^2+r^2,\ \ \sinh ^3 d_{23}=8r^3w^3,\ \
          $$
which yield
      \begin{equation*} 
       \begin{split}
        \lambda
        &= \frac{1}{4mr^2z^2}\left[\frac{m^2(-r^2 \cosh d_{12})}{\sinh ^3d_{12}} +\frac{m^2(-r^2 \cosh d_{13})}{\sinh ^3d_{13}} +\frac{m^2(-2r^2-2r^2 \cosh d_{23})}{\sinh ^3d_{23}}\right] \\
        &=-\frac{m}{2w^2}\left(\frac{\cosh d_{12}}{\sinh ^3d_{12}} +\frac{1+\cosh d_{23}}{\sinh ^3d_{23}}\right)=-\frac{m}{2r^3}\left(\frac{1}{w} +\frac{1}{4w^3}\right).
       \end{split}
       \end{equation*} 
    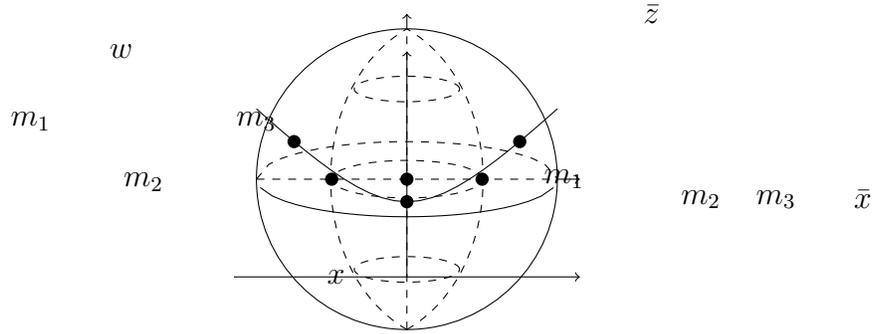
\begin{figure}
         \centering
        \begin{tikzpicture}
        
         \vspace*{0cm}\hspace*{-3.5cm}
        \draw[black, domain=-2:2] plot (\x, {sqrt(\x*\x+1) });

        \fill (1.5,1.8) circle (2.5pt) node[above] {$m_3$};
         \fill (0,1) circle (2.5pt) node[above] {$m_2$};
       \fill (-1.5,1.8) circle (2.5pt) node[above] {$m_1$};

         \draw [->] (0, 0)--(0, 3) node (zaxis) [left] {$w$};
               \draw [->](-2.3, 0)--(2.3, 0) node (yaxis) [right] {$x$};
                \vspace*{0cm}\hspace*{7cm}
                
 \draw (0,  1.3) circle (2);
    \draw [dashed] (0, -0.7)--(0, 3.3);
\draw [dashed] (-2, 1.3)--(2, 1.3); 
 \draw [->] (0, 3.3)--(0, 3.5) node (zaxis) [left] {$\bar{z}$};
          \draw [->] (2, 1.3)--(2.3, 1.3) node (zaxis) [below right] {$\bar{x}$};
            
 \draw [dashed][domain=-2:2] plot (\x, {1.3+sqrt(4-\x*\x)/4});   
      \draw [domain=-1.95:1.95] plot (\x, {1.3-sqrt(4-\x*\x)/4}); 
     \draw [dashed] (0,1.3) ellipse (1  and .25); 

\draw [dashed] (0,2.5) ellipse (0.7  and .17);
\draw [dashed] (0,0.1) ellipse (0.7  and .17);
  \draw[dashed]  (0,-0.7) to [out=150,in=210] (0,3.3);                         
 \draw[dashed]  (0,-0.7) to [out=30,in=330] (0,3.3);

 \fill (-1,1.3) circle (2.5pt) node[left] {$m_1$};
  \fill (0,1.3) circle (2.5pt) node[ below right] {$m_2$}; 
   \fill (1,1.3) circle (2.5pt) node[below right] {$m_3$};        
 \end{tikzpicture}
        \caption{ Geodesic central configurations on $\H^1_{xw}$ and the associated relative equilibria in the Poincar\'e ball} \label{fig:euler-}
          \end{figure} 
As in the last example, we can also represent the associated relative equilibria in the Poincar\'e ball model, see Figure \ref{fig:euler-}, where the bodies rotate around the $\bar{z}$-axis and  move up or down, one along the $\bar z$-axis, and the other two along the projection of the hyperbolic cylinder ${\bf C}_{r\rho}$, thus maintaining constant mutual distances.

   \section{Moulton's theorem}  \label{Moulton}    
 In 1910, Forest Ray Moulton sought to extend Euler's results about the collinear central configurations in the Newtonian $N$-body problem to any number $N$ of point masses. He showed that for a given ordering of the bodies on a straight line, there is exactly one class of central configurations, \cite{Moulton}. In this section we are asking whether Moulton's theorem has a natural correspondent in spaces of nonzero constant curvature. As we will further prove, this extension is true on geodesics of $\H^3$, but not on geodesics of $\S^3$, where even the case $N=2$ leads to a complicated count. 
 
Before we get to the curved $N$-body  problem, let us make some comments about the Euclidean case. The class of central configurations in the above statement of the theorem is meant as the set of central configurations factorized to homotheties. So another equivalent way of stating Moulton's result is to say that, for every ordering of any given masses  with $I(\q)=$ constant, there is exactly one central configuration. This new formulation is the one we adopt here, since the value of $I(\q)$ could never be the same for central configurations with different sizes,
as Definition \ref{equi-central configuration} implies.

 \subsection{Geodesic  central configurations in  $\H^3$}

Theorem \ref{1d_CC} states that every geodesic central configuration in $\H^3$ is equivalent to some geodesic central configuration on $\H^1_{xw}$. 
 Thus we  assume that the point masses $m_1,\dots,m_N$ lie on $\H^1_{xw}$. Expressing the position of each mass $m_i$ in terms of the oriented hyperbolic distance $\th_i\in\R, \ i=\overline{1,N}$, measured from the vertex, $(0,0, 0,1)$,  we can represent the position vectors and the distances between bodies as
 $$
 \q_i=(\sinh\theta_i, 0, 0, \cosh\theta_i),\ \ d_{ij}=|\theta_i-\theta_j|, 
 \ \ i,j=\overline{1,N},
 $$ 
 respectively. Then the force function and the moment of inertia can be written as
 $$
 U(\q)=\sum_{1\le i<j\le N}m_im_j\coth d_{ij}\ \ {\rm and}\ \
 I(\q)=\sum_{i=1}^Nm_i\sinh^2\theta_i.
 $$

 By the critical point characterization of central configurations introduced  in Section \ref{criticalpoint},  we only need to find the number of critical points of $\hat{U}$ on $\hat{S_c}$ for a constant $c>0$. In this case, we have 
 $$
 \hat{S_c}= S_c/ SO(2)\times SO(1,1)=S_c=\{\q\in(\H^1_{xw})^N\setminus\Delta\ \! |\ \! I(\q)=c\},\ \ \hat{U}=U,
  $$
 where $\Delta$ denotes the collision set. Equivalently,  we only need to find the number of critical points of $U -\lambda I$ in $(\H^1_{xw})^N\setminus\Delta$, where $\lambda$ is fixed. 
We can now state and prove the following result.
 \begin{theorem}
 For any given  point masses $m_1,\dots,m_N>0$ in  $\H^3$ and  each $c>0$, there are exactly $N!/2$ geodesic central configurations with $I(\q)=c$, one for each ordering of the masses on the geodesic.
 \end{theorem}
 \begin{proof}
 We will follow the idea used to prove the classical theorem of Moulton, \cite{Abraham}, \cite{Moeckel2}, and show first that the manifold $S_c$ contains $N!$ components, each homemorphic to an $(N-1)$-dimensional disk. We will then prove that the critical points of $\hat{U}$, or equivalently, of $U$,  are local minima on these disks, and finally show that there is just one minimum on each such disk. 
 
To prove that each ordering corresponds to an $(N-1)$-dimensional open disk, it suffices to consider one of the orderings, $\th_1<\cdots<\th_N$. Denote the corresponding component by $S_c'$ (see Figure \ref{fig:geo-central configuration-h}). Consider the homemorphism $\phi: (\H^1_{xw})^N \to \R^N$, $\phi(\th_1, \cdots, \th_N)=(x_1, \cdots, x_N)$, where $x_i=\sinh \th_i$.  Then $S_c'$ is homemorphic to 
 \[ \{(x_1,\dots,x_N)\in\R^N\ | \  x_1<\cdots <x_N, \ \ \sum_{i=1}^N m_ix_i^2=c\}, \]  
 which is an  $(N-1)$-dimensional open disk, \cite{Moeckel2}. Thus the set $S_c$ has exactly $N!$ components, each homemorphic to an $(N-1)$-dimensional open  disk. By an argument similar to the one in the proof of Theorem \ref{existence}, we can establish the existence of a critical point, or a central configuration, on each component.
\begin{figure}
   \centering
  \begin{tikzpicture}
  \draw[black, domain=-2:2] plot (\x, {sqrt(\x*\x+1) });

  \fill (.5,1.1) circle (2.5pt) node[above] {$m_N$};
   \fill (-1,1.4) circle (2.5pt) node[above] {$m_2$};
 \fill (-1.5,1.8) circle (2.5pt) node[above] {$m_1$};
 \fill (0,1) circle (2.5pt); 
 \fill (-.5,1.1) circle (2.5pt);
 
 \fill (-1.5,0) circle (1.4pt) node[above] {$\sinh \th_1$};
 \fill (-1,0) circle (1.4pt) node[below] {$\sinh \th_2$};
 \fill (.5,0) circle (1.4pt) node[above] {$\sinh \th_N$};
 
   \draw [->] (0, 0)--(0, 3) node (zaxis) [left] {$w$};
         \draw [->](-3, 0)--(3, 0) node (yaxis) [right] {$x$};
  \end{tikzpicture}
  \caption{ A configuration of $N$-masses on $\H^1_{xw}$} \label{fig:geo-central configuration-h}
    \end{figure}
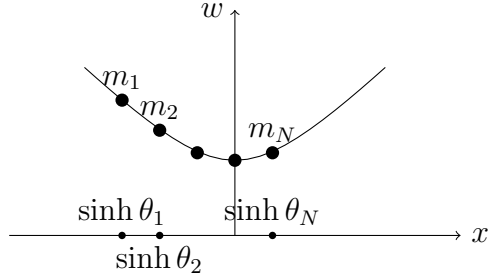 
 
 Denote such a critical point by $\q'$. We will show that  $\q'$ must be a local  minimum  of  $U$ in $S_c$.  For this, we  first prove that $\q'$  is a  local minimum  of 
 $U(\q)-\lambda I(\q)$ in $(\H_{xw}^1)^N\setminus \Delta$, where $\lambda=\lambda(\q')<0$ is a constant determined by equation \eqref{lambda-}.  To reach this
   goal, we   compute the Hessian of $U(\q)-\lambda I(\q)$ and show that it is positive definite. By straightforward computations, we obtain 
   \begin{align*}
  {\rm Hess}_\q=&\ \!D^2 U(\q) -\lambda D^2I(\q)\\ 
  =& 2\begin{bmatrix}
    \sum\limits_{\substack{j=1, j\ne 1}}^N \frac{m_1m_j\cosh d_{1j}}{\sinh^3d_{1j}} & -\frac{m_1m_2\cosh d_{12}}{\sinh^3d_{12}}& \cdots& -\frac{m_1m_N\cosh d_{1N}}{\sinh^3d_{1N}}\\
    -\frac{m_2m_1\cosh d_{12}}{\sinh^3d_{12}}&  \sum\limits_{\substack{j=1,j\ne 2}}^N \frac{m_2m_j\cosh d_{2j}}{\sinh^3d_{2j}}&\cdots& -\frac{m_2m_N\cosh d_{2N}}{\sinh^3d_{2N}}\\
    \cdots&  \cdots& \cdots& \cdots\displaybreak[0]\\
    -\frac{m_1m_N\cosh d_{1N}}{\sinh^3d_{1N}}&\cdots&\cdots&\sum\limits_{\substack{j=1, j\ne N}}^N \frac{m_Nm_j\cosh d_{Nj}}{\sinh^3d_{Nj}}
    \end{bmatrix}\\
  &-2\lambda \begin{bmatrix}
    m_1 \cosh 2\theta_1&0&\cdots&0\\
    0&m_2 \cosh 2\theta_2&\cdots&0\\
    \cdots&  \cdots& \cdots& \cdots\\
    0&\cdots &\cdots &m_N \cosh 2\theta_N\
   \end{bmatrix}.
   \end{align*}
  Notice first that $-\lambda D^2I(\q)$,  the second term in ${\rm Hess}_\q$, is positive definite. Indeed,  the matrix $D^2I(\q)$ is obviously positive definite, and the coefficient $-\lambda$ is positive from equation \eqref{lambda-}. 
  
For the  first term, $D^2U$,  let us  take any nonzero vector 
  $\v=(v_1,\cdots, v_N)^T\in T_{\q'}\left((\H_{xw}^1)^N\setminus \Delta\right)$.  Regarding $D^2U$ as a bilinear form,  we get 
   \begin{align*}
 \v^T (D^2U)\v=\sum_{i=1}^N\sum_{j=1}^N(D^2 U)_{ij} v_iv_j=\ & 2\sum_{i=1}^{N}\sum\limits_{\substack{j=1\\j\ne i}}^N \frac{m_im_j\cosh d_{ij}}{\sinh^3d_{ij}}v_i^2\\- 2\sum_{i=1}^N\sum \limits_{\substack{j=1\\j\ne i}}^N \frac{m_im_j\cosh d_{ij}}{\sinh^3d_{ij}}v_iv_j
  =\ &\sum_{i=1}^N\sum\limits_{\substack{j=1\\j\ne i}}^N \frac{m_im_j\cosh d_{ij}}{\sinh^3d_{ij}}(v_i-v_j)^2 \geq 0.
   \end{align*}
  We can conclude that
  $\mbox{Hess}_\q(\v,\v)>0$ for all $\v \in T_{\q'}(\H_{xw}^1)^N\setminus \Delta$, so  $\q'$ is a local minimum of  $U(\q)-\lambda I(\q)$ on  $(\H_{xw}^1)^N\setminus \Delta$. Then $\q'$ is also a local minimum of the new function  $U(\q)-\lambda I(\q)+\lambda c$, restricted to the submanifold $S_c$. Note that, on $S_c$, this new function becomes $U$. Consequently  $\q'$ is a local minimum of $U$ on $S_c$.
 
 To show that such a minimum of $U$ is unique on each $(N-1)$-dimensional open disk, we can apply  a mountain pass theorem, \cite{Evans}. Assume that  there are two such minima. Connect these two points with a continuous family of curves. As the two ends are local minima, there must be a local maximum on each curve. Then the minimum of all these maxima must be a saddle point of $U$, in contradiction with the positive definiteness of the Hessian. 
 
 Note that a $180^\circ$ rotation in the $xy$-plane changes the ordering, which means that we counted each case twice, so there are exactly $N!/2$ classes of  geodesic central configurations, a remark that completes the proof.  
\end{proof}   
   
  \subsection{Geodesic central configurations in $\S^3$}
 
 Unlike in the hyperbolic case, Moulton's theorem has no straightforward generalization to $\S^3$. We give an example of geodesic central configurations for two masses to show that the number of central configurations on $S_c$ depends on the value of $c$. This example also provides some degenerate central configurations, as defined in Section \ref{criticalpoint}, and means that the corresponding critical points of $\hat{U}$ on $\hat{S_c}$ are degenerate.   
 
According to Theorem \ref{1d_CC}, any geodesic central configuration in $\S^3$ is equivalent to some geodesic central configuration on $\S^1_{xz}$. The example we will exhibit is that of  central configurations for two masses on $\S^1_{xz}$. Special central configurations cannot exist under these circumstances since any nonsingular configuration would force the two masses to lie inside a semicircle, which turns out to be impossible because such a configuration cannot generate relative equilibria, as proved in \cite{Diacu03}. Expressing the positions of $m_1$ and $m_2$ in terms of the oriented spherical distance, $\th_i\in [0,2\pi], \ i=1,2$, measured  from $(0,0,1,0)$ (see Figure \ref{fig:2central configuration}),  we can write the position vectors as
   $$
    \q_1=(-\sin\theta_1,0,\cos\theta_1,0),\ \
    \q_2=(-\sin\theta_2,0,\cos\theta_2,0), \ \ 0\le \theta_1<\theta_2\le 2\pi.
    $$
    Then the force function and the moment of inertia have the form
    $$
   U(\q)=m_1m_2\cot d_{12}\ \ {\rm and}\ \ 
    I(\q)=m_1\sin^2\theta_1+m_2\sin^2\theta_2,
    $$
respectively, where $d_{12}=\min\{\theta_2-\theta_1, 2\pi-\theta_2+\theta_1\}$ is the distance between the bodies. We can also assume, without loss of generality, that $\theta_1\in[0,\pi/2]$. This is all the preparation we need to state and prove the following result.
 \begin{figure}
  \centering
  \begin{tikzpicture}
  \draw (0,  0) circle (2);
  \draw[->] (0, 0) -- (45:2);
  \draw[->] (0, 0) -- (165:2);

  \draw[->] (90:.5)arc(90:405:.5);
  \draw[->] (90:.6)arc(90:165:.6);
  
  \node[right] at(210:.9) {$\theta_2$};
  \node[above] at(120:.8) {$\theta_1$};
  \fill (45:2) circle (2.5pt) node[right] {$m_2$};
  \fill (165:2) circle (2.5pt) node[left] {$m_1$};

   \draw [->] (0, -2.2)--(0, 2.2) node (zaxis) [left] {$z$};
         \draw [->](-2.2, 0)--(2.2, 0) node (yaxis) [right] {$x$};
  \end{tikzpicture}
  \caption{ A configuration of two masses on $\S^1_{xz}$} \label{fig:2central configuration}
   \end{figure}
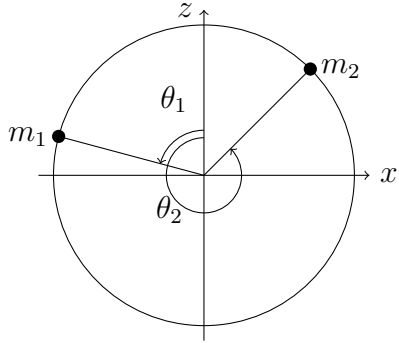
 \begin{theorem}\label{s1}
 Consider two masses $m_1$ and $m_2$ on  $\S^1_{xz}$ with positions $\q_1$ and $\q_2$ as above. Then these bodies can form a central configuration if and only if 
 \begin{equation} \label{2cc+}
 m_1 \sin 2\theta_1+ m_2 \sin 2\theta_2=0  \ {\rm with}\   \sin 2\theta_1\neq0.
 \end{equation} 
 The number of geodesic central configurations depends on the size $I(\q)=c$ of each configuration and is given in the table below, where $M:=m_1+m_2$. The table on the left is for $m_1<m_2$, whereas the table on the right is for the $m_1=m_2=:m$. 
 \begin{center}
 \begin{tabular}
 {|p{3cm}|p{1.5cm}|}
  \hline
    size: $I(\q)=c$ & number \\ \hline
    $c\in(0, m_1)$ & 2 \\  \hline
     $c\in[m_1,m_2]$ & 0 \\ \hline
     $c\in (m_2,M)$ & 2 \\ 
   \hline
 \end{tabular}
 \begin{tabular}{|p{3cm}|p{1.5cm}|}
   \hline
   size: $I(\q)=c$  & number \\ \hline
    $c\in(0, m)$ & 2\\  \hline
   $c=m$ & $\infty$ \\ \hline
   $c\in(m,M)$ & 2 \\ 
   \hline
 \end{tabular}
 \end{center}
 When the masses are equal and $c=m$, all central configurations are degenerate critical points of $U$ on $S_{m}$
  and the set they form has the power of the continuum.
 \end{theorem}
 \begin{proof}
 In this case, the central configuration equation  \eqref{CCE}, $\nabla_{\q_i}U=\lambda\nabla_{\q_i}I, i=1,2$,  reduces to 
 $$
 \frac{\partial U}{\partial \theta_1} = \lambda \frac{\partial I}{\partial \theta_1} \ \ {\rm and}\ \
  \frac{\partial U}{\partial \theta_2} = \lambda \frac{\partial I}{\partial \theta_2},   
 $$
 which implies that 
 $$
 \frac{\pm m_1m_2}{\sin^2(\theta_2-\theta_1)}=\lambda m_1 \sin 2\theta_1\ \ {\rm and}\ \   \frac{\mp m_1m_2}{\sin^2(\theta_2-\theta_1)}=\lambda m_2 \sin 2\theta_2,  
 $$
 where the signs depend on whether $d_{12}$ equals $\theta_2-\theta_1$ or  $2\pi-\theta_2+\theta_1$. From these equations we obtain the condition
 $$
 m_1 \sin 2\theta_1+ m_2 \sin 2\theta_2=0 \ {\rm with}\  \sin 2\theta_1\neq 0. 
 $$
 This relationship implies that $\theta_1\in \left(0, \frac{\pi}{2}\right)$ and $\theta_2 \in \left(\frac{1}{2}\pi,\pi\right) $ or $\theta_2\in\left(\frac{3}{2}\pi,2\pi\right)$. 
 
 To find the number of central configurations on $S_c$, we solve the system
 \begin{equation*}
 \begin{cases}
 m_1\sin^2\theta_1+m_2\sin^2\theta_2=c\cr
 m_1 \sin 2\theta_1+ m_2 \sin 2\theta_2=0,
 \end{cases}
 \end{equation*}
 and obtain
 \[  \sin^2\theta_2 = \frac{c(m_1-c)}{m_2(M-2c)} \ \  \mbox{and}\ \ \sin^2\theta_1 = \frac{c(m_2-c)}{m_1(M-2c)}. \] 
 Notice that $\sin 2\theta_i\neq0$, so let
 \[ 0<\frac{c(m_1-c)}{m_2(M-2c)} <1, \ \ 0<\frac{c(m_2-c)}{m_1(M-2c)}<1. \]
 We are then led to 
 \[ 0<c<m_1, \ \  m_2<c<M,\]
 a fact that can also be seen in the graphs of Figure \ref{fig:M1}, where a typical function of the form $\frac{c(m_1-c)}{m_2(M-2c)}$ is represented for $m_1<m_2$, on the left, and $m_1=m_2$, on the right.
 \begin{figure}
 \centering
 \begin{tikzpicture}
 \vspace*{0cm}\hspace*{-3.5cm}
 \draw [->] (0, -1)--(0, 1.5);
  \draw [-](-0.5, 1)--(3.1, 1) node at (-0.3,1) [left] {$1$};
            \draw [->](-1, 0)--(3.3, 0) node  [right] {$c$};
  \draw [dotted] (1.5, -0.5)--(1.5,1.5);

  \fill (1.4,0) circle (1.4pt) node[below left] {$m_1$};
  \fill (1.6,0) circle (1.4pt) node[below right] {$m_2$};
  \fill (3,0) circle (1.4pt) node[below] {$M$};
  
 \draw [  domain=-0.2:1.45] plot (\x, {\x*(1.4-\x)/(1.6*(3-2*\x))});
 \draw [   domain=1.55:3.2] plot (\x, {\x*(1.4-\x)/(1.6*(3-2*\x))});
 
 \vspace*{0cm}\hspace*{8cm}
 
 \draw [->] (0, -1)--(0, 1.5);
  \draw [-](-0.5, 1)--(3.1, 1) node at (-0.3,1) [left] {$1$};
            \draw [->](-1, 0)--(3.3, 0) node  [right] {$c$};
  \draw [dotted] (1.5, -0.5)--(1.5,1.5); 
  
  \draw  (1.5,0.5) circle [radius=0.07];;
  \fill (1.5,0) circle (1.4pt)node[below] {$m$};
  \fill (3,0) circle (1.4pt) node[below] {$M$};
  
 \draw [  domain=-.2:1.45] plot (\x, {\x/3});
 \draw [  domain=1.61:3.1] plot (\x, {\x/3});
 \end{tikzpicture}
 \caption{The graphs of $\sin^2\theta_2 = \frac{c(m_1-c)}{m_2(M-2c)}$ for $m_1<m_2$ (left) and $m_1=m_2=:m$ (right) in coordinates $(c,\sin^2\theta_2$). } \label{fig:M1}
 \end{figure}
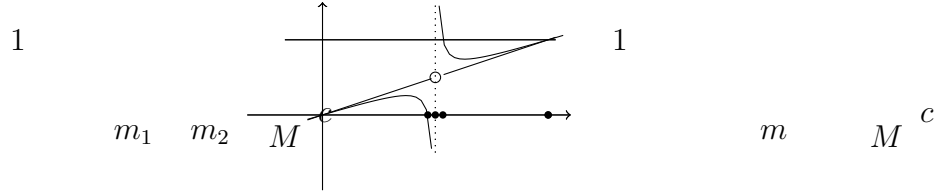
 
 Thus having $c$ in this range, we can obtain the values for $\sin^2 \theta_i<1, i=1,2$. 
 Using the fact that $\theta_1\in \left(0, \frac{\pi}{2}\right)$ and $\theta_2 \in \left(\frac{1}{2}\pi,\pi\right) $ or $\theta_2 \in\left(\frac{3}{2}\pi,2\pi\right)$, we see that there are exactly two central configurations for each $c$:
 $$
 (\theta_1, \theta_2) \in \left(0, \frac{\pi}{2}\right) \times \left(\frac{1}{2}\pi,\pi\right) \ {\rm and}\ \ \! (\theta_1, \theta_2 + \pi) \in \left(0, \frac{\pi}{2}\right) \times \left(\frac{3}{2}\pi,2\pi\right).
 $$
  
 If $m_1=m_2=m$ and $I=m$, we have 
 $$
 \hat{S}_m=S_m=\{ (\th_1,\th_2)\in (0,\pi/2)\times [0,2\pi] \ | \  \th_1<\th_2,\  \sin^2\theta_1+\sin^2\theta_2=1\},
 $$ 
 which implies that 
$$\hat{S}_m=\{ \th_1 \in (0,\pi/2),\
 \theta_2=\theta_1 +\pi/2 \ \ \mbox{or} \ \ \theta_2=\theta_1 +3\pi/2\}. 
 $$
 Thus $d_{12}=\pi/2$ and $\hat{U}=U=m_1m_2\cot d_{12}=0$ on $S_m$, which means that all elements of this set are degenerate critical points of $U$ on $S_m$, so they are degenerate central configurations. This remark justifies the values in the above tables and completes the proof. 
 \end{proof}
 The related problem of finding relative equilibria on $\S_{xz}^1$ has also been considered by A.A.\ Kilin, who obtained the same criterion given in the first part of Theorem \ref{s1}, \cite{Kilin}.
 \begin{remark}
 The complicated count of geodesic  central configurations in $\S^3$ is a consequence of two facts: the boundary of some components in $S_c$ may contain points in $\Delta^+$ and $\Delta^-$, which can destroy the existence of critical points on those components; and the geodesic central configurations are not necessarily minima of $U$ on $S_c$.  
 \end{remark}
 
 \section{Conclusions}
 
 So far, the only classes of solutions found for the $N$-body problem in spaces of constant curvature were relative equilibria and rotopulsators, the latter allowing dilations and contractions of the configuration, which, of course, failed to maintain similarity, \cite{Diacu06}, although, very recently, some numerical results point out to the existence of choreographies, including the figure eight solution on the sphere $\S^2$, \cite{Mont}. However, these studies are only at the beginning, and the current paper shows that the approach we took here offers another way to answer some of the natural problems that occur in the qualitative study of the equations of motion
 and the dynamics of the solutions.
 
But most questions related to central configurations are far from easy, as it also happens in the Euclidean case. Even finding all the central configurations in the curved 3-body problem, which has been settled in the classical case long time ago, is not trivial in curved space and requires a separate study. As we have already seen, new central configurations, such as the isosceles triangles, or the scalene triangles on the equator of the sphere, none of which have correspondents in the Euclidean case, show up. So far, all these central configurations on $\S^2_{xyz}$ lie in planes parallel with the $xy$-plane, except for the geodesic ones.  But at this point we have some indication that most triangular central configurations do not lie in planes parallel with the $xy$-plane, and hope to be able to prove this statement in a future paper. So even for only three bodies, the set of central configurations of the curved problem is significantly richer than in the Euclidean case, especially in the case of the sphere.
 
These investigations hint at the rich dynamics of the curved $N$-body problem and show that the  questions occurring from its study allow us to view the classical case from a new perspective. Having now extended the concept of central configuration to the curved problem, we have a new tool and a new direction of research, which will hopefully shed more light on the equations of motion that govern this mathematical model. 
 
 \bigskip
\noindent{\bf Acknowledgments.}  Cristina Stoica and Florin Diacu enjoyed partial support from Discovery Grants awarded by NSERC of Canada. Shuqiang Zhu was funded by a University of Victoria Scholarship and a David and Geoffrey Fox Graduate Fellowship.



\begin{thebibliography}{99}



\bibitem{Abraham} R.~Abraham and J.\ Marsden, {\it Foundations of Mechanics}, 2nd ed., Addison-Wesley, 1987.

\bibitem{Albouy} A.\ Albouy and V.\ Kaloshin, Finiteness of central configurations of five bodies in the plane, {\it Ann.\ of Math.} {\bf 176} (2012), 535--588.

\bibitem{Alfaro} F.\ Alfaro and E.\ P\'erez-Chavela, amilies of continua of central configurations in charged problems,
{\it Dyn.\ Cont.\ Discrete Impuls.\ Syst.\ Ser.\ A Math.\ Anal.} {\bf 9} (2002), 463--465.

\bibitem{Ber}  P.\ de Bernardis et. al., A flat Universe from high-resolution maps of the cosmic microwave background radiation, {\it Nature} {\bf 404}, 6781 (2000), 955--959.

\bibitem{Bertrand} J.~Bertrand, Th\'eor\`eme relatif au mouvement d'un point attir\'e vers
un center fixe, {C.\ R.\ Acad.\ Sci.}~{\bf 77} (1873), 849--853.

\bibitem{Bolyai} W.~Bolyai and J.~Bolyai, {\it Geometrische Untersuchungen}, Teubner, Leipzig-Berlin, 1913.

\bibitem{Diacu01} F.\ Diacu, On the singularities of the curved $N$-body problem, {\it Trans.\ Amer.\ Math.\ Soc.} {\bf 363}, 4 (2011), 2249--2264.

\bibitem{Diacu02} F.\ Diacu, Polygonal homographic
orbits of the curved 3-body problem, {\it Trans.\ Amer.\ Math.\ Soc.} {\bf 364} (2012), 2783--2802.

\bibitem{Diacu03} F.\ Diacu, {\it Relative equilibria of the curved $N$-body problem}, Atlantis Studies in Dynamical Systems, vol. 1, Atlantis Press, Amsterdam, 2012.

\bibitem{Diacu05} F.\ Diacu, Relative equilibria of the 3-dimensional curved $n$-body problem, {\it Memoirs Amer. Math. Soc.} {\bf 228}, 1071 (2013).

\bibitem{Diacu07} F.\ Diacu, The curved $N$-body problem: risks and rewards, {\it Math.\ Intelligencer} {\bf 35}, 3 (2013), 24--33.

\bibitem{Diacu77} F.\ Diacu, The classical $N$-body problem in the context of curved space, arXiv:1405.0453.

\bibitem{Diacu78} F.\ Diacu, Bifurcations of the Lagrangian orbits from the classical to the curved 3-body problem, arXiv:1508.06043.

\bibitem{Diacu06} F.\ Diacu and S.\ Kordlou, Rotopulsators of the curved $N$-body problem, {\it J.\ Differential Equations}  {\bf 255} (2013) 2709--2750.

\bibitem{Diacu08} F.\ Diacu, R.\ Mart\'inez, E.\ P\'erez-Chavela, and C.\ Sim\'o,
On the stability of tetrahedral relative equilibria in the positively curved 4-body problem, {\it Physica D} {\bf 256-7} (2013), 21--35.

\bibitem{Diacu09} F.~Diacu and E.~P\'erez-Chavela, Homographic solutions of the
curved $3$-body problem, {\it J.\ Differential Equations} {\bf 250} (2011), 340--366.

\bibitem{Diacu10} F.~Diacu, E.~P\'erez-Chavela, and M.~Santoprete, Saari's conjecture for the collinear $N$-body problem, {\it Trans.~Amer.~Math.~Soc.} {\bf 357}, 10 (2005), 4215--4223. 

\bibitem{Diacu11} F.~Diacu, E.~P\'erez-Chavela, and M.~Santoprete, The
$N$-body problem in spaces of constant curvature. Part I: Relative equilibria,
{\it J.\ Nonlinear Sci.} {\bf 22}, 2 (2012), 247--266, DOI: 10.1007/s00332-011-9116-z.

\bibitem{Diacu12} F.~Diacu, E.~P\'erez-Chavela, and M.~Santoprete, The $N$-body problem in spaces of constant curvature. Part II: Singularities,
{\it J.\ Nonlinear Sci.} {\bf 22}, 2 (2012), 267--275, DOI: 10.1007/s00332-011-9117-y.

\bibitem{Diacu13} F.\ Diacu, E.\ P\'erez-Chavela, and J.\ Guadalupe Reyes Victoria, An intrinsic approach in the curved $N$-body problem. The negative curvature case, {\it J.\ Differential Equations} {\bf 252} (2012), 4529--4562.

\bibitem{Diacu-Popa} F.\ Diacu and S.\ Popa, All Lagrangian relative equilibria have equal masses, {\it J.\ Math.\ Phys.} {\bf 55}, 112701 (2014).

\bibitem{Diacu15-1} F.\ Diacu, J.M.\ S\'{a}nchez-Cerritos, and S.Q.\ Zhu,   On the stability of fixed-points in the 3-body problem on $S^2$, submitted. 

\bibitem{Diacu14} F.\ Diacu and B.\ Thorn, Rectangular orbits of the curved 4-body problem, {\it Proc.\ Amer.\ Math.\ Soc.} {\bf 143} (2015), 1583--1593. 

\bibitem{dictionary}  {\it Dictionary.com Unabridged},  http://dictionary.reference.com/browse/moment+of+inertia based on the Random House Dictionary, 2015.

\bibitem{Dziobek} O.\ Dziobek, \"Uber einen merkw\"urdigen Fall des Vielk\"orperproblems, {\it Astron.\ Nachr.} {\bf 152}
(1900), 33--46. 

\bibitem{Eu} L.~Euler,  {\it Theoria motus corporum solidorum seu rigidorum: Ex primis nostrae cognitionis principiis stabilita et ad omnes motus, qui in huiusmodi corpora cadere possunt, accommodata} [The theory of motion of solid or rigid bodies: established from first principles of our knowledge and appropriate for all motions which can occur in such bodies], A.F.\ R\"ose, Rostock and Greifswald, 1765.

\bibitem{Euler} L.~Euler, Considerationes de motu corporum coelestium, {\it Novi commentarii academiae scientiarum Petropolitanae} {\bf 10} (1764), 1766, pp. 544Ð558 (read at Berlin in april 1762). Also in Opera Omnia, S. 2, vol. 25, pp. 246--257 with corrections and comments by M. Sch\"urer.

\bibitem{Evans} L.C.\ Evans,  \textit{Partial Differential Equations}, American Mathematical Society,  Providence, RI, 1998.

\bibitem{Garcia} L.C.\ Garc\'ia-Naranjo, J.C.\ Marrero, E.\ P\'erez-Chavela, M.\ Rodr\'iguez-Olmos, Classification and stability of relative equilibria for the two-body problem in the hyperbolic space of dimension 2, arXiv:1505.01452.

\bibitem{Hirsch} M.W.\ Hirsch, {\it Differential Topology}, Graduate Texts in Mathematics, vol.\ 33, Springer Verlag,1976.

\bibitem{Jacobi} C.G.J. Jacobi, Vorlesungen \"uber Dynamik, in C.G.J. JacobiÕs Gesammelte Werke, vol. VIII, Druck und Verlag Von G. Reimer, Berlin, 1884.

\bibitem{Kilin} A.A.\ Kilin, Libration points in spaces ${\bf{S}}^2$ and ${\bf L}^2$, {\it Regul.\ Chaotic Dyn.} {\bf 4}, 1 (1999), 91--103.

\bibitem{Killing} W.~Killing, Die Rechnung in den nichteuklidischen Raumformen, {\it J.\ Reine Angew.\ Math.} {\bf 89} (1880), 265--287.

\bibitem{Kozlov}  V.~V.~Kozlov and A.~O.~Harin, Kepler's problem in constant curvature spaces, {\it Celestial Mech.~Dynam.~Astronom} {\bf 54} (1992), 393-399.

\bibitem{Kragh} H.~Kragh, Is space Flat? Nineteenth century astronomy and non-Euclidean geometry, {\it J.\ Astr.\ Hist.\ Heritage} {\bf 15}, 3 (2012), 149--158.

\bibitem{Lagrange} J.L.\ Lagrange, Essai sur le probl\`eme des trois corps, 1772, \OE{}uvres tome 6.

\bibitem{Laplace} P.S.\ Laplace, {\it Oeuvres}, vol.\ 4, pp.\ 307--513, vol.\ 11, pp.\ 553--558.

\bibitem{Lie} S.\ Lie, Theorie der Transformationsgruppen. Zweiter Abschnitt, Teubner Verlag, 1890.

\bibitem{Liebmann1} H.~Liebmann, Die Kegelschnitte und die Planetenbewegung im nichteuklidischen Raum, {\it Berichte K\"onigl.~S\"achsischen Gesell. Wiss., Math.~Phys.~Klasse} {\bf 54} (1902), 393--423.

\bibitem{Liebmann2}H.~Liebmann, \"Uber die Zentralbewegung in der nichteuklidische
Geometrie, {\it Berichte K\"onigl.~S\"achsischen Gesell. Wiss., Math.~Phys.~Klasse} {\bf 55} (1903), 146-153.

\bibitem{Lipschitz} R.~Lipschitz, Extension of the planet-problem to a space of $n$ dimensions and constant integral curvature, {\it Quart.~J.~Pure Appl.~Math.} {\bf 12} (1873), 349--370.

\bibitem{Lobachevsky} N.~I.~Lobachevsky, The new foundations of geometry with full theory of parallels [in Russian], 1835-1838, in Collected Works, vol.\ 2, GITTL, Moscow, 1949.

\bibitem{Marsden} J.\ Marsden, {\it Lectures on Mechanics}, Cambridge University Press, 2009.

\bibitem{Ratiu} J.\ Marsden and T.\ Ratiu, {\it Introduction to Mechanics and Symmetry: A Basic Exposition of Classical Mechanical Systems}, Springer Verlag, 1999. 

\bibitem{Martinez1} R.\ Mart\'inez and C.\ Sim\'o, On the stability of the Lagrangian homographic solutions in a curved three-body problem on $\mathbb S^2$, {\it Discrete Contin.\ Dyn.\ Syst.\ Ser.\ A} {\bf 33} (2013) 1157--1175.

\bibitem{Martinez2} R.\ Mart\'inez and C.\ Sim\'o, Relative equilibria of the restricted 3-body problem in curved spaces,
(in preparation).

\bibitem{Moeckel} R.\ Moeckel, Finiteness of relative equilibria of the four-body problem, {\it Invent.\ Math.} {\bf 163} (2006), 289--312. 

\bibitem{Moeckel2} R.\ Moeckel, {\it Celestial Mechanics---especially central configurations}, unpublished lecture notes: http://www.math.umn.edu/\~{}rmoeckel/notes/CMNotes.pdf

\bibitem{Mont} H.\ Montanelli and N.I.\ Gushterov, Computing planar and spherical choreographies, {\it SIAM J.\ Appl.\ Dyn.\ Syst.}, to appear.

\bibitem{Moulton} F.R.\ Moulton, The straight line solutions of $n$ bodies, {\it Ann. of Math.} {\bf 12}, 1--17. 

\bibitem{Perez} E.~P\'erez-Chavela and J.G.~Reyes Victoria, An intrinsic approach in the curved $N$-body problem. The positive curvature case, {\it Trans.\ Amer.\ Math.\ Soc.} {\bf 364}, 7 (2012), 3805--3827.

\bibitem{Riemann} B.\ Riemann, \"Uber die Hypothesen welche der Geometrie zu Grunde liegen, {\it Abhandl.\
K\"onigl.\ Ges.\ Wiss.\ G\"ott.}, {\bf 13}, 1854.

\bibitem{Roberts} G.\ Roberts, Continua of central configurations with a negative mass in the $n$-body problem,
{\it Celestial Mech.\ Dyn.\ Astron.} {\bf 115}, 4 (2013), 427--438. 

\bibitem{Saari} D.\ Saari, On the role and properties of central configurations, {\it Celestial Mech.} {\bf 21} (1980), 9--20.

\bibitem{Saari2} D.\ Saari, {\it Collisions, Rings, and Other Newtonian $N$-Body Problems}, CBMS Regional Conference Series in Mathematics, American Math.\ Society, 2005.

\bibitem{Schering1} E.~Schering, Die Schwerkraft im Gaussischen R\"aume, {\it Nachr.\ K\"onigl.\ Ges.\ Wiss.\ G\"ott.} {\bf 15}, (1870), 311--321.

\bibitem{Schering2} E.~Schering, Die Schwerkraft in mehrfach ausgedehnten Gaussischen und Riemmanschen R\"aumen. {\it Nachr.\ K\"onigl.\ Ges.\ Wiss.\ G\"ott.} {\bf 6}, (1873), 149--159 

\bibitem{Shchepetilov} A.V.~Shchepetilov,  Nonintegrability of the two-body problem in constant curvature spaces, {\it J.\ Phys. A: Math.\ Gen.} V.\ 39 (2006), 5787-5806; corrected version at math.DS/0601382.

\bibitem{Singer} S.F.\ Singer, {\it Symmetry in Mechanics: A Gentle, Modern Introduction}, Springer Verlag, 2004.

\bibitem{Smale} S.\ Smale, Mathematical Problems for the Next Century, {\it Math.\ Intelligencer} {\bf 20}, 2 (1998) 7--15. 

\bibitem{Smale70-1} S. Smale, Topology and mechanics. I, {\it Invent.\  Math.} \textbf{10}, 4 (1970), 305--331.

\bibitem{Tibboel1} P.\ Tibboel, Polygonal homographic orbits in spaces of constant curvature, {\it Proc.\ Amer.\ Math.\ Soc.} {\bf 141} (2013), 1465--1471.

\bibitem{Tibboel2} P.\ Tibboel, Existence of a class of rotopulsators, {\it J.\ Math.\ Anal.\ Appl.} {\bf 404} (2013), 185--191.

\bibitem{Tibboel3} P.\ Tibboel, Existence of a lower bound for the distance between point masses of relative equilibria in spaces of constant curvature, {\it J.\ Math.\ Anal.\ Appl.} {\bf 416} (2014), 205--211.

\bibitem{Wintner} A.\ Wintner, {\it The Analytical Foundations of Celestial Mechanics}, Princeton University Press, 1947.

\bibitem{Zhu} S.Q.\ Zhu, Eulerian relative equilibria of the curved 3-body problems in $\mathbb S^2$, {\it Proc.\ Amer.\ Math.\ Soc.} {\bf  142} (2014), 2837--2848.


\end{thebibliography}
\end{document}